\documentclass[12pt,a4paper]{article}

\usepackage{graphicx}
\usepackage{subfig}
\usepackage{amsmath,amsthm,amssymb}
\usepackage{esint}
\usepackage{mathrsfs}

\usepackage[T1]{fontenc}

\usepackage{hyperref}

\usepackage[left=2.2cm,right=2.2cm,top=3cm,bottom=3.5cm]{geometry}

\usepackage{color}
\usepackage{array}
\usepackage{multirow}

\theoremstyle{plain}
\newtheorem{thm}{Theorem}[section]
\newtheorem{prop}[thm]{Proposition}
\newtheorem{lemma}[thm]{Lemma}
\newtheorem{cor}[thm]{Corollary}
\newtheorem{taggedthmx}{Theorem}
\newenvironment{taggedthm}[1]
 {\renewcommand\thetaggedthmx{#1}\taggedthmx}
 {\endtaggedthmx}

\theoremstyle{definition}
\newtheorem{defn}[thm]{Definition}

\theoremstyle{remark}
\newtheorem{rem}[thm]{Remark}

\numberwithin{equation}{section}
\numberwithin{figure}{section}
\numberwithin{table}{section}

\usepackage{csquotes}
\usepackage[backend=bibtex8, style=numeric-comp, maxnames=6, sorting=nyt]{biblatex}
\usepackage{enumitem}
\renewcommand{\labelenumi}{(\alph{enumi})}
\renewcommand{\theenumi}{\labelenumi}

\newcommand{\R}{\mathbb{R}} 
\newcommand{\N}{\mathbb{N}}
\newcommand{\Z}{\mathbb{Z}}

\newcommand{\Grad}{\nabla}  
\newcommand{\Div}{{\rm div}\,} 
\newcommand{\Curl}{{\rm curl}\,} 
\newcommand{\Lap}{\Delta}  
\newcommand{\Gradh}{\nabla_h}  
\newcommand{\Divh}{{\rm div}_h\,} 
\newcommand{\parthree}[3]{\partial_{#1#2#3}^3}

\newcommand{\dx}{\,{\rm d}x}
\newcommand{\dy}{\,{\rm d}y}

\newcommand{\dt}{\,{\rm d}t}  
\newcommand{\dr}{\,{\rm d}r} 
\newcommand{\ds}{\,{\rm d}s} 
\newcommand{\dTheta}{\,{\rm d}\Theta} 
\newcommand{\dPhi}{\,{\rm d}\Phi} 
\newcommand{\dS}{\,{\rm d}S}

\newcommand{\id}{\mathbb{I}} 

\newcommand{\symz}[1]{{\rm Sym}_0(#1)}
\newcommand{\sz}{\symz{2}} 

\newcommand{\tr}{{\rm tr}\,} 
\newcommand{\trans}{\top} 

\newcommand{\closure}[1]{\ov{#1}} 
\newcommand{\interior}[1]{{#1}^\circ} 

\newcommand{\imag}{{\rm i}} 
\newcommand{\T}{\mathbb{T}} 
\renewcommand{\rho}{\varrho}
\newcommand{\half}{\tfrac{1}{2}}  
\newcommand{\Cc}{C^\infty_{\rm c}} 
\newcommand{\Cb}{C_{\rm b}} 
\newcommand{\Cweak}{C_{\rm weak}}
\newcommand{\ov}[1]{\overline{#1}}  
\newcommand{\un}[1]{\underline{#1}}  
\newcommand{\ep}{\varepsilon} 
\newcommand{\co}{{\rm co}} 
\newcommand{\pert}{{\rm pert}}
\newcommand{\dist}{{\rm dist}}
\newcommand{\sign}{{\rm sign}}
\newcommand{\incomp}{{\rm incomp}}
\newcommand{\hydrost}{{\rm hydrost}}
\newcommand{\Artanh}{{\rm Artanh}}
\renewcommand{\mod}{\,{\rm mod}\,}
\let\subsetnotused\subset 
\renewcommand{\subset}{\subseteq}
\newcommand{\sU}{\mathcal{U}}
\newcommand{\sW}{\mathcal{W}}
\newcommand{\sK}{\mathcal{K}}
\newcommand{\sM}{\mathcal{M}} 
\renewcommand{\setminus}{\smallsetminus}
\newcommand{\opL}{\mathscr{L}}
\newcommand{\subsetcomp}{\subsetnotused\joinrel\subsetnotused}

\newcommand{\tf}{{\rm f}}
\renewcommand{\emptyset}{\varnothing}
\newcommand{\name}[1]{\textsc{#1}} 
\newcommand{\I}{I} 
\newcommand{\topology}{\mathcal{T}}
\newcommand{\QGunknown}{V} 
\newcommand{\QGmatrixset}{\sM^{3\times 2}} 

\allowdisplaybreaks

\begin{document}

\title{A Generalized Framework for $L^r$ Convex Integration and its Application to Geophysical Models} 

\author{Daniel W. Boutros\footnote{Department of Applied Mathematics and Theoretical Physics, University of Cambridge, Cambridge CB3 0WA, UK. Email: \textsf{dwb42@cam.ac.uk}} \and Simon Markfelder\footnote{Department of Mathematics and Statistics, University of Konstanz, 78457 Konstanz, Germany. Email: \textsf{simon.markfelder@uni-konstanz.de}} \and Edriss S. Titi\footnote{Department of Mathematics, Texas A\&M University, College Station, TX 77843-3368, USA; Department of Applied Mathematics and Theoretical Physics, University of Cambridge, Cambridge CB3 0WA, UK; also Department of Computer Science and Applied Mathematics, Weizmann Institute of Science, Rehovot 76100, Israel. Emails: \textsf{titi@math.tamu.edu} \; \textsf{Edriss.Titi@maths.cam.ac.uk} }} 

\date{June 10, 2026}

\maketitle
	
\begin{abstract} 
    In this paper a general framework for convex integration is developed, in order to construct weak solutions to the Cauchy problem, by building on ideas from [C. De Lellis and L. Sz{\'e}kelyhidi, Arch. Ration. Mech. Anal., 195 (2010)] and [S. Markfelder, Nonlinearity, 37 (2024)]. This framework may be applied to a large family of partial differential equations in order to construct weak solutions in $L^\infty ((0,T) \times \Omega)$ (for a bounded domain $\Omega)$ which are weakly continuous in time with respect to the weak topology of $L^r (\Omega)$ for some $r \in (1,\infty)$. This allows us to construct solutions which obey an energy inequality. 

    In the second part of the paper we apply the framework to several inviscid models appearing in the field of geophysical fluid mechanics in order to show existence of (infinitely many) weak solutions for all initial data, and to prove that there exist initial data for which there are infinitely many solutions which satisfy an energy inequality (such initial data are sometimes called ``wild data''). More precisely, we first consider the incompressible and the barotropic compressible Euler equations to recover the corresponding results from the literature. In addition, the framework allows us to prove a new result for the incompressible Euler equations, namely the global existence for the Cauchy problem in $L^\infty$. We then apply the framework to the shallow water and lake equations. Moreover, we use the framework in the context of the hydrostatic Euler equations (also known as the incompressible inviscid primitive equations), which leads to the first convex integration approach which is able to construct admissible solutions with the natural energy for this system. A crucial ingredient in the proof of this result is the computation of a large subset of the convex hull, as an explicit characterization of the convex hull does not seem to be available. Finally, we apply the framework to the compressible inviscid primitive equations and to the inviscid quasi-geostrophic equations to obtain the first results on existence of wild data for these two geophysical models. 
\end{abstract}

\bigskip


\bigskip


\bigskip

\tableofcontents

\section{Introduction} \label{sec:intro} 

\subsection{Convex integration} \label{subsec:intro-ci} 

The technique called \emph{convex integration} has drawn substantial interest within the field of mathematical fluid mechanics in the past years. The ideas behind convex integration go back to \name{Nash}'s work on isometric embeddings (see, e.g., \cite{Nash54}) and were further developed by \name{Gromov} (see, e.g., \cite{Gromov}) who used the term \emph{convex integration} for the first time. 

More recently, this technique was brought into the area of mathematical fluid mechanics in order to construct solutions to partial differential equations (PDEs) that are studied in this field. More precisely, it was the pioneering work by \name{De~Lellis} and \name{Sz{\'e}kelyhidi}~\cite{DelSze09} that showed for the first time how convex integration can be used in this context. 

In \cite{DelSze09} weak solutions to the incompressible Euler equations (see Sect.~\ref{subsubsec:intro-incomp-euler} below) with compact support in time and space were constructed. In particular, these solutions do not conserve energy. In the follow-up work \cite{DelSze10}, solutions which satisfy several kinds of energy inequality were generated. Interestingly, a certain freedom when applying convex integration leads to non-uniqueness of the aforementioned solutions. Shortly after these results, \name{Wiedemann}~\cite{Wiedemann11} proved existence of infinitely many weak solutions of the incompressible Euler equations for all $L^2$ initial data. These solutions however do not comply with any kind of admissibility criterion regarding the energy. 

Later, convex integration was also applied to other systems of PDEs, in particular, the compressible Euler equations, see \cite{DelSze10,Chiodaroli14,Feireisl14,ChiDelKre15,AkrWie21,Markfelder,DebSkiWie23} among others, but also the Euler-Korteweg-Poisson system \cite{DonFeiMar15}, the Euler-Fourier system \cite{ChiFeiKre15} and the ideal MHD equations \cite{FarLinSze21,FarLinSze24}, just to list a few. 

In the works mentioned so far, the constructed solutions are merely in $L^r$ for some $r\in [1,\infty]$. Convex integration has been further refined in \cite{DelSze13} in order to construct continuous solutions to the incompressible Euler equations. After stepwise improvements of the regularity of the constructed solutions in \cite{IsettPHD,DelSze14,BuckmasterPHD,Buckmaster15,BDIS15,DanSze17}, convex integration was used by \name{Isett}~\cite{Isett18} and \name{Buckmaster~et~al.}~\cite{BDSV18} to complete the proof of a conjecture by \name{Onsager}~\cite{Onsager49}, which provides a regularity threshold for energy conservation/dissipation by a (weak) solution of the incompressible Euler equations. The conservation part of Onsager's conjecture, i.e.~the sufficient regularity conditions to guarantee energy conservation, has been established in \cite{constantin1994}, see also \cite{eyink1994} as well as \cite{duchon2000,cheskidov2008,bardos2018,bardos2025}. The two-dimensional case of the dissipative part of the conjecture was treated later in \cite{giri2024b}.

Convex integration was also used to construct H{\"o}lder continuous solutions to other systems of PDEs, e.g.~the ideal MHD equations \cite{beekie2020}. In particular, in the geophysical context there have been several papers applying convex integration to the Boussinesq equations, see for example \cite{miao2024,xu2025,tao2017,tao2018a,tao2018b,luo2020b}. In addition to that, a further type of convex integration scheme (usually called \emph{intermittent convex integration}) has been introduced in \cite{BucVic19} and utilized in several papers, e.g.~\cite{ModSze18,modena2019,modena2020,BucColVic22,BMNV23,novack2023,CheLuo22,cheskidov2021,cheskidov2023,cheskidov2024,luo2020a,giri2023,giri2024}. In some of the latter papers convex integration is applied to the transport equation and to the incompressible Navier-Stokes equations. However in this paper we will only deal with $L^r$ convex integration as has been explained before.

The main goal of this paper is to provide a general framework for $L^r$ convex integration, which is applicable to a large family of PDEs. In particular, we want to produce solutions to a linear, first order system of PDEs of the form of \eqref{eq:lin-eq} below, which take values in a family of so-called constitutive sets $\sK_\theta$, which usually contain a non-linear constraint. Such a problem (i.e., a linear, first order system of PDEs together with a non-linear constraint in the form of a constitutive set) is also known as a \emph{differential inclusion}. More details can be found in Sect.~\ref{subsubsec:gen-prelim-pde+K}, below. Following \emph{Tartar's framework} (see, e.g., \cite{Tartar79}), one can rewrite non-linear first order systems of PDEs as differential inclusions, see also Rem.~\ref{rem:tartars-framework}, below. This way, our framework can be used to solve non-linear first order systems of PDEs. 

A framework which is related to the one established in this paper, has been introduced by the second author in \cite{Markfelder24}. However, there are two main differences between the method developed in \cite{Markfelder24} and the technique established in this paper, namely regarding the continuity in time of the constructed solutions as well as the energy balance. The solutions produced in \cite{Markfelder24} are merely bounded, i.e., they lie in $L^\infty$ in time and space. Moreover, if one applies the framework \cite{Markfelder24} to a ``cylindrically shaped'' space-time domain $[0,T)\times \Omega$, one obtains an energy jump to a higher value at the initial time, which makes the solutions inadmissible in the sense that they do not satisfy the energy inequality. In the $L^\infty$ framework there is no way to fix this (at least in cylindrical space-time domains $[0,T)\times \Omega$). Since in \cite{Markfelder24} convex integration is applied in a ``wedge'' in space-time rather than in a cylindrical domain $[0,T)\times \Omega$, the constructed solutions are still admissible as at the initial time the wedge shrinks down to a single point. 

In contrast to that, in the current work we construct solutions which are continuous in time with respect to the weak topology of $L^r$, where $r\in (1,\infty)$. As a consequence of that, we can make sense of the solution not just almost everywhere in time and space, but \emph{everywhere} in time and almost everywhere in space. This way, we are able to construct admissible solutions to the Cauchy problem on cylindrical domains $[0,T)\times \Omega$. In particular, the solutions constructed in the present paper do not just take values in the constitutive set $\sK_\theta$ almost everywhere (like in \cite{Markfelder24}) but for all times and almost every position in space. This has an important advantage, namely it helps to construct initial data and infinitely many solutions which do not exhibit the aforementioned energy jump at the initial time. We emphasize that the aforementioned differences between the framework established in \cite{Markfelder24} and the present paper imply the need for several new ideas as part of the construction. 

For the incompressible Euler equations, solutions with these properties (i.e., weak continuity in time and taking values in $\sK_\theta$ for all $t$ and a.e.~$x$) were already constructed by \name{De~Lellis}-\name{Sz{\'e}kelyhidi}~\cite{DelSze10}. So our framework may also be viewed as a generalization of \cite{DelSze10} to a larger class of PDEs. 

Let us also mention that \name{Feireisl}~\cite{Feireisl16} developed a framework for $L^r$ convex integration (also producing solutions that are weakly continuous in time). This framework is however much less universal compared to ours as it only applies to a very special class of PDEs which are inspired by the incompressible Euler equation. In particular, the hydrostatic Euler equations \eqref{eq:hydrost-euler-mass}-\eqref{eq:hydrost-euler-p}, below, which we study in this paper (among other PDEs), cannot be treated by means of the framework \cite{Feireisl16}.

In contrast, our result may be applied to a very large class of PDEs (including the hydrostatic Euler equations \eqref{eq:hydrost-euler-mass}-\eqref{eq:hydrost-euler-p}, below), more precisely to their associated differential inclusions. This class is characterized by four ``structural'' assumptions which are stated rigorously in the main theorem (Thm.~\ref{thm:conv-int}, below). In particular, we need 
\begin{itemize}
    \item some mild assumptions on the family of constitutive sets $\sK_\theta$, 
    \item the existence of a differential operator which has certain properties with regard to the given linear system of PDEs,
    \item the convex hull and the so-called $\Lambda$-convex hull of the constitutive sets $\sK_\theta$ must coincide, and 
    \item existence of a distance map which enables us to measure how far a so-called \emph{subsolution} (which can be understood as special kind of approximate solution) is away from being a solution. 
\end{itemize} 
Here, the third item, above, (i.e., the fact that the convex hull and the $\Lambda$-convex hull coincide) tends to be the most stringent condition.

\subsection{Examples of PDEs to which the general framework is applicable} \label{subsec:intro-examples}

Another aim of this paper is to exhibit some examples of PDEs which fit into the general framework that is established in this paper. For the sake of completeness, we demonstrate that the framework is able to recover some existing results in the literature, but we also show that the framework is sufficiently general and flexible in order to improve these existing results (for example in the case of the incompressible Euler equations). One of the primary motivations is to apply the framework to a wide class of inviscid geophysical models, see below for more details of the models under consideration. For several of these models the method of convex integration as it exists in the literature does not apply in a straightforward manner. 

In this paper we are interested in two types of statements:

\begin{taggedthm}{A} \label{thm:general-exforalldata} 
    For all initial data, there exist infinitely many weak solutions. 
\end{taggedthm}
 
\begin{taggedthm}{B} \label{thm:general-wilddata} 
	There exist initial data for which there are infinitely many admissible weak solutions. 
\end{taggedthm}

In this paper, weak solutions are solutions in the sense of distributions which are bounded (i.e., in $L^\infty$ in space and time) and continuous in time with respect the the weak topology of some $L^r$, $r\in (1,\infty)$. Weak solutions are called \emph{admissible} if the corresponding local energy inequality, as it is described in Sects.~\ref{subsubsec:intro-incomp-euler}-\ref{subsubsec:intro-QG}, below, is satisfied in the sense of distributions. 

In Sects.~\ref{subsubsec:intro-incomp-euler}-\ref{subsubsec:intro-QG}, below, we will have a more detailed look at the PDEs which are considered in this paper. There, we will give rigorous definitions of the notions of \emph{weak solutions} and \emph{admissible weak solutions}. In addition, we will state the appropriate versions of Thms.~\ref{thm:general-exforalldata} and \ref{thm:general-wilddata} rigorously in each case. Note however, that for some PDEs we will only show a result of type Thm.~\ref{thm:general-wilddata}. In this case we will explain the reason for not showing Thm.~\ref{thm:general-exforalldata} in the corresponding Sects.~\ref{subsubsec:intro-incomp-euler}-\ref{subsubsec:intro-QG}. The rigorously stated theorems will be proved in Sects.~\ref{sec:ex-incomp-euler}-\ref{sec:ex-QG}.

Table~\ref{tab:main-thms} summarizes the examples of inviscid PDEs which are treated in this paper. Furthermore, it refers to the rigorously formulated theorems in each case, so that the reader may quickly find them. The comments in Table~\ref{tab:main-thms} will be explained in more detail in the corresponding Sects.~\ref{subsubsec:intro-incomp-euler}-\ref{subsubsec:intro-QG}. 

\renewcommand{\arraystretch}{1.2} 
\begin{table}[h] 
	\centering 
	\begin{tabular}{|c|c|c|c|} \hline 
        \textsl{Model} & \textsl{Thm.~\ref{thm:general-exforalldata}} & \textsl{Thm.~\ref{thm:general-wilddata}} & \textsl{Comments} \tabularnewline \hline \hline 
		\multirow{2}{*}{Incompressible Euler}{} & Thm.~\ref{thm:incomp-euler-exforalldata} & & New global existence result in $L^\infty$ \tabularnewline \cline{2-4}
		& & Thm.~\ref{thm:incomp-euler-wilddata} & Already in \cite[Thm.~1(a)]{DelSze10} \tabularnewline \hline \hline 
        \multirow{2}{*}{Compressible Euler} & $-$ & & \tabularnewline \cline{2-4}
		& & Thm.~\ref{thm:compr-euler-wilddata} & Already in \cite[Thm.~1.3]{Feireisl14} \tabularnewline \hline \hline 
        \multirow{2}{*}{Shallow water} & $-$ & & \tabularnewline \cline{2-4} 
		& & Cor.~\ref{cor:sw-wilddata} & First such result for this system \tabularnewline \hline \hline 
        \multirow{2}{*}{Lake equation} & Thm.~\ref{thm:lake-exforalldata} & & Only for $C^1$ initial data \tabularnewline \cline{2-4} 
		& & Thm.~\ref{thm:lake-wilddata} & First such result for this system \tabularnewline \hline \hline
        \multirow{3}{*}{Hydrostatic Euler} & Thm.~\ref{thm:hydrost-euler-exforalldata} & & New global existence result in $L^\infty$ \tabularnewline 
        \cline{2-4}
		& & Thm.~\ref{thm:hydrost-euler-wilddata} & First construction of admissible solutions \tabularnewline 
        & & & (by using the natural energy) \tabularnewline \hline \hline 
        \multirow{2}{*}{Compr. invisc. primitive} & $-$ & & \tabularnewline \cline{2-4}
		& & Thm.~\ref{thm:compr-prim-wilddata} & First such result for this system \tabularnewline \hline \hline 
        \multirow{2}{*}{Quasi-geostrophic} & $-$ & & \tabularnewline \cline{2-4}
		& & Thm.~\ref{thm:QG-wilddata} & First such result for this system \tabularnewline \hline
	\end{tabular} 
	\caption{Examples of PDEs to which we apply the framework established in Sect.~\ref{sec:general} and where to find rigorous formulations of the two main statements in each case.} \label{tab:main-thms} 
\end{table} 
\renewcommand{\arraystretch}{1}

\subsubsection{Incompressible Euler equations} \label{subsubsec:intro-incomp-euler}

Firstly, we would like to demonstrate that the general convex integration framework developed in this paper recovers and extends the results for the incompressible Euler equations which were founded by \name{De~Lellis}-\name{Sz{\'e}kelyhidi}~\cite{DelSze10} and \name{Wiedemann}~\cite{Wiedemann11}. 

The incompressible Euler equations read
\begin{align}
	\Div v &= 0, \label{eq:incomp-euler-mass} \\
	\partial_t v + \Div (v\otimes v) + \Grad p &= 0. \label{eq:incomp-euler-mom}
\end{align}
Here the velocity $v$ and the pressure $p$ are the unknowns which take values in $\R^n$ and $\R$, respectively. Both unknowns are functions of time $t\in [0,T)$, for some $T\in (0,\infty) \cup \{\infty\}$, and space $x\in \Omega\subset \R^n$. In this paper we will only deal with two cases regarding the spatial domain $\Omega$: 
\renewcommand{\labelenumi}{(\roman{enumi})}
\renewcommand{\theenumi}{\labelenumi}
\begin{enumerate} 
	\item \label{item:domain-torus} either $\Omega=\T^3$, i.e.~the three-dimensional torus which may be viewed as a cube with periodic boundaries, 
	\item \label{item:domain-bdd} or $\Omega\subset \R^n$ a bounded Lipschitz domain. In this case we suppose the space dimension to be $n=2$ or $n=3$, and we supplement system \eqref{eq:incomp-euler-mass}, \eqref{eq:incomp-euler-mom} with the impermeability boundary conditions, i.e., 
    \begin{equation} \label{eq:impermeability-BC}
        v\cdot \nu|_{\partial\Omega} = 0,
    \end{equation}
    where $\nu$ is the outward pointing normal vector to the boundary $\partial \Omega$.
\end{enumerate}
\renewcommand{\labelenumi}{(\alph{enumi})}
\renewcommand{\theenumi}{\labelenumi}

We are interested in the initial value problem of system \eqref{eq:incomp-euler-mass}, \eqref{eq:incomp-euler-mom} where the initial data are denoted by $v_0$. The latter must have the property $\Div v_0 = 0$ in order to be compatible with \eqref{eq:incomp-euler-mass}. For initial data $v_0\in L^\infty(\Omega)$, the aforementioned incompressibility compatibility condition must at least be satisfied in the sense of distributions, i.e., 
\begin{equation} \label{eq:incomp-euler-compatibility-initialdata}
    \int_\Omega v_0 \cdot \Grad \phi \dx = 0 \qquad \text{ for all test functions\footnotemark }\ \phi\in \Cc (\closure{\Omega}). 
\end{equation} 
\footnotetext{See Rem.~\ref{rem:incomp-euler-testfunctions} below for a remark on the support of the test function $\phi$.} 

A simple computation shows that classical solutions of the incompressible Euler system \eqref{eq:incomp-euler-mass}, \eqref{eq:incomp-euler-mom} satisfy the following conservation law for the kinetic energy:
$$
	\partial_t \left(\half |v|^2\right) + \Div \left[\left(\half |v|^2 + p\right)v\right] = 0. 
$$
Besides weak solutions as they are defined below in Defn.~\ref{defn:incomp-euler-sol}~\ref{item:incomp-euler-sol-weak} with no additional constraint, we will also be interested in admissible weak solutions which satisfy in addition the local energy inequality
\begin{equation} \label{eq:incomp-euler-enineq}
	\partial_t \left(\half |v|^2\right) + \Div \left[\left(\half |v|^2 + p\right)v\right] \leq 0. 
\end{equation}

Now, we are ready to define weak solutions and admissible weak solutions. 

\begin{defn} \label{defn:incomp-euler-sol} 
    \begin{enumerate} 
        \item \label{item:incomp-euler-sol-weak} A pair\footnote{Usually in the literature, in the absence of physical boundaries a weak solution to the incompressible Euler equations \eqref{eq:incomp-euler-mass}, \eqref{eq:incomp-euler-mom} is defined to consist only of a velocity $v$. This is due to the fact that the pressure can be recovered from the velocity by solving the Poisson equation $\Delta p = - \Div\Div(v\otimes v)$ which follows from \eqref{eq:incomp-euler-mass}, \eqref{eq:incomp-euler-mom}. However, in the presences of physical boundaries the formulation of the correct boundary-value problem for the pressure for weak solutions of the Euler equations is not a priori clear. The correct form of the boundary-value problem was recently established in \cite{bardos2025} for Hölder continuous weak solutions, but this problem remains open for velocities with Besov or $L^p$ regularity. To keep the presentation simple, we decided to treat the pressure as a part of a weak solution. In particular, for the solutions constructed in this paper, the pressure can be explicitly stated without having to recover it from $v$.} $(v,p)\in L^\infty((0,T)\times \Omega;\R^n\times \R)$ is called a \emph{weak solution} of the incompressible Euler equations \eqref{eq:incomp-euler-mass}, \eqref{eq:incomp-euler-mom} with initial data\footnote{We would like to remark that for weak solutions in $L^\infty((0,T)\times\Omega)$ the initial data are ``attained'' only in a very weak sense, which is described in more detail in \cite[p.~6]{FeiNov}. Note that the solutions we construct in Thms.~\ref{thm:incomp-euler-exforalldata} and \ref{thm:incomp-euler-wilddata} below are weakly continuous in time. In this case, the initial data are attained as a weak limit.} $v_0\in L^\infty(\Omega;\R^n)$ (where $v_0$ satisfies the incompressibility compatibility condition \eqref{eq:incomp-euler-compatibility-initialdata}) if the equations are satisfied in the sense of distributions, i.e., 
        \begin{align*} 
            \int_0^T \int_\Omega v \cdot \Grad \phi \dx \dt &= 0, \\
            \int_0^T \int_\Omega \Big[ v \cdot \partial_t \varphi + v\otimes v : \Grad \varphi + p \ \Div \varphi \Big] \dx\dt + \int_\Omega v_0 \cdot \varphi(0,\cdot) \dx &= 0, 
        \end{align*}
        for all test functions 
        $$
        	(\phi,\varphi)\in \left\{ 
        	\begin{array}{ll} 
        		\Cc ([0,T)\times \T^3; \R \times \R^3) & \text{ in case \ref{item:domain-torus}, i.e. }\Omega = \T^3, \\
       			\Cc ([0,T)\times \closure{\Omega}; \R \times \R^n) \text{ with }\varphi\cdot \nu|_{\partial\Omega} = 0 & \text{ in case \ref{item:domain-bdd}}. 
       		\end{array} \right.
        $$
        Note that $\phi$ is a scalar-valued test function while $\varphi$ is vector-valued.
        
        \item \label{item:incomp-euler-sol-adm} A weak solution $(v,p)$ of the incompressible Euler equations \eqref{eq:incomp-euler-mass}, \eqref{eq:incomp-euler-mom} with initial data $v_0$ is called \emph{admissible} if the local energy inequality \eqref{eq:incomp-euler-enineq} holds in the sense of distributions, i.e. 
        \begin{align*} 
            \int_0^T \int_\Omega \Big[ \half |v|^2 \partial_t \Phi + \left(\half |v|^2 + p\right) v\cdot \Grad\Phi \Big] \dx \dt + \int_\Omega \half |v_0|^2 \Phi (0,\cdot) \dx &\geq  0, 
        \end{align*}
        for all test functions $\Phi\in \Cc ([0,T)\times \closure{\Omega}; \R^+_0)$. 
    \end{enumerate}
\end{defn}

\begin{rem} \label{rem:incomp-euler-testfunctions} 
    The reader should notice that the test functions $\phi,\varphi,\Phi$ in Defn.~\ref{defn:incomp-euler-sol} have compact support in $[0,T)\times \closure{\Omega}$. In case~\ref{item:domain-torus} (i.e.~$\Omega=\T^3$), it holds that $\closure{\Omega} = \T^3$ and consequently the test functions are just periodic. In case~\ref{item:domain-bdd} (i.e.~$\Omega\subset \R^n$ bounded with $n=2$ or $n=3$, see above) we obtain in particular, that the test functions do not need to vanish at the boundary $\partial\Omega$. This way the boundary condition $v\cdot \nu|_{\partial\Omega} = 0$ is implicitly incorporated into Defn.~\ref{defn:incomp-euler-sol}.
\end{rem}

Now, we are able to precisely state Thms.~\ref{thm:general-exforalldata} and \ref{thm:general-wilddata} in the context of the incompressible Euler equations \eqref{eq:incomp-euler-mass}, \eqref{eq:incomp-euler-mom}. 

\begin{thm} \label{thm:incomp-euler-exforalldata} 
    Let $\Omega = \T^3$ (case~\ref{item:domain-torus}) and $T\in (0,\infty)\cup \{\infty\}$. Moreover, let $v_0\in L^\infty(\Omega;\R^3)$ satisfying the compatibility condition \eqref{eq:incomp-euler-compatibility-initialdata}. Then there exist infinitely many weak solutions $(v,p)$ of the incompressible Euler equations \eqref{eq:incomp-euler-mass}, \eqref{eq:incomp-euler-mom} in the sense of\footnote{In particular, in contrast to \cite{Wiedemann11} the solutions are in $L^\infty$.} Defn.~\ref{defn:incomp-euler-sol}~\ref{item:incomp-euler-sol-weak}, all of which satisfy $v\in \Cweak([0,T);L^2(\Omega;\R^3))$.
\end{thm} 

\begin{thm} \label{thm:incomp-euler-wilddata} 
    Let $n=2,3$, $\Omega\subset \R^n$ a bounded Lipschitz domain (case~\ref{item:domain-bdd}), and $T\in (0,\infty)\cup \{\infty\}$. There exist initial data $v_0\in L^\infty(\Omega;\R^n)$ satisfying the compatibility condition \eqref{eq:incomp-euler-compatibility-initialdata}, for which there are infinitely many admissible weak solutions $(v,p)$ of the incompressible Euler equations \eqref{eq:incomp-euler-mass}, \eqref{eq:incomp-euler-mom} in the sense of Defn.~\ref{defn:incomp-euler-sol}~\ref{item:incomp-euler-sol-adm}. All of these solutions satisfy $v\in \Cweak([0,T);L^2(\Omega;\R^n))$.
\end{thm} 

To the best knowledge of the authors, Thm.~\ref{thm:incomp-euler-exforalldata} is a new result in the literature on the incompressible Euler equations. The global existence of weak solutions for $L^2$-data has been proved by \name{Wiedemann} in \cite[Thm.~1]{Wiedemann11}. The global existence of weak solutions for Hölder continuous data has been established in \cite{khor2022} (with an infinitesimal loss of Hölder regularity between the data and the solution). Results on the density of ``wild data'', i.e., initial data for which there are infinitely many admissible weak solutions, have been obtained in \cite{SzeWie12,DanSze17,daneri2021} (and also \cite{daneri2014}). Thm.~\ref{thm:incomp-euler-wilddata} is already available in the literature, see \name{De~Lellis}-\name{Sz{\'e}kelyhidi}~\cite[Thm.~1(a)]{DelSze10}. 

\begin{rem} \label{rem:incomp-euler-exforalldata-otherdomains} 
	In this paper we do not prove a two-dimensional version of Thm.~\ref{thm:incomp-euler-exforalldata} (with initial data still in $L^\infty$). To obtain such a 2D version, one needs a 2D version of Prop.~\ref{prop:app-sfad}, below, (i.e.~the construction of a subsolution for all $L^\infty$ data) which we omit for the sake of simplicity of the presentation, and leave for future work. On the other hand, if one is only interested in existence of weak solutions for all initial data in $C^1$, then the restriction to the three-dimensional torus $\T^3$ is not needed. In particular, one could prove a version of Thm.~\ref{thm:incomp-euler-exforalldata} for all data $v_0\in C^1(\closure{\Omega};\R^3)$ where $\Omega=\T^2$ or $\Omega\subset \R^n$ as in case~\ref{item:domain-bdd}. We obtain such a result in the context of the lake equations, see Thm.~\ref{thm:lake-exforalldata}, below. 
\end{rem} 

\begin{rem} \label{rem:incomp-euler-wilddata-otherdomains}
    Thm.~\ref{thm:incomp-euler-wilddata} still holds if $\Omega=\T^3$ (i.e.~case~\ref{item:domain-torus}) or $\Omega=\T^2$. 
\end{rem}

\begin{rem} \label{rem:incomp-euler-furtherresults} 
    Let us also mention that our framework is able to recover even more results on convex integration for the incompressible Euler equations \eqref{eq:incomp-euler-mass}, \eqref{eq:incomp-euler-mom}, e.g., the results in \cite{DelSze10,SzeWie12}, but due to space limitations we will not do so here.
\end{rem}

\subsubsection{Barotropic compressible Euler equations} \label{subsubsec:intro-compr-euler}

The barotropic compressible Euler equations are given by 
\begin{align}
    \partial_t \rho + \Div (\rho u) &= 0, \label{eq:compr-euler-mass} \\ 
    \partial_t (\rho u) + \Div (\rho u\otimes u) + \Grad p(\rho) &= 0, \label{eq:compr-euler-mom}
\end{align}
where the unknowns are the density $\rho$ and the velocity $u$ which take values in\footnote{In this paper we exclude vacuum, i.e.~$\rho$ will be strictly positive.} $\R^+$ and $\R^n$, respectively. In the context of the barotropic compressible Euler equations \eqref{eq:compr-euler-mass}, \eqref{eq:compr-euler-mom}, the pressure is a given function of the density, i.e., $p=p(\rho)$, which we assume to be in $C^1((0,\infty);\R^+)$. Regarding the spatial domain, we only consider $\Omega=\T^n$ where $n=2$ or $n=3$. 

In contrast to the incompressible Euler equations introduced in Sect.~\ref{subsubsec:intro-incomp-euler}, there is no compatibility condition for the initial data $\rho_0,u_0$ in the context of the compressible Euler equations \eqref{eq:compr-euler-mass}, \eqref{eq:compr-euler-mom}. 

The local energy inequality in the compressible case reads 
\begin{equation} \label{eq:compr-euler-enineq}
    \partial_t \left( \half \rho |u|^2 + P(\rho) \right) + \Div \left[ \left( \half \rho |u|^2 + P(\rho) + p(\rho) \right) u \right] \leq 0,
\end{equation} 
where $P(\rho)$ is the pressure potential given by 
\begin{equation} \label{eq:pressure-potential} 
    P(\rho) = \rho \int_{\ov{\rho}}^\rho \frac{p(r)}{r^2} \dr, 
\end{equation}
with an arbitrary number $\ov{\rho}$.

\begin{defn} \label{defn:compr-euler-sol} 
    \begin{enumerate} 
        \item \label{item:compr-euler-sol-weak} A pair\footnote{The reader should keep in mind that in the context of the compressible Euler system \eqref{eq:compr-euler-mass}, \eqref{eq:compr-euler-mom}, we only consider $\Omega=\T^n$ with $n=2$ or $n=3$.} $(\rho,u)\in L^\infty((0,T)\times \Omega;\R^+ \times \R^n)$ is called a \emph{weak solution} of the barotropic compressible Euler equations \eqref{eq:compr-euler-mass}, \eqref{eq:compr-euler-mom} with initial data $(\rho_0,u_0)\in L^\infty(\Omega;\R^+\times \R^n)$ if the equations are satisfied in the sense of distributions, i.e. 
        \begin{align*} 
            \int_0^T \int_\Omega \Big[ \rho \partial_t \phi + \rho u \cdot \Grad \phi \Big] \dx \dt + \int_\Omega \rho_0 \cdot \phi(0,\cdot) \dx &= 0, \\
            \int_0^T \int_\Omega \Big[ \rho u \cdot \partial_t \varphi + \rho u\otimes u : \Grad \varphi + p(\rho) \Div \varphi \Big] \dx\dt + \int_\Omega \rho_0 u_0 \cdot \varphi(0,\cdot) \dx &= 0, 
        \end{align*}
        for all test functions $(\phi,\varphi)\in \Cc ([0,T)\times \Omega; \R \times \R^n)$.

        \item \label{item:compr-euler-sol-adm} A weak solution $(\rho,u)$ of the compressible Euler equations \eqref{eq:compr-euler-mass}, \eqref{eq:compr-euler-mom} with initial data $\rho_0,u_0$ is called \emph{admissible} if the energy inequality \eqref{eq:compr-euler-enineq} holds in the sense of distributions, i.e. 
        \begin{align*} 
            \int_0^T \int_\Omega \Big[ \left( \half \rho |u|^2 + P(\rho) \right) \partial_t \Phi + \left(\half \rho |u|^2 + P(\rho) + p(\rho)\right) u\cdot \Grad\Phi \Big] \dx \dt & \\
            + \int_\Omega \left( \half \rho_0 |u_0|^2 + P(\rho_0) \right) \Phi (0,\cdot) \dx &\geq  0, 
        \end{align*}
        for all test functions $\Phi\in \Cc ([0,T)\times \Omega; \R^+_0)$. 
    \end{enumerate}
\end{defn}

In the context of the compressible Euler equations, we do not prove Thm.~\ref{thm:general-exforalldata}. The reason for this is explained in the following remark. 

\begin{rem}[See also Rem.~\ref{rem:incomp-euler-exforalldata-otherdomains}] \label{rem:compr-models-no-exforalldata} 
	In order to prove Thm.~\ref{thm:general-exforalldata} for the barotropic compressible Euler equations \eqref{eq:compr-euler-mass}, \eqref{eq:compr-euler-mom}, a ``compressible version'' of Prop.~\ref{prop:app-sfad}, below, would be necessary. To keep the presentation simple, we do not show this in the current paper. If one is only interested in existence for more regular initial data, one could also proceed using the ideas of \name{Donatelli~et~al.}~\cite{DonFeiMar15} together with the approach elaborated in \name{Feireisl}~\cite{Feireisl14} or \name{Chiodaroli}~\cite{Chiodaroli14}.
\end{rem}

Next, we state the specific version of Thm.~\ref{thm:general-wilddata} for this model. 

\begin{thm} \label{thm:compr-euler-wilddata} 
    Let $n=2,3$, $\Omega=\T^n$ and $T\in (0,\infty)\cup \{\infty\}$. Let furthermore $p\in C^1((0,\infty);\R^+)$, and $0<\un{\rho}<\ov{\rho}<\infty$. Then there exists $\ep>0$, depending only on $\un{\rho}$, $\ov{\rho}$ and the pressure function $p$, with the following property. For any initial density $\rho_0\in C^1(\Omega;\R^+)$ with 
    $$
        \un{\rho}\leq \rho_0(x) \leq \ov{\rho} \quad \text{ for all }x\in \Omega\qquad \text{ and } \qquad\sup_{x\in\Omega}|\Grad \rho_0(x)|\leq \ep, 
    $$
    there exists an initial velocity $u_0\in L^\infty(\Omega;\R^n)$, for which there are infinitely many admissible weak solutions $(\rho,u)$ of the compressible Euler equations \eqref{eq:compr-euler-mass}, \eqref{eq:compr-euler-mom} in the sense of Defn.~\ref{defn:compr-euler-sol}~\ref{item:compr-euler-sol-adm}. All of these solutions satisfy $\rho\in C^1([0,T)\times \Omega;\R^+)$ with 
    $$
        \un{\rho}\leq \rho(t,x) \leq \ov{\rho} \quad \text{ for all }(t,x)\in [0,T)\times \Omega, 
    $$
    and $u\in \Cweak([0,T);L^2(\Omega;\R^n))$. 
\end{thm}

Thm.~\ref{thm:compr-euler-wilddata} was originally proved by \name{Feireisl}~\cite[Thm.~1.3]{Feireisl14}. We recall it here for the sake of completeness.

\begin{rem} \label{rem:compr-euler-wilddata-DelSze}
    If one is only interested in existence of $(\rho_0,u_0)\in L^\infty(\Omega;\R^+  \times \R^n)$ for which there are infinitely many admissible weak solutions, then one could choose $\rho\equiv 1$ and utilize Thm.~\ref{thm:incomp-euler-wilddata}. This way one can even consider other domains, in particular $\Omega\subset \R^n$ bounded with impermeability boundary conditions. This procedure was originally carried out by \name{De~Lellis}-\name{Sz{\'e}kelyhidi}~\cite[Thm.~2]{DelSze10}. 
\end{rem}

\begin{rem} \label{rem:compr-euler-wilddata-Chiod}
    Let us also mention that a local-in-time version of Thm.~\ref{thm:compr-euler-wilddata} was established by \name{Chiodaroli}~\cite[Thm.~1.1]{Chiodaroli14}. Here ``local-in-time'' means that the final time $T$ depends on $\rho_0$. In this case, the requirement $\sup_{x\in\Omega}|\Grad \rho_0(x)|\leq \ep$ is no longer necessary and can be dropped. We prove such a result in the context of the lake equations \eqref{eq:lake-mass}, \eqref{eq:lake-mom}, see Thm.~\ref{thm:lake-wilddata}, below.
\end{rem}

\begin{rem} \label{rem:compr-euler-wilddata-dependencies}
    As already observed by \name{Feireisl}~\cite{Feireisl14}, when splitting the initial momentum $m_0 = \rho_0 u_0$ in Thm.~\ref{thm:compr-euler-wilddata} in terms of Helmholtz decomposition, i.e., $m_0 = v_0 + \Grad\Psi_0$, where $\Div v_0=0$, then $\Psi_0$ depends on the initial density $\rho_0$, whereas $v_0$ can be constructed in such a way that it only depends on $\un{\rho}$, $\ov{\rho}$ and $p$ but not on the specific choice of $\rho_0$.
\end{rem} 

\begin{rem}
    We would like to stress that our convex integration framework which is presented in Sect.~\ref{sec:general}, below, works on any bounded Lipschitz domain $\Omega$, and not only just on the torus $\T^n$. However, the construction of a subsolution, which is also needed to prove Thm.~\ref{thm:compr-euler-wilddata}, is quite involved for the case $\Omega\neq \T^n$. To be precise, it is even unclear if a subsolutions exists in this case. For this reason we restrict ourselves to the case $\Omega=\T^n$. 
\end{rem}

\subsubsection{Shallow water equations} \label{subsubsec:intro-sw}

Next, we look at the shallow water equations 
\begin{align} 
	\partial_t h + \Div (h u) &= 0, \label{eq:sw-mass} \\
	\partial_t (h u) + \Div (h u \otimes u) + \Grad \left( \half h^2\right) &= 0, \label{eq:sw-mom}
\end{align}
with unknown water height $h$ and velocity $u$ taking values in\footnote{Similar to the compressible Euler case, where we excluded vacuum (i.e.~we required $\rho>0$, see above), we exclude dry points when studying the shallow water equations. The latter means that the water height will be strictly larger than $0$ in this paper.} $\R^+$ and $\R^2$ respectively. In the context of the shallow water equations \eqref{eq:sw-mass}, \eqref{eq:sw-mom}, the spatial domain $\Omega$ is considered to be the two-dimensional torus, i.e., $\Omega= \T^2$. The shallow water equations model is derived from the incompressible Euler equations with constant density at the limit of small aspect ratio of the domain occupied by the fluid with solid bottom boundary and free upper boundary. Therefore, one can assume that the hydrostatic balance is valid and utilize it in the derivation. For a more detailed introduction of the shallow water equations, we refer to \cite[Chap.~4]{Vallis} and also \cite{bresch2009}. 

The local energy inequality for the shallow water equations \eqref{eq:sw-mass}, \eqref{eq:sw-mom} is given by 
\begin{equation} \label{eq:sw-enineq}
	\partial_t \left( \half h |u|^2 + \half h^2\right) + \Div \left[ \left( \half h |u|^2 + h^2 \right) u \right] \leq 0. 
\end{equation}

It is a simple and well-known observation that the shallow water equations \eqref{eq:sw-mass}, \eqref{eq:sw-mom} coincide with the compressible Euler system \eqref{eq:compr-euler-mass}, \eqref{eq:compr-euler-mom} via identifying the height $h$ with the density $\rho$, and choosing the pressure to be $p(\rho)=\half \rho^2$. Consequently it is not necessary to write down the definition of weak and admissible weak solutions in the context of \eqref{eq:sw-mass}, \eqref{eq:sw-mom} because these definitions are already contained in the corresponding definition for the compressible Euler system \eqref{eq:compr-euler-mass}, \eqref{eq:compr-euler-mom}, see Defn.~\ref{defn:compr-euler-sol}.

Like for the compressible Euler equations \eqref{eq:compr-euler-mass}, \eqref{eq:compr-euler-mom}, we do not show Thm.~\ref{thm:general-exforalldata}, above, in the context of the shallow water equations \eqref{eq:sw-mass}, \eqref{eq:sw-mom}, see also Rem.~\ref{rem:compr-models-no-exforalldata}. 

For the shallow water equations Thm.~\ref{thm:general-wilddata} is rigorously stated as follows.

\begin{cor} \label{cor:sw-wilddata} 
	Let $\Omega= \T^2$ and $T\in (0,\infty)\cup \{\infty\}$. Let furthermore $0<\un{h}<\ov{h}<\infty$. Then there exists $\ep>0$, depending only on $\un{h}$ and $\ov{h}$, with the following property. For any initial height $h_0\in C^1(\Omega;\R^+)$ with 
    $$
        \un{h}\leq h_0(x) \leq \ov{h} \quad \text{ for all }x\in \Omega\qquad \text{ and } \qquad\sup_{x\in\Omega}|\Grad h_0(x)|\leq \ep, 
    $$
    there exists an initial velocity $u_0\in L^\infty(\Omega;\R^2)$, for which there are infinitely many admissible weak solutions $(h,u)$ of the shallow water equations \eqref{eq:sw-mass}, \eqref{eq:sw-mom}. All of these solutions satisfy $h\in C^1([0,T)\times \Omega;\R^+)$ with 
    $$
        \un{h}\leq h(t,x) \leq \ov{h} \quad \text{ for all }(t,x)\in [0,T)\times \Omega, 
    $$ 
    and $u\in \Cweak([0,T);L^2(\Omega;\R^2))$. 
\end{cor} 

The authors are not aware of any paper which states Cor.~\ref{cor:sw-wilddata} in this form (i.e., for the shallow water equations), although naturally it is closely related to the existing results in the literature for the compressible Euler equations (see, e.g., \cite{Feireisl14}). However, we believe that it is worth mentioning Cor.~\ref{cor:sw-wilddata} here since the shallow water equations fit into the context of this paper, namely as a geophysical model.

Rems.~\ref{rem:compr-euler-wilddata-DelSze}-\ref{rem:compr-euler-wilddata-dependencies} hold analogously for the shallow water equations \eqref{eq:sw-mass}, \eqref{eq:sw-mom}. In particular, if one is only interested in a local-in-time version of Cor.~\ref{cor:sw-wilddata}, the requirement $\sup_{x\in\Omega}|\Grad h_0(x)|\leq \ep$ is not needed.

\subsubsection{Lake equations} \label{subsubsec:intro-lake}

Another system, which we consider, are the lake equations
\begin{align} 
	\Div (bu) &= 0, \label{eq:lake-mass} \\
	\partial_t (bu) + \Div (bu\otimes u) + b \Grad p &= 0. \label{eq:lake-mom}
\end{align} 
Here the velocity $u$ (which takes values in $\R^2$) and the pressure $p$ (taking values in $\R$) are the unknowns, while the topography $b\in C^1(\closure{\Omega};\R^+)$ is a given function. We consider a bounded Lipschitz domain $\Omega\subset \R^2$ together with impermeability boundary conditions $u\cdot \nu|_{\partial\Omega} = 0$. The lake equations can be formally derived from the three-dimensional incompressible Euler equations (with a free upper surface) by taking the small aspect ratio and small Froude number limits in a domain with varying bottom topography \cite{LevOliTit96}.

Analogously to the incompressible Euler equations \eqref{eq:incomp-euler-mass}, \eqref{eq:incomp-euler-mom}, the initial data $u_0\in L^\infty(\Omega;\R^2)$ for the lake equations must be compatible with \eqref{eq:lake-mass}, i.e.~we require
\begin{equation} \label{eq:lake-compatibility-initialdata}
	\int_\Omega b u_0 \cdot \Grad \phi \dx = 0 \qquad \text{ for all test functions }\phi\in \Cc (\closure{\Omega}) .
\end{equation} 

For the lake equations, the local energy inequality reads
\begin{equation} \label{eq:lake-enineq}
	\partial_t \left(\half b |u|^2\right) + \Div \left[\left( \half b |u|^2 + bp \right) u \right] \leq 0. 
\end{equation}
The global well-posedness of the lake equations was established in \cite{LevOliTit96}, under the assumption of nondegenerate bottom topography. The case with vanishing topography on the boundary (i.e., the shore) was studied in \cite{bresch2006}. It was proved in \cite{lacave2014} that weak solutions of the lake equations are stable under perturbations in the topography, which also leads to the existence of weak solutions for singular domains. The case of emerging or vanishing islands was treated in \cite{hientzsch2021}. 

The vanishing viscosity limit from the viscous lake equations with Navier boundary conditions to the (inviscid) lake equations was proved in \cite{jiu2012}. It was shown in \cite{altaki2022} that the (inviscid) lake equations possess global classical solutions and that the weak solutions of the viscous system converge to the classical solution. 

In this paper, weak and admissible weak solutions to the lake equations \eqref{eq:lake-mass}, \eqref{eq:lake-mom} are defined analogously to Defns.~\ref{defn:incomp-euler-sol} and \ref{defn:compr-euler-sol}.

Now we are ready to provide rigorous formulations of Thms.~\ref{thm:general-exforalldata} and \ref{thm:general-wilddata} for this model.

\begin{thm} \label{thm:lake-exforalldata} 
    Let $\Omega\subset \R^2$ be a bounded Lipschitz domain and $T\in (0,\infty)\cup \{\infty\}$. Moreover, let $b\in C^1(\closure{\Omega};\R^+)$, and $u_0\in C^1(\closure{\Omega};\R^2)$ satisfying the compatibility condtion \eqref{eq:lake-compatibility-initialdata}. Then there exist infinitely many weak solutions $(u,p)$ of the lake equations \eqref{eq:lake-mass}, \eqref{eq:lake-mom}, all of which satisfy $u\in \Cweak([0,T);L^2(\Omega;\R^2))$. 
\end{thm} 

\begin{thm} \label{thm:lake-wilddata} 
	Let $\Omega\subset \R^2$ be a bounded Lipschitz domain and $b\in C^1(\closure{\Omega};\R^+)$. There exists a time $T\in (0,\infty)\cup \{\infty\}$, depending on $b$, and initial data $u_0\in L^\infty(\Omega;\R^2)$ satisfying the compatibility condition \eqref{eq:lake-compatibility-initialdata}, for which there are infinitely many admissible weak solutions $u$ of the lake equations \eqref{eq:lake-mass}, \eqref{eq:lake-mom} on $[0,T)$. All of these solutions satisfy $u\in \Cweak([0,T);L^2(\Omega;\R^2))$.
\end{thm} 

Thms.~\ref{thm:lake-exforalldata} and \ref{thm:lake-wilddata} as stated above are -- to the best of the authors' knowledge -- not known in the literature so far. However, these results are similar in spirit to \name{Chiodaroli}'s work \cite{Chiodaroli14} on the compressible Euler equations. We will sketch the proofs of Thms.~\ref{thm:lake-exforalldata} and \ref{thm:lake-wilddata} in Sect.~\ref{subsec:ex-lake-pf}, below.

\begin{rem}
    In comparision with Thm.~\ref{thm:incomp-euler-exforalldata}, we would like to underline that the initial data considered in Thm.~\ref{thm:lake-exforalldata} are $C^1$ rather than merely $L^\infty$. To extend the latter result to $L^\infty$ data, one needs a 2D version of Prop.~\ref{prop:app-sfad}, below, see also Rem.~\ref{rem:incomp-euler-exforalldata-otherdomains}. Furthermore, we note that Thm.~\ref{thm:lake-exforalldata} can be also shown for $\Omega=\T^2$.
\end{rem}

\begin{rem}
    Thm.~\ref{thm:lake-wilddata} is also valid for $\Omega=\T^2$. Moreover, a global-in-time version of Thm.~\ref{thm:lake-wilddata} analogous to Thm.~\ref{thm:compr-euler-wilddata} is true as well. Then, however, there are some restrictions on the topography $b$ like for the initial density in Thm.~\ref{thm:compr-euler-wilddata}. See also Rem.~\ref{rem:compr-euler-wilddata-Chiod}.
\end{rem}

\subsubsection{Hydrostatic Euler equations -- inviscid primitive equations} \label{subsubsec:intro-hydrost-euler}

The hydrostatic Euler equations are given by
\begin{align}
	\Divh u + \partial_z w &= 0, \label{eq:hydrost-euler-mass} \\
	\partial_t u + \Divh (u\otimes u) + \partial_z (uw) + \Gradh p &=0, \label{eq:hydrost-euler-mom} \\
	\partial_z p &=0 \label{eq:hydrost-euler-p} 
\end{align}
with unknowns horizontal velocity $u$, vertical velocity $w$ and pressure $p$ which take values in $\R^2$, $\R$ and $\R$, respectively. All unknowns are functions of time $t$ and the three-dimensional space variable $x=(x_1,x_2,z)$. For the divergence and the gradient with respect to $(x_1,x_2)$ we write $\Divh$ and $\Gradh$, respectively (here the subscript $h$ stands for \emph{horizontal}). 

The hydrostatic Euler equations \eqref{eq:hydrost-euler-mass}-\eqref{eq:hydrost-euler-p} can be formally derived from the incompressible Euler equations \eqref{eq:incomp-euler-mass}, \eqref{eq:incomp-euler-mom} by considering the small aspect ratio limit. The hydrostatic Euler equations are also known as the incompressible inviscid primitive equations of oceanic and atmospheric dynamics.

For the spatial domain $\Omega$ we consider the channel $\Omega = \T^2 \times (0,1)$. The boundary condition reads $w|_{z=0,1} = 0$, meaning that the fluid does not penetrate bottom and top of the channel.

Like in the context of the incompressible Euler equations \eqref{eq:incomp-euler-mass}, \eqref{eq:incomp-euler-mom}, the initial data $(u_0,w_0) \in L^\infty(\Omega;\R^2 \times \R)$ for the hydrostatic Euler system have to be compatible with \eqref{eq:hydrost-euler-mass}, i.e., they must satisfy 
\begin{equation} \label{eq:hydrost-euler-compatibility-initialdata}
    \int_\Omega \Big[ u_0 \cdot \Gradh \phi + w_0 \partial_z \phi\Big] \dx = 0 \qquad \text{ for all test functions }\phi\in \Cc (\closure{\Omega}) .
\end{equation} 

\begin{rem} \label{rem:hydrost-euler-initialdata} 
    Since the vertical velocity $w$ is not a dynamical unknown in the the context of the hydrostatic Euler equations \eqref{eq:hydrost-euler-mass}-\eqref{eq:hydrost-euler-p}, one usually treats $u$ as the only unknown for the system while $w$ and $p$ are recovered from $u$. For the same reason, one only imposes initial data $u_0$ for the horizontal velocity $u$. To keep the presentation simple, we decided to also impose initial data $w_0$ for the vertical velocity $w$. Then $w_0$ may be understood as the vertical velocity which has been recovered from the initial horizontal velocity $u_0$, i.e., $u_0$ and $w_0$ need to be compatible satisfying \eqref{eq:hydrost-euler-mass}. Similarly a solution will be a triple $(u,w,p)$ rather than just a horizontal velocity $u$. The role of the vertical velocity $w$ is ``similar to the role of the pressure $p$'' in the incompressible Euler equations \eqref{eq:incomp-euler-mass}, \eqref{eq:incomp-euler-mom}.

    On the other hand, it is very important that for weak solutions we require that both $u$ and $w$ lie in $L^\infty$, see Defn.~\ref{defn:hydrost-euler-sol}. Note that in general if the horizontal velocity is in $L^\infty$, then the corresponding vertical velocity $w$ could be of lower regularity than $L^\infty$, see \cite{BouMarTit23} for more details on this issue and how it can be addressed by introducing a different class of weak solutions. We exclude such a scenario by assumption. Similarly, we will assume in Thms.~\ref{thm:hydrost-euler-exforalldata} and \ref{thm:hydrost-euler-wilddata} that both $u_0$ and $w_0$ are in $L^\infty$.
\end{rem} 

The local energy inequality in the context of the hydrostatic Euler system \eqref{eq:hydrost-euler-mass}-\eqref{eq:hydrost-euler-p} reads 
\begin{equation} \label{eq:hydrost-euler-enineq}
	\partial_t \left(\half |u|^2\right) + \Divh \left[\left(\half |u|^2 + p\right)u\right] + \partial_z \left[\left(\half |u|^2 + p\right)w\right] \leq 0. 
\end{equation}

\begin{rem} \label{rem:hydrost-euler-energy}
	Note that the left-hand side of \eqref{eq:hydrost-euler-enineq} describes the time evolution of the \emph{horizontal} local kinetic energy $\half |u|^2$ rather than the complete three-dimensional kinetic energy $\half (|u|^2 + w^2)$. The horizontal kinetic energy is indeed the right energy to be considered in this context because it is conserved by weak solutions of \eqref{eq:hydrost-euler-mass}-\eqref{eq:hydrost-euler-p} with sufficient regularity as shown by the authors in \cite{BouMarTit23}.
\end{rem}

The viscous primitive equations, i.e., the system \eqref{eq:hydrost-euler-mass}-\eqref{eq:hydrost-euler-p} with viscosity, were first introduced in \cite{richardson1922} for numerical weather prediction. The rigorous mathematical study of these equations was initiated in \cite{LioTemWan92,LioTemWan1995,LioTemWan92b}, in which the global existence of weak solutions was proved. The local existence of strong solutions in three dimensions was then established in \cite{guillen2001}, while the global well-posedness in two dimensions was proved in \cite{bresch2003,bresch2005}.

Subsequently, the global well-posedness of the three-dimensional viscous primitive equations was first proved in \cite{cao2005} (see also \cite{kobelkov2006}). The case with Dirichlet boundary conditions was then treated in \cite{kukavica2007a,kukavica2007b}. The existence of strong solutions was established in \cite{hieber2016} for a larger class of initial data, by means of semigroup techniques. The non-uniqueness of very weak solutions to the viscous primitive equations was established in \cite{BouMarTit24}. In \cite{cao2014,cao2016,cao2020} (and see references therein) the global well-posedness of the primitive equations with partial viscosities and diffusivities was established. The global well-posedness of these equations with a nonlinear equation of state was established in \cite{korn2021}, while in \cite{korn2024} the equations with eddy parametrization were treated. 

In the inviscid case, in \cite{renardy2009} the primitive equations were shown to be linearly ill-posed. Thereafter the nonlinear ill-posedness of the inviscid primitive equations was demonstrated in \cite{ibrahim2021}. The local well-posedness of the system for analytic data was proved in \cite{kukavica2010,kukavica2011}, and subsequently in \cite{GILT22} rotation was shown to prolong the lifespan of the analytic strong solution. One can prove the local well-posedness in Sobolev spaces under convexity assumptions (the local Rayleigh condition) on the initial data \cite{brenier1999,masmoudi2012}. In \cite{bianchini2024a} the local well-posedness in Sobolev spaces was established under stable stratification and including eddy diffusivity. The local well-posedness for the free boundary problem for analytic data was established in \cite{ignatova2012}.

It was found in \cite{cao2015,wong2015} (see also \cite{canulef2024,cui2024}) that solutions of these equations become singular in finite time. A more quantitative description of the singularity formation was given in \cite{collot2023} (see also \cite{ILQT25pre}). An analogue of the Onsager conjecture was considered in \cite{BouMarTit23}, in which new types of weak solutions for the equations were introduced and a `family' of Onsager conjectures was identified. The non-uniqueness of weak solutions for these equations was proved in \cite{ChiMic17,BouMarTit24}. 

In the viscous case the aspect ratio limit was justified in a weak sense in \cite{azerad2001}, while it was significantly established in a strong sense in \cite{Li2019} (with explicit estimates for the convergence rate in terms of the aspect ratio). The hydrostatic limit with different scalings of the viscosity under the aspect ratio has been considered in \cite{li2022,furukawa2023}. It was proved in \cite{grenier1999,bianchini2024b} that in the inviscid case the hydrostatic limit does not hold in certain settings. The hydrostatic limit was justified under the local Rayleigh condition in \cite{brenier2003,grenier1999}. The viscosity limit for the primitive equations with Dirichlet boundary conditions was established in \cite{kukavica2016} for analytic data.

For the hydrostatic Euler equations \eqref{eq:hydrost-euler-mass}-\eqref{eq:hydrost-euler-p} several notions of weak solutions are studied in the literature, e.g.~generalized weak solutions introduced by the authors in \cite{BouMarTit23,BouMarTit24}. However in the current paper we will work with the standard notion of weak solution, which is analogous to Defns.~\ref{defn:incomp-euler-sol} and \ref{defn:compr-euler-sol}. For completeness and due to some technical details regarding the boundary condition (see Rem.~\ref{rem:hydrost-euler-testfunctions} below), we give a precise definition as follows.

\begin{defn} \label{defn:hydrost-euler-sol} 
	\begin{enumerate}
		\item \label{item:hydrost-euler-sol-weak} A triple $(u,w,p)\in L^\infty((0,T)\times \Omega; \R^2 \times \R \times \R)$ is called a \emph{weak solution} of the hydrostatic Euler equations \eqref{eq:hydrost-euler-mass}-\eqref{eq:hydrost-euler-p} with initial data $(u_0,w_0)\in L^\infty(\Omega;\R^2 \times \R)$ (which satisfy the compatibility condition \eqref{eq:hydrost-euler-compatibility-initialdata}) if the equations are satisfied in the sense of distributions, i.e., 
		\begin{align*} 
			\int_0^T \int_\Omega \Big[ u \cdot \Gradh \phi + w \partial_z \phi\Big] \dx \dt &= 0, \\
			\int_0^T \int_\Omega \Big[ u \cdot \partial_t \varphi + u\otimes u : \Gradh \varphi + uw \cdot \partial_z \varphi + p \ \Divh \varphi \Big] \dx\dt + \int_\Omega u_0 \cdot \varphi(0,\cdot) \dx &= 0, \\
			\int_0^T \int_\Omega p \partial_z \psi \dx\dt &= 0,
		\end{align*}
		for all test functions $(\phi,\varphi)\in \Cc ([0,T)\times \closure{\Omega}; \R \times \R^2)$, $\psi\in \Cc ([0,T)\times \Omega)$. 
		
		\item \label{item:hydrost-euler-sol-adm} A weak solution $(u,w,p)$ of the hydrostatic Euler equations \eqref{eq:hydrost-euler-mass}-\eqref{eq:hydrost-euler-p} with initial data $(u_0,w_0)$ is called \emph{admissible} if the local energy inequality \eqref{eq:hydrost-euler-enineq} holds in the sense of distributions, i.e.,  
		\begin{align*} 
			\int_0^T \int_\Omega \Big[ \half |u|^2 \partial_t \Phi + \left(\half |u|^2 + p\right) ( u\cdot \Gradh\Phi + w \partial_z \Phi ) \Big] \dx \dt + \int_\Omega \half |u_0|^2 \Phi (0,\cdot) \dx &\geq  0, 
		\end{align*}
		for all test functions $\Phi\in \Cc ([0,T)\times \closure{\Omega}; \R^+_0)$. 
	\end{enumerate}
\end{defn}

\begin{rem} \label{rem:hydrost-euler-testfunctions} 
	Like in the case of the incompressible Euler equations, see Rem.~\ref{rem:incomp-euler-testfunctions}, the boundary condition $w|_{z=0,1} = 0$ is implicitly included in Defn.~\ref{defn:hydrost-euler-sol} by the choice of the test functions $\phi,\varphi$ and $\Phi$. The reader should notice that the test function $\psi$ has to vanish at the boundary $\{z=0,1\}$.
\end{rem}

Now we are ready to present the rigorous statements of Thms.~\ref{thm:general-exforalldata} and \ref{thm:general-wilddata} in the context of the hydrostatic Euler equations \eqref{eq:hydrost-euler-mass}-\eqref{eq:hydrost-euler-p}.

\begin{thm} \label{thm:hydrost-euler-exforalldata} 
	Let $\Omega= \T^2\times (0,1)$ and $T\in (0,\infty)\cup \{\infty\}$. Moreover let\footnote{As explained in Rem.~\ref{rem:hydrost-euler-initialdata}, we work with initial data for the horizontal \emph{and} the vertical velocity. One may understand this $w_0$ as being recovered from $u_0$. We would like to emphasize that we not only require that the initial data for the horizontal velocity $u_0\in L^\infty(\Omega;\R^2)$ but also the initial data for the vertical velocity $w_0 \in L^\infty(\Omega)$. Furthermore, we underline that the solutions constructed in Thm.~\ref{thm:hydrost-euler-exforalldata} not only satisfy $u\in L^\infty((0,T)\times\Omega;\R^2)$ but also $w\in L^\infty((0,T)\times\Omega)$.} $(u_0,w_0)\in L^\infty(\Omega;\R^3)$ satisfying the compatibility condition \eqref{eq:hydrost-euler-compatibility-initialdata}. Then there exist infinitely many weak solutions $(u,w,p)$ of the hydrostatic Euler equations \eqref{eq:hydrost-euler-mass}-\eqref{eq:hydrost-euler-p} in the sense of Defn.~\ref{defn:hydrost-euler-sol}~\ref{item:hydrost-euler-sol-weak}, all of which satisfy $(u,w)\in \Cweak([0,T);L^2(\Omega;\R^3))$. 
\end{thm} 

\begin{thm} \label{thm:hydrost-euler-wilddata} 
	Let $\Omega= \T^2\times (0,1)$ and $T\in (0,\infty)\cup \{\infty\}$. There exist initial data $(u_0,w_0)\in L^\infty(\Omega;\R^3)$ satisfying the compatibility condition \eqref{eq:hydrost-euler-compatibility-initialdata}, for which there are infinitely many admissible weak solutions $(u,w,p)$ of the hydrostatic Euler equations \eqref{eq:hydrost-euler-mass}-\eqref{eq:hydrost-euler-p} in the sense of Defn.~\ref{defn:hydrost-euler-sol}~\ref{item:hydrost-euler-sol-adm}. All of these solutions satisfy $(u,w)\in \Cweak([0,T);L^2(\Omega;\R^3))$. 
\end{thm} 

Thm.~\ref{thm:hydrost-euler-exforalldata} can by viewed as an improvement of \name{Chiodaroli}-\name{Mich{\'a}lek}~\cite[Thm.~2.3]{ChiMic17} who consider only continuous initial data.

\begin{rem}
    Thm.~\ref{thm:hydrost-euler-exforalldata} still holds if $\Omega=\T^3$. 
\end{rem}

\begin{rem} \label{rem:hydrost-euler-differences-ChiMic}
    We note that Thm.~\ref{thm:hydrost-euler-wilddata} is the first construction of admissible weak solutions for the hydrostatic Euler equations (i.e., which uses the 2D instead of the 3D kinetic energy). In particular, our results are fundamentally different from the work by \name{Chiodaroli}-\name{Mich{\'a}lek} \cite{ChiMic17} who treat the hydrostatic Euler equations by constructing solutions to the 3D Euler equations where the pressure does not depend on $z$. In particular, they work with the 3D kinetic energy rather than the horizontal one. But we observe once again that this is not the right energy to be considered for the hydrostatic Euler equations, see also Rem.~\ref{rem:hydrost-euler-energy}. We emphasize that using the correct energy fundamentally changes the convex integration procedure which results in essential differences between our approach and the one used in \cite{ChiMic17}, see Sect.~\ref{sec:ex-hydrost-euler} for more details.
\end{rem}

\begin{rem} \label{rem:hydrost-euler-decreasing-energy}
    We note that infinitely many of the solutions constructed in Thm.~\ref{thm:hydrost-euler-wilddata} even conserve the horizontal kinetic energy, i.e., \eqref{eq:hydrost-euler-enineq} holds with equality (in the sense of distributions). Moreover -- even though the vertical kinetic energy is not formally conserved, see Rem.~\ref{rem:hydrost-euler-energy} -- the constructed solutions also conserve the local vertical kinetic energy, i.e., 
    $$
        \partial_t \left(\half w^2\right) + \Divh \left[\left(\half w^2 + p\right)u\right] + \partial_z \left[\left(\half w^2 + p\right)w\right] = 0
    $$
    in the sense of distributions. Consequently they conserve the 3D kinetic energy $\half (|u|^2 + w^2)$.

    A slight modification of the proof of Thm.~\ref{thm:hydrost-euler-wilddata} yields initial data for which there are infinitely many admissible weak solutions with strictly decreasing horizontal kinetic energy and strictly decreasing vertical kinetic energy. Even the combinations ``strictly decreasing horizontal energy / conserved vertical energy'' and ``conserved horizontal energy / strictly decreasing vertical energy'' are possible. 
\end{rem} 

\begin{rem}
    Thm.~\ref{thm:hydrost-euler-wilddata} also holds in the case of $\Omega=\T^3$, and $\Omega\subset \R^3$ bounded with impermeability boundary condition, see case~\ref{item:domain-bdd} in Sect.~\ref{subsubsec:intro-incomp-euler}.
\end{rem}

\subsubsection{Compressible inviscid primitive equations} \label{subsubsec:intro-compr-prim}

A compressible version of the hydrostatic Euler equations \eqref{eq:hydrost-euler-mass}-\eqref{eq:hydrost-euler-p} is given by the compressible inviscid primitive equations 
\begin{align}
	\partial_t \rho + \Divh (\rho u) + \partial_z (\rho w) &= 0, \label{eq:compr-prim-mass} \\
	\partial_t (\rho u) + \Divh (\rho u\otimes u) + \partial_z (\rho uw) + \Gradh p(\rho) &=0, \label{eq:compr-prim-mom} \\
	\partial_z p(\rho) &=0. \label{eq:compr-prim-p} 
\end{align}
Here the unknowns are the density $\rho$, the horizontal velocity $u$ and the vertical velocity $w$, which take values in\footnote{Like for the consideration of the compressible Euler equations \eqref{eq:compr-euler-mass}, \eqref{eq:compr-euler-mom}, we exclude vacuum in this paper, i.e.~we require that $\rho$ is strictly positive.} $\R^+$, $\R^2$ and $\R$, respectively. Like for the compressible Euler equations (see Sect.~\ref{subsubsec:intro-compr-euler}), $p\in C^1((0,\infty);\R^+)$ is a given function of $\rho$. We consider the same domain as in the incompressible case, i.e., the channel $\Omega=\T^2 \times (0,1)$ with boundary condition $w|_{z=0,1} = 0$. The compressible primitive equations are a fundamental model for atmospheric dynamics on a planetary scale. We note that we are considering the case without gravity. We also emphasize that in this work we will not use pressure coordinates ($p$-coordinates) but rather work with physical cartesian coordinates.

Moreover, we are interested in the initial value problem with initial data $(\rho_0,u_0,w_0)\in L^\infty(\Omega;\R^+\times \R^2\times \R)$. In this case, $\rho_0$ must satisfy 
\begin{equation} \label{eq:compr-prim-compatibility-initialdata}
    \rho_0(x)= \rho_0(x_1,x_2), \quad \text{i.e., }\rho_0\text{ is independent of }z, 
\end{equation}
in order to be compatible with the hydrostatic balance equation \eqref{eq:compr-prim-p}. 

The local energy inequality reads
$$
    \partial_t\left( \half \rho |u|^2 + P(\rho) \right) + \Divh \left[ \left( \half \rho |u|^2 + P(\rho) + p(\rho) \right) u \right] + \partial_z \left[ \left( \half \rho |u|^2 + P(\rho) + p(\rho) \right) w \right] \leq 0
$$
where as in the compressible Euler equations the pressure potential $P$ is given by \eqref{eq:pressure-potential}.

Like in the incompressible case, the viscous compressible primitive equations were first studied in \cite{LioTemWan92} using pressure coordinates. The global existence of weak solutions in two dimensions was established in \cite{gatapov2005,ersoy2012}, while the stability of weak solutions was considered in \cite{ErsNgoSy11,tang2015}. Sufficient conditions for the energy equality to hold for weak solutions were identified in \cite{necasova2023}. 

The local well-posedness of the viscous compressible primitive equations was established in \cite{liu2021} for the cases of gravity without vacuum and vacuum without gravity. The vanishing Mach number limit towards the incompressible viscous primitive equations was rigorously justified in \cite{liu2020,liu2023}, while the inviscid case was considered in \cite{gao2019}. In \cite{gao2021,gao2022}, under the assumption that the density of the solution remains independent of $z$, i.e., in the absence of fast vertical acoustic waves, the hydrostatic limit was established (see also \cite{TanNec23} for results on the inviscid case) as a vanishing asymptotic limit of small aspect ratio between vertical and horizontal spatial scales. In \cite{liu2024a} the hydrostatic limit was justified rigorously, with rate of convergence in terms of the small aspect ratio, without this assumption, i.e., allowing for fast vertical acoustic waves which induce, in the limit, an averaging mechanism of the density in the vertical direction. Finally, the asymptotic stability of the equilibrium for the free boundary problem of the viscous equations was established in \cite{liu2024b}. 

The notions of weak and admissible weak solutions in the context of the compressible inviscid primitive equations used in this paper are defined in the same way as in the previous subsections, see Defns.~\ref{defn:incomp-euler-sol}, \ref{defn:compr-euler-sol} and \ref{defn:hydrost-euler-sol}. 

Similar to the compressible Euler equations, we will not prove Thm.~\ref{thm:general-exforalldata} in the context of the compressible inviscid primitive equations \eqref{eq:compr-prim-mass}-\eqref{eq:compr-prim-p}, see Rem.~\ref{rem:compr-models-no-exforalldata} for the reason. 

Let us next state Thm.~\ref{thm:general-wilddata}, for this model, precisely.

\begin{thm} \label{thm:compr-prim-wilddata} 
    Let $\Omega= \T^2\times (0,1)$ and $T\in (0,\infty)\cup \{\infty\}$. Furthermore, let $p\in C^1((0,\infty);\R^+)$, and $0<\un{\rho}<\ov{\rho}<\infty$. Then there exists $\ep>0$, depending only on $\un{\rho}$, $\ov{\rho}$ and the pressure function $p$, with the following property. For any initial density $\rho_0\in C^1(\closure{\Omega};\R^+)$ which satisfies the compatibility condition \eqref{eq:compr-prim-compatibility-initialdata} and 
    $$
        \un{\rho}\leq \rho_0(x) \leq \ov{\rho} \quad \text{ for all }x\in \closure{\Omega}\qquad \text{ and } \qquad\sup_{x\in\closure{\Omega}}|\Grad \rho_0(x)|\leq \ep, %
    $$
    there exists an initial velocity $(u_0,w_0)\in L^\infty(\Omega;\R^3)$, for which there are infinitely many admissible weak solutions $(\rho,u,w)$ of the compressible inviscid primitive equations \eqref{eq:compr-prim-mass}-\eqref{eq:compr-prim-p}. All of these solutions satisfy $\rho\in C^1([0,T)\times \closure{\Omega};\R^+)$ 
    with 
    $$
        \un{\rho}\leq \rho(t,x) \leq \ov{\rho} \quad \text{ for all }(t,x)\in [0,T)\times \closure{\Omega}, 
    $$ 
    and $(u,w)\in \Cweak([0,T);L^2(\Omega;\R^3))$. 
\end{thm}

Rems.~\ref{rem:compr-euler-wilddata-DelSze}-\ref{rem:compr-euler-wilddata-dependencies} hold in a similar fashion for the compressible inviscid primitive equations. Moreover, an analogous version of Rem.~\ref{rem:hydrost-euler-decreasing-energy} is true, more precisely:

\begin{rem}
    One can even achieve that the solutions established in Thm.~\ref{thm:compr-prim-wilddata} satisfy an energy inequality for the 3D energy $\half \rho (|u|^2 + w^2) +P(\rho)$. 
\end{rem}

To the best of the authors' knowledge, no results similar to Thm.~\ref{thm:compr-prim-wilddata} (in the context of the compressible inviscid primitive equations \eqref{eq:compr-prim-mass}-\eqref{eq:compr-prim-p}) are known so far.

\subsubsection{Quasi-geostrophic equations} \label{subsubsec:intro-QG}

Finally, we will apply our general convex integration framework to the quasi-geostrophic equations, which model a fluid that is close to the so-called geostrophic balance, see, e.g., \cite{Pedlosky}. The quasi-geostrophic equations read in their original form
\begin{align}
    \partial_t (\Lap \Psi) + (\Gradh \Psi)^\perp \cdot \Gradh (\Lap \Psi) &= 0, \qquad (t,x,y,z)\in [0,T) \times \T^2\times [0,1], \label{eq:QG-orig-1} \\
    \partial_t (\partial_\nu \Psi) + (\Gradh \Psi)^\perp \cdot \Gradh (\partial_\nu \Psi) &= 0, \qquad (t,x,y,z)\in [0,T) \times \T^2\times \{0,1\}, \label{eq:QG-orig-2}
\end{align}
see, e.g., \cite{dutton1974}. The unknown in the quasi-geostrophic equations \eqref{eq:QG-orig-1}, \eqref{eq:QG-orig-2} is the stream function $\Psi$, which takes values in $\R$. The velocity field is given by $(\Gradh \Psi)^\perp$ and the potential vorticity by $\Lap \Psi$. We use the notation $\QGunknown_h$ for the horizontal part of a three-dimensional vector $\QGunknown$, i.e.~$\QGunknown_h = ([\QGunknown]_1 , [\QGunknown]_2 )^\trans\in \R^2$, and $\QGunknown_v = [\QGunknown]_3\in \R$ for the vertical part. Furthermore, we denote $v^\perp = (-[v]_2 ,[v]_1)^\trans$ for any $v\in \R^2$. In addition, $\partial_\nu\Psi$ denotes the outward pointing normal derivative of $\Psi$ at $z=0,1$.

In this paper we consider the following ``reformulated'' version of the quasi-geostrophic equations, which was introduced in \cite{PueVas15} (see also \cite{Novack20}) and reads\footnote{Note that \eqref{eq:QG-vassuer-mom} is indeed a system of three equations. The fact that the terms $\partial_t (\Grad \Psi)$ and $\Curl Q$ have three components is obvious. Moreover, the product $\Grad \Psi \otimes (\Gradh \Psi)^\perp$ can be understood as a matrix consisting of three rows and two columns. Hence its row-wise horizontal divergence $\Divh$ is well-defined and takes values in $\R^3$ as desired.} 
\begin{align}
    \partial_t (\Grad \Psi) + \Divh ( \Grad \Psi \otimes (\Gradh \Psi)^\perp ) - \Curl Q &= 0. \label{eq:QG-vassuer-mom}
\end{align} 
Equation \eqref{eq:QG-vassuer-mom} is formally closed by the elliptic equation 
\begin{equation} \label{eq:QG-recover-Q}
    - \Delta Q = \Curl ( (\nabla_h \Psi)^\perp \cdot \nabla_h (\nabla \Psi)). 
\end{equation}
The spatial domain for \eqref{eq:QG-vassuer-mom} is the same as in as in Sects.~\ref{subsubsec:intro-hydrost-euler} and \ref{subsubsec:intro-compr-prim}, i.e., $\Omega = \T^2 \times (0,1)$. The boundary condition can then be expressed by 
\begin{equation} \label{eq:BC-QG}
    \Curl Q|_{z=0,1} \cdot (0,0,1)=0.
\end{equation}

The mathematical analysis of the quasi-geostrophic equations \eqref{eq:QG-orig-1}, \eqref{eq:QG-orig-2} was initiated in \cite{dutton1974,dutton1976}. These equations were derived rigorously in \cite{bourgeois1994} in a bounded domain for initial data which are close to the geostrophic balance, the so-called well-prepared initial data. The case of general initial data (which lead to resonance terms) was studied in \cite{embid1996,embid1998,majda1998} in the case of periodic boundary conditions, and in \cite{bardos2024} for general initial data in the presence of physical boundaries. The global existence of weak solutions of the inviscid quasi-geostrophic equations was established in \cite{PueVas15}, using the aforementioned reformulation of the system. This result was extended in \cite{novack2019} to a larger class of initial data. 

The global existence of weak solutions (with bounded potential vorticity) for the quasi-geostrophic equations with physical lateral boundaries was established in \cite{novack2020a}, while the local well-posedness was proved in \cite{novack2020b}. The non-uniqueness of weak solutions was established in \cite{Novack20} for Hölder exponents less than $\frac{1}{5}$ (for the velocity field), in the presence of physical boundaries in the vertical direction. Note that for the constructed solutions the potential vorticity is not necessarily defined (unlike for the previous results). A sufficient condition for the conservation of energy was found in \cite{novack2019}. 

In addition, in \cite{novack2018} the global well-posedness was established for the quasi-geostrophic equations with Ekman pumping (which leads to a viscous dissipation term on the boundary). The global existence of weak solutions in the presence of Ekman pumping with a degenerate density background profile was achieved in \cite{hu2024}. The rigorous derivation of the viscous equations was carried out in \cite{desjardins1998} (see also \cite{niu2010}). 

In this paper we will consider $\QGunknown:= \Grad \Psi$ instead of $\Psi$ itself as the unknown. Consequently, we deal with the following equations
\begin{align} 
    \Curl \QGunknown &=0 , \label{eq:QG-curl} \\
    \partial_t \QGunknown + \Divh ( \QGunknown \otimes (\QGunknown_h)^\perp ) - \Curl Q &= 0. \label{eq:QG-mom}
\end{align}
Note that \eqref{eq:QG-curl} is equivalent to the existence of $\Psi$ with $\QGunknown = \Grad \Psi$, see \cite[Thm.~1]{AmrCiaCia07} which is used in Cor.~\ref{cor:QG}, below. Moreover, we would like to stress that $\QGunknown$ does not play the role of a velocity\footnote{In particular, the velocity in this case is given by $(\Grad_h\Psi)^\perp$.}. 

In the context of system \eqref{eq:QG-curl}, \eqref{eq:QG-mom}, we treat $Q$ as an additional unknown rather than recovered from $\QGunknown$ via \eqref{eq:QG-recover-Q} (where $\Psi$ has to be replaced by $\QGunknown$ through $\QGunknown:= \Grad \Psi$). This is in the same spirit as we treated the pressure $p$ in the context of the incompressible Euler equations, see Sect.~\ref{subsubsec:intro-incomp-euler}, above.

We supplement the system \eqref{eq:QG-curl}, \eqref{eq:QG-mom} by initial data $\QGunknown_0\in L^\infty(\Omega;\R^3)$ which we require to fulfill
\begin{equation} \label{eq:QG-compatibility-initialdata}
    \int_\Omega \QGunknown_0 \cdot \Curl \phi \dx = 0 \qquad \text{ for all test functions }\phi\in \Cc(\Omega;\R^3),
\end{equation}
in order to be compatible with \eqref{eq:QG-curl}. The boundary condition for system \eqref{eq:QG-curl}, \eqref{eq:QG-mom} is given by \eqref{eq:BC-QG}.

The local energy inequality for the quasi-geostrophic equations in the form \eqref{eq:QG-curl}, \eqref{eq:QG-mom} is given by 
\begin{equation} \label{eq:QG-enineq} 
    \partial_t \left( \half |\QGunknown|^2\right) + \Divh \left( \half |\QGunknown|^2 (\QGunknown_h)^\perp \right) - \Div(Q\times \QGunknown ) \leq 0.
\end{equation}

In this paper we work with the following definition of weak solutions and admissible weak solutions of the quasi-geostrophic equations \eqref{eq:QG-curl}, \eqref{eq:QG-mom}. 

\begin{defn} \label{defn:QG-sol} 
	\begin{enumerate}
		\item \label{item:QG-sol-weak} A pair $(\QGunknown,Q)$ with $\QGunknown\in L^\infty((0,T)\times \Omega; \R^3)$ and $Q\in C^1([0,T]\times \closure{\Omega}; \R^3)$ is called a \emph{weak solution} of the quasi-geostrophic equations \eqref{eq:QG-curl}, \eqref{eq:QG-mom} with initial data $\QGunknown_0\in L^\infty(\Omega;\R^3)$ (which satisfy the compatibility condition \eqref{eq:QG-compatibility-initialdata}) if the equations are satisfied in the sense of distributions, i.e., 
		\begin{align*} 
			\int_0^T \int_\Omega \QGunknown \cdot \Curl \phi \dx \dt &= 0, \\
			\int_0^T \int_\Omega \Big[ \QGunknown \cdot \partial_t \varphi + (\QGunknown\otimes (\QGunknown_h)^\perp) : \Gradh \varphi + \Curl Q \cdot \varphi \Big] \dx\dt + \int_\Omega \QGunknown_0 \cdot \varphi(0,\cdot) \dx &= 0, 
		\end{align*}
		for all test functions $\phi\in \Cc ([0,T)\times \Omega; \R^3)$, $\varphi\in \Cc ([0,T)\times \closure{\Omega};\R^3)$, and the boundary condition \eqref{eq:BC-QG} holds. 
		
		\item \label{item:QG-sol-adm} A weak solution $(\QGunknown,Q)$ of the quasi-geostrophic equations \eqref{eq:QG-curl}, \eqref{eq:QG-mom} with initial data $\QGunknown_0$ is called \emph{admissible} if the local energy inequality \eqref{eq:QG-enineq} holds in the sense of distributions, i.e., 
		\begin{align*} 
			\int_0^T \int_\Omega \Big[ \half |\QGunknown|^2 \partial_t \Phi + \half |\QGunknown|^2 (\QGunknown_h)^\perp \cdot \Gradh \Phi - (Q\times \QGunknown) \cdot \Grad \Phi \Big] \dx \dt + \int_\Omega \half |\QGunknown_0|^2 \Phi (0,\cdot) \dx &\geq  0, 
		\end{align*}
		for all test functions $\Phi\in \Cc ([0,T)\times \Omega; \R^+_0)$. 
	\end{enumerate}
\end{defn}

\begin{rem} \label{rem:QG-Q} 
    We remark that our requirement $Q\in C^1$ is very strong. However, for the solutions that are constructed in this paper (see Thm.~\ref{thm:QG-wilddata}, below), $Q$ will satisfy such high regularity. A major reason why we will not consider the case with lower regularity for $Q$ is that a detailed discussion of the attainment of the boundary condition \eqref{eq:BC-QG} in the sense of trace at low regularity of $Q$ is outside the scope of this paper. 
\end{rem}

In the context of the quasi-geostrophic equations \eqref{eq:QG-curl}, \eqref{eq:QG-mom}, we will not prove Thm.~\ref{thm:general-exforalldata}. If we wanted to show it, we would need a version of Prop.~\ref{prop:app-sfad}, below, which is adjusted to the quasi-geostrophic equations. As in Rem.~\ref{rem:incomp-euler-exforalldata-otherdomains} such a version of Prop.~\ref{prop:app-sfad} might be feasible, but in order to keep the presentation simple and concise, we leave this for future work.

For the quasi-geostrophic equations \eqref{eq:QG-curl}, \eqref{eq:QG-mom}, Thm.~\ref{thm:general-wilddata} can be precisely stated as follows.

\begin{thm} \label{thm:QG-wilddata}
    Let $\Omega= \T^2\times (0,1)$ and $T\in (0,\infty)\cup \{\infty\}$. There exist initial data $\QGunknown_0\in L^\infty(\Omega;\R^3)$ satisfying the compatibility condition \eqref{eq:QG-compatibility-initialdata}, for which there are infinitely many admissible weak solutions $(\QGunknown,Q)$ of the quasi-geostrophic equations \eqref{eq:QG-curl}, \eqref{eq:QG-mom} in the sense of Defn.~\ref{defn:QG-sol}~\ref{item:QG-sol-adm}. All of these solutions satisfy $\QGunknown\in \Cweak([0,T);L^2(\Omega;\R^3))$. 
\end{thm} 

To the best of the authors' knowledge, no results similar to Thm.~\ref{thm:QG-wilddata} (in the context of the quasi-geostrophic equations \eqref{eq:QG-curl}, \eqref{eq:QG-mom}) are known so far. 

Let us finally rewrite the statement of Thm.~\ref{thm:QG-wilddata} in terms of the stream function $\Psi$:
\begin{cor} \label{cor:QG}
    Let $\Omega= \T^2\times (0,1)$ and $T\in (0,\infty)\cup \{\infty\}$. There exist initial data $\Psi_0\in H^1(\Omega;\R^3)$, for which there are infinitely many weak solutions $(\Psi,Q)$ of the quasi-geostrophic equations \eqref{eq:QG-vassuer-mom} in the sense of \cite{PueVas15} (see Sect.~2 therein). All of these solutions satisfy the local energy inequality \eqref{eq:QG-enineq} (with $\QGunknown:= \Grad \Psi)$ in the sense of distributions, as well as $\Psi\in \Cweak([0,T);H^1(\Omega;\R^3))$. 
\end{cor}

Cor.~\ref{cor:QG} immediately follows from Thm.~\ref{thm:QG-wilddata} by using \cite[Thm.~1]{AmrCiaCia07}. 

\begin{rem} 
    As shown in \cite[Thm.~2.1]{PueVas15}, the reformulated quasi-geostrophic equations \eqref{eq:QG-vassuer-mom} are equivalent to the original quasi-geostrophic equations \eqref{eq:QG-orig-1}, \eqref{eq:QG-orig-2} as long as 
    \begin{equation} \label{eq:QG-lapl2}
        \Lap\Psi\in L^\infty((0,T);L^2(\Omega)).
    \end{equation}
    It should be emphasized that the solutions in Cor.~\ref{cor:QG} do not necessarily satisfy \eqref{eq:QG-lapl2}. Consequently, they will not necessarily solve the original quasi-geostrophic equations \eqref{eq:QG-orig-1}, \eqref{eq:QG-orig-2}. Note that the same holds for the solutions constructed in \cite{Novack20}.
\end{rem}

\subsection{Organization of this paper} 

This paper is organized as follows. In Sect.~\ref{sec:general} the general framework for convex integration is elaborated. This is the main part of the paper. The framework itself is stated in Thm.~\ref{thm:conv-int}. 

In Sects.~\ref{sec:ex-incomp-euler}-\ref{sec:ex-QG} the general convex integration framework (i.e., Thm.~\ref{thm:conv-int}) is applied to the PDE systems listed in Table~\ref{tab:main-thms} and specified in Sects.~\ref{subsubsec:intro-incomp-euler}-\ref{subsubsec:intro-QG}. The goal in Sects.~\ref{sec:ex-incomp-euler}-\ref{sec:ex-QG} is to prove the theorems stated in Sects.~\ref{subsubsec:intro-incomp-euler}-\ref{subsubsec:intro-QG}. Each of the Sects.~\ref{sec:ex-incomp-euler}-\ref{sec:ex-QG} contains
\begin{itemize}
    \item a subsection where the problem is reformulated such that the framework which is established in Sect.~\ref{sec:general} can be applied (see Sects.~\ref{subsec:ex-incomp-euler-prel}, \ref{subsec:ex-compr-euler-prel}, \ref{subsec:ex-lake-prel}, \ref{subsec:ex-hydrost-euler-prel}, \ref{subsec:ex-compr-prim-prel}, \ref{subsec:ex-QG-prel} and also Rem.~\ref{rem:tartars-framework}), 
    
    \item a subsection which proves that the ``structural assumptions'' of Thm.~\ref{thm:conv-int} hold (see Sects.~\ref{subsec:ex-incomp-euler-ass}, \ref{subsec:ex-compr-euler-ass}, \ref{subsec:ex-lake-ass}, \ref{subsec:ex-hydrost-euler-ass}, \ref{subsec:ex-compr-prim-ass} and \ref{subsec:ex-QG-ass}), and 

    \item a subsection where the proofs of the theorems stated in Sects.~\ref{subsubsec:intro-incomp-euler}-\ref{subsubsec:intro-QG} are completed (see Sects.~\ref{subsec:ex-incomp-euler-pf}, \ref{subsec:ex-compr-euler-pf}, \ref{subsec:ex-lake-pf}, \ref{subsec:ex-hydrost-euler-pf}, \ref{subsec:ex-compr-prim-pf} and \ref{subsec:ex-QG-pf}). 
\end{itemize}

Appendix~\ref{app:linear-algebra} contains a lemma on the largest eigenvalues of a symmetric $2\times 2$ matrix. Some facts regarding convex and $\Lambda$-convex hulls are collected in appendices~\ref{app:convex}-\ref{app:Lconvex}. In Appendix~\ref{app:sfad} a suitable subsolution for any initial data in $L^\infty$ is constructed in the context of the incompressible Euler equations.

\section{Generalized framework for convex integration} \label{sec:general}

As indicated in Sect.~\ref{subsec:intro-ci}, the convex integration framework developed in the present paper may be viewed as complementary to the framework that was presented by the second author in \cite[Sect.~2]{Markfelder24}. We refer to Sect.~\ref{subsec:intro-ci} for a detailed discussion of the differences between the framework established in the present paper and the one developed in \cite{Markfelder24}. However, as some of the individual steps of both approaches are similar, for some parts we will follow \cite{Markfelder24}. 

Sects.~\ref{subsec:general-prelim} and \ref{subsec:general-geometric} contain the preliminaries (formulation of the problem and some definitions) and the geometric setup, respectively. Since the formulation of the problem as well as the geometric setup do not depend on whether the desired solutions are weakly continuous in time or just bounded, Sects.~\ref{subsec:general-prelim} and \ref{subsec:general-geometric} are only a slight modification of the corresponding sections in \cite{Markfelder24}. 

In contrast to that, in Sects.~\ref{subsec:general-functional} and \ref{subsec:general-pert-prop} it becomes essential that we are constructing solutions that are weakly continuous in time instead of merely bounded, so our Sects.~\ref{subsec:general-functional} and \ref{subsec:general-pert-prop} differ significantly from the corresponding sections in \cite{Markfelder24}. Sect.~\ref{subsec:general-functional} contains the functional setup, the main theorem of the framework (Thm.~\ref{thm:conv-int}) and its proof using Prop.~\ref{prop:pert-prop}, which is called \emph{perturbation property}. The perturbation property, which may be viewed as the heart of convex integration, is proved in Sect.~\ref{subsec:general-pert-prop}.

Sect.~\ref{subsec:general-idiK} provides a tool to modify subsolutions at some time positive $T_0$. This tool allows to construct solutions without initial energy jump. 

As mentioned in Sect.~\ref{subsec:intro-ci}, solutions that are weakly continuous in time were already constructed in the context of the incompressible Euler equations in \cite{DelSze10}. Consequently, one may say that our framework generalizes \cite{DelSze10} to a large family of first order PDEs. Indeed, we will make use of many ideas which were originally introduced in \cite{DelSze10}. 

Readers who are unfamiliar with the method of convex integration, might want to consult \cite[Sect.~4.1]{Markfelder}, where the the essential ideas behind convex integration are sketched.

\subsection{Preliminaries} \label{subsec:general-prelim}

\subsubsection{The linear PDE-system and the family of constitutive sets $(\sK_\mu)_{\mu\in \Theta}$} \label{subsubsec:gen-prelim-pde+K}

Let us first introduce some notation. Let $T\in (0,\infty)\cup\{\infty\}$, and $\Omega\subset\R^n$ be a bounded\footnote{See Rem.~\ref{rem:unbounded-domains} for a way how to generalize our framework to unbounded domains $\Omega$.} Lipschitz domain, where $n\in \N$ is the space dimension. The unknown is a vector-valued function\footnote{By $[0,T]$ we mean the closure of the open interval $(0,T)$. In particular $[0,T]$ is the classical closed interval from $0$ to $T$ if $T<\infty$, whereas $[0,T]=[0,\infty)$ for $T=\infty$.} $\zeta: [0,T]\times \Omega \to \R^M$, where $M\in \N$ is the ``number of scalar unknowns''. Let furthermore $a,b_1,b_2:\R^M \to \R^m$ and $A,B: \R^M \to \R^{m\times n}$ be linear maps, where $m\in \N$ is the ``number of scalar equations''. Finally, let\footnote{As commonly used in the literature, $\Cb$ denotes the set of all continuous and bounded functions.} $\theta\in \Cb([0,T]\times\closure{\Omega};\R^m)\cap C^1([0,T]\times\closure{\Omega};\R^m)$ and set $\Theta:= \closure{\theta([0,T]\times\closure{\Omega})}$. Note that $\Theta$ is a compact subset of $\R^m$. 

We study a general linear, first order and not necessarily homogeneous system of PDEs
\begin{equation} \label{eq:lin-eq}
	\partial_t a(\zeta(t,x)) + \Div A (\zeta(t,x)) = \partial_t b_1(\theta(t,x)) + \Div B(\theta(t,x)) + b_2(\theta(t,x)), 
\end{equation} 
supplemented with some constitutive laws which we collect in a family of sets, the so-called \emph{constitutive sets} $\sK_{\mu}\subset \R^M$, where $\mu\in \Theta$, see below. In this paper, we will use the convention that the divergence of a matrix is meant to act row-wise, i.e., 
$$
	[\Div A (\zeta)]_i = \sum_{j=1}^n \partial_j [A(\zeta)]_{ij},
$$
where $\partial_j:= \frac{\partial}{\partial x_j}$ for $j=1,...,n$, and analogously for $\Div B$.

Note that the linearity of $\zeta\mapsto a(\zeta)$ and $\zeta\mapsto A(\zeta)$ implies 
\begin{align}
    \partial_t a(\zeta) &= a (\partial_t \zeta), \label{eq:linearity-a} \\
	[\Div A (\zeta)]_i &= \sum_{j=1}^n \partial_j [A(\zeta)]_{ij} = \sum_{j=1}^n [A(\partial_j\zeta)]_{ij} \qquad \text{ for all }i=1,...,m . \label{eq:linearity-A}
\end{align}
Similar identities hold for $b_1$ and $B$.

As mentioned above, in addition to the linear system \eqref{eq:lin-eq}, we consider for each $\mu\in \Theta$ the constitutive set $\sK_{\mu}\subset \R^M$. Moreover, let $r\in (1,\infty)$. Our goal is to find weak solutions 
$$
    \zeta\in \Cweak([0,T];L^r(\Omega;\R^M)) \cap L^\infty((0,T)\times \Omega;\R^M)
$$ 
of the linear system \eqref{eq:lin-eq} which satisfy 
\begin{equation} \label{eq:zetainK} 
	\zeta(t,x)\in \sK_{\theta(t,x)}\qquad \text{ for all }t\in (0,T)\text{ and a.e. }x\in \Omega. 
\end{equation}

\begin{rem} \label{rem:inequality-possible} 
    We would like to note that it is also possible to study a system of partial differential \emph{inequalities} instead of the system of equations \eqref{eq:lin-eq} as done in \cite[Sect.~2]{Markfelder24}. All proofs and statements in the sequel will still hold with obvious modifications. 
\end{rem}

\begin{rem}
    Note that, heuristically, as \eqref{eq:lin-eq} is a system of $m$ equations for $M$ unknowns, there should hold $m\leq M$. In the case $m<M$ (i.e., the linear system \eqref{eq:lin-eq} is underdetermined), the constitutive laws in $(\sK_\mu)_{\mu\in \Theta}$ formally close the system. Note, however, that our main result (see Thm.~\ref{thm:conv-int}, below) states existence of (infinitely many) solutions under the assumption that a so-called subsolution exists. Since in the overdetermined case (i.e., $m>M$) subsolutions cannot exist, our main statement (i.e., Thm.~\ref{thm:conv-int}, below) becomes void. For this reason, there is no need to insist on $m\leq M$. In other words, we do not need to require any relation between $M$ and $m$. 
\end{rem}

\begin{rem} \label{rem:tartars-framework} 
    In usual applications of our framework, a non-linear system of PDEs is rewritten by replacing all non-linearities by new unknowns. This leads to an (underdetermined) linear system together with a family of constitutive sets. In other words, the non-linear system of PDEs is expressed in the form \eqref{eq:lin-eq}, \eqref{eq:zetainK}, which makes our framework applicable in this case. This procedure is known as \emph{Tartar's framework}, see \cite{Tartar79}. In Sects.~\ref{subsec:ex-incomp-euler-prel}, \ref{subsec:ex-compr-euler-prel}, \ref{subsec:ex-lake-prel}, \ref{subsec:ex-hydrost-euler-prel}, \ref{subsec:ex-compr-prim-prel} and \ref{subsec:ex-QG-prel} Tartar's framework is illustrated by examples. 
\end{rem}

\begin{rem}
    Let us repeat that the main difference between the framework established in the current paper and the one in \cite{Markfelder24}, is that the weak solutions $\zeta$ which were constructed in \cite{Markfelder24} only lie in $L^\infty((0,T)\times \Omega;\R^M)$ and they satisfy \eqref{eq:zetainK} only for a.e.~$t\in (0,T)$ and a.e.~$x\in \Omega$.
\end{rem}

Let us now recall the definition of suitability for a family of constitutive sets from \cite{Markfelder24}.

\begin{defn}[{See \cite[Defn.~2.3]{Markfelder24}}] \label{defn:suitable-K}
	We call the family of constitutive sets $(\sK_\mu)_{\mu\in\Theta}$ \emph{suitable} if it satisfies the following three properties:
	\begin{enumerate} 
		\item \label{item:suitable-K-compact} For all $\mu\in \Theta$, $\sK_\mu$ is compact in $\R^M$.

		\item \label{item:suitable-K-continuity} The dependence of $\sK_{\mu}$ on $\mu$ is uniformly continuous in the following sense: For all $\ep>0$, there exists $\delta>0$ such that for all $\mu_1,\mu_2\in \Theta$ with $|\mu_1-\mu_2|<\delta$ and all $\zeta_1\in \sK_{\mu_1}$ there exists $\zeta_2\in \sK_{\mu_2}$ with $|\zeta_1-\zeta_2|<\ep$.
		
		\item \label{item:suitable-K-boundary} We have $\interior{\big((\sK_{\mu})^\co\big)} \cap \sK_{\mu} = \emptyset$ for all $\mu\in \Theta$, where $(\sK_{\mu})^\co$ denotes the convex hull of $\sK_{\mu}$, see Appendix~\ref{app:convex}. 
	\end{enumerate}
\end{defn} 

The reader should notice that in some places (in particular in Defn.~\ref{defn:suitable-K}, Lemmas~\ref{lemma:uniform-boundedness-KU} and \ref{lemma:properties-dist}, as well as in Sect.~\ref{subsec:general-geometric}), $\zeta$ denotes a point in $\R^M$ while elsewhere $\zeta$ denotes the unknown (i.e.~a function of $(t,x)$) as specified above. Whenever $\zeta$ is understood as a point in $\R^M$, this is explicitly stated, e.g.~by $\zeta\in \sK_\mu\subset \R^M$.

Next we recall some observations from \cite{Markfelder24}, whose proofs can be found therein.

\begin{lemma}[{Cf.~\cite[Lemmas~2.4 and 2.9]{Markfelder24}}] \label{lemma:uniform-boundedness-KU} 
	Let $(\sK_\mu)_{\mu\in\Theta}$ a suitable family of constitutive sets. Then there exists\footnote{In particular $c$ is does not depend on $\mu$.} $c>0$ such that 
	$$
		|\zeta|\leq c \qquad \text{ for all }\mu\in \Theta \text{ and all } \zeta\in \sK_{\mu} .
	$$
    Moreover we have 
    $$
        |\zeta|\leq c \qquad \text{ for all }\mu\in \Theta \text{ and all } \zeta\in (\sK_{\mu})^\co .
	$$
\end{lemma} 

\begin{lemma}[{See \cite[Lemma~2.16]{Markfelder24}}] \label{lemma:properties-dist} 
	The following statements hold.
	\begin{enumerate}
		\item We have 
		\begin{equation} \label{eq:dist-bounded}
			0\leq \dist (\zeta, \sK_{\mu}) \leq 2c \qquad \text{ for all } \mu\in \Theta \text{ and all } \zeta\in (\sK_\mu)^\co,
		\end{equation}
		where $c>0$ is the bound from Lemma~\ref{lemma:uniform-boundedness-KU}.
  
		\item For all $\zeta_1,\zeta_2\in \R^M$ and $\mu\in \Theta$ it holds that 
		\begin{equation} \label{eq:dist-cont-z}
			\Big| \dist(\zeta_1, \sK_{\mu}) - \dist(\zeta_2, \sK_{\mu}) \Big| \leq |\zeta_1 - \zeta_2|.
		\end{equation}
		Estimate \eqref{eq:dist-cont-z} also holds if one replaces $\sK_\mu$ by\footnote{In fact one can replace $\sK_\mu$ by any other compact set $S\subset \R^M$ and \eqref{eq:dist-cont-z} still holds.} $(\sK_\mu)^\co$ or $\partial(\sK_\mu)^\co$.
		
		\item Let $\ep>0$ and $\delta>0$ the $\delta$ from item \ref{item:suitable-K-continuity} of Defn.~\ref{defn:suitable-K} which corresponds to $\frac{\ep}{3}$. Then for all $\zeta\in \R^M$ and $\mu_1,\mu_2\in \Theta$ with $|\mu_1 - \mu_2 |<\delta$ it holds that 
		\begin{equation} \label{eq:dist-cont-b}
			\Big| \dist(\zeta, \sK_{\mu_1}) - \dist(\zeta, \sK_{\mu_2}) \Big| \leq \ep .
		\end{equation}
		Estimate \eqref{eq:dist-cont-b} also holds if one replaces $\sK_\mu$ by $(\sK_\mu)^\co$ or $\partial(\sK_\mu)^\co$. 
	\end{enumerate}
\end{lemma}

\subsubsection{Plane waves and the wave cone $\Lambda$} \label{subsubsec:gen-prelim-pl.waves+Lambda}

In this paper we use localized plane waves as building blocks for the convex integration scheme. To this end we have to study plane wave solutions of the linear system 
\begin{equation} \label{eq:lin-eq-homo}
	\partial_t a(\zeta(t,x)) + \Div A (\zeta(t,x)) = 0,
\end{equation}
and the corresponding wave cone $\Lambda$. Note that \eqref{eq:lin-eq-homo} is homogeneous in contrast to \eqref{eq:lin-eq}. 

A plane wave solution $\widetilde{\zeta}$ of \eqref{eq:lin-eq-homo} is a solution of the form 
\begin{equation} \label{eq:plane-wave}
	\widetilde{\zeta}(t,x) = \zeta \, h((t,x)\cdot \eta),
\end{equation} 
where $\zeta\in \R^M$ is a constant vector, $h:\R\to\R$ is a profile and $\eta\in \R^{1+n}\setminus\{0\}$ is a direction in space-time. We write $\eta=(\eta_t,\eta_x)$ with $\eta_t\in \R$ and $\eta_x\in \R^n$. In other words $\eta_t$ and $\eta_x$ denote the temporal and spatial components of the space-time vector $\eta$, respectively. 

Next we define the wave cone $\Lambda$. 

\begin{defn} \label{defn:wave-cone} 
	The wave cone $\Lambda$ is defined by 
	$$
		\Lambda := \left\{ \zeta\in \R^M\,\Big|\,\exists\eta=(\eta_t,\eta_x)\in \R^{1+n}\ \text{ with }\ \eta_x \neq 0\ \text{ and }\ a(\zeta) \eta_t + A(\zeta) \cdot \eta_x = 0 \right\}. 
	$$
\end{defn}

It is simple to see that if $\zeta\in \Lambda$ and $h\in C^1(\R)$, then the plane wave defined in \eqref{eq:plane-wave} is a solution of \eqref{eq:lin-eq-homo}. Indeed using the linearity \eqref{eq:linearity-a}, \eqref{eq:linearity-A}, we compute for any $i=1,...,m$
\begin{align*}
    [\partial_t a(\widetilde{\zeta}) + \Div A (\widetilde{\zeta})]_i &= [a(\partial_t \widetilde{\zeta})]_i + \sum_{j=1}^n [A(\partial_j \widetilde{\zeta})]_{ij} \\
    &= \big[a(\zeta h'((t,x)\cdot \eta) \eta_t)\big]_i + \sum_{j=1}^n \big[A(\zeta h'((t,x)\cdot \eta) \eta_j)\big]_{ij} \\
    &= h'((t,x)\cdot \eta) \left( [a(\zeta)]_i \eta_t + \sum_{j=1}^n [A(\zeta)]_{ij} \eta_j \right) \\
    &= h'((t,x)\cdot \eta) \big[ a(\zeta) \eta_t + A(\zeta) \cdot \eta_x \big]_i \ = \ 0.
\end{align*}

\begin{rem} \label{rem:lambdatilde-vs-lambda}
    First we remark that Defn.~\ref{defn:wave-cone} and the corresponding definition in \cite{Markfelder24} (see Defn.~2.6 therein) differ slightly. In this paper we require $\eta_x\neq 0$ whereas in \cite{Markfelder24} one only asks for $\eta\neq 0$. This difference results from the fact that here we are looking for solutions that are weakly continuous in time, while in \cite{Markfelder24} only $L^\infty$ solutions are sought, see the proof of Lemma~\ref{lemma:pert-prop-step1}.

    Secondly, we would like to note that in cases where the wave cone $\Lambda$ is too large to guarantee existence of suitable differential operators $\opL_\zeta$ for all $\zeta\in \Lambda$ (see Defn.~\ref{defn:suitable-operator}), one may consider a smaller cone $\widetilde{\Lambda}\subset \Lambda$ instead. Such a procedure is implemented in detail in \cite{Markfelder24}, see also Rem.~2.12 therein. We note that the wave cones in most examples that are studied in the present paper (see Sects.~\ref{sec:ex-incomp-euler}-\ref{sec:ex-compr-prim}) are not too large in the sense mentioned above, which means that there is no need to shrink the wave cone $\Lambda$ to some $\widetilde{\Lambda}$. However in the last example (see Sect.~\ref{sec:ex-QG}) we will have to shrink the wave cone to some smaller cone. We will abuse notation and still denote this smaller cone by $\Lambda$, see Sect.~\ref{subsec:ex-QG-prel} and Rem.~\ref{rem:QG-wave-cone} for more details. 
\end{rem}

\subsubsection{Suitable differential operator $\opL_\zeta$} \label{subsubsec:gen-prelim-operator}

Let us now recall the notion of a suitable differential operator $\opL_\zeta$ from \cite{Markfelder24}. We need such operators to localize the plane waves. 

\begin{defn}[{See \cite[Defn.~2.7]{Markfelder24}}] \label{defn:suitable-operator} 
	Let $\zeta\in\Lambda$ and $\ell\in \N$. An $\ell$-th order homogeneous differential operator 
	$$
		\opL_\zeta : C^\infty(\R^{1+n}) \to C^\infty(\R^{1+n};\R^M) 
	$$
	is called \emph{suitable} if it satisfies the following two properties:
	\begin{enumerate}
		\item \label{item:suitable-operator-a} For any function $g\in C^\infty(\R^{1+n})$, $\opL_\zeta[g]$ solves the linear system \eqref{eq:lin-eq-homo}, i.e.
		$$
			\partial_t a(\opL_\zeta[g]) + \Div A (\opL_\zeta[g]) = 0
		$$

		\item \label{item:suitable-operator-b} If we set $g(t,x):= h((t,x)\cdot \eta)$ with an arbitrary function $h\in C^\infty(\R)$ and where $\eta$ corresponds to $\zeta\in \Lambda$ (see Defn.~\ref{defn:wave-cone}), we obtain
		$$
			\opL_\zeta [g](t,x) = \zeta \,h^{(\ell)}((t,x)\cdot \eta) ,
		$$
		where $h^{(\ell)}$ denotes the $\ell$-th derivative of $h$. 
	\end{enumerate}
\end{defn}

\subsection{Geometric setup} \label{subsec:general-geometric}

Next we consider the family of relaxed sets $(\sU_\mu)_{\mu\in \Theta}$. These sets have to be chosen in such a way that they are compatible with plane waves. Therefore we set 
$$
    \sU_{\mu}:= \interior{\big((\sK_{\mu})^\Lambda\big)}\qquad\text{ for all }\mu\in \Theta,
$$
i.e.~the interior of the $\Lambda$-convex hull of the set $\sK_{\mu}$, see Appendix~\ref{app:Lconvex} for the definition of the $\Lambda$-convex hull and its properties. 

Similar to \cite{Markfelder24} we will assume that 
\begin{equation} \label{eq:KbLambda=KbCo}
	\interior{\big((\sK_{\mu})^\Lambda\big)} = \interior{\big((\sK_{\mu})^\co\big)} \qquad \text{ for all } \mu\in \Theta,
\end{equation}
i.e.~the interior of the $\Lambda$-convex hull of $\sK_{\mu}$ coincides with the interior of its convex hull. Note that assumption \eqref{eq:KbLambda=KbCo} is crucial in the proof of the convex integration theorem.

\begin{rem} \label{rem:interiorVSnoint}
    We would like to remark that in \cite{Markfelder24} one works with the stronger assumption 
    \begin{equation} \label{eq:KbLambda=KbCo-noint}
	   (\sK_{\mu})^\Lambda = (\sK_{\mu})^\co \qquad \text{ for all } \mu\in \Theta,
    \end{equation}
    rather than \eqref{eq:KbLambda=KbCo}. It is however not difficult to realize that one could alternatively work with \eqref{eq:KbLambda=KbCo} instead of \eqref{eq:KbLambda=KbCo-noint} even in \cite{Markfelder24} since the proofs therein would still be true without modification. On the other hand in many settings one can even show \eqref{eq:KbLambda=KbCo-noint}, see e.g.~the examples considered in Sects.~\ref{sec:ex-incomp-euler}-\ref{sec:ex-compr-prim}. However in Sect.~\ref{sec:ex-QG} (on the quasi-geostrophic equations) we will only be able to prove \eqref{eq:KbLambda=KbCo}, which therefore makes it necessary to consider the more general case (i.e.~assumption \eqref{eq:KbLambda=KbCo}).
\end{rem}

In the construction below, we will need a map which measures the distance of $\zeta\in \sU_\mu$ to $\sK_\mu$ in some sense. Note that in \cite{Markfelder24} $\dist(\zeta,\sK_\mu)$ serves as such a map. However in the present paper where we construct solutions that are weakly continuous in time, the above choice used in \cite{Markfelder24} is not appropriate here\footnote{In particular, $\dist(\zeta,\sK_\mu)$ does not satisfy property \ref{item:suitable-D-concave} in Defn.~\ref{defn:suitable-distance-map}. However, this property will be essential when proving Lemma~\ref{lemma:properties-I}~\ref{item:prop-I-Baire1} below.}. Instead we will assume existence of a suitable distance map in the following sense.

\begin{defn} \label{defn:suitable-distance-map} 
    A map\footnote{Here $\closure{B_c}(0)$ denotes the closed ball in $\R^M$ centered at $0$ with radius $c$, where the latter is the bound from Lemma~\ref{lemma:uniform-boundedness-KU}.} $D\in C ( \Theta \times \closure{B}_c(0) ; \R)$ is called a \emph{suitable distance map} if it satisfies the following properties: 
    \begin{enumerate}
        \item \label{item:suitable-D-concave} For all $\mu\in \Theta$ the map $\zeta\mapsto D(\mu,\zeta)$ is concave on $\closure{B}_c(0)$. 

        \item \label{item:suitable-D-=0} For all $\mu\in \Theta$ and all $\zeta\in \sK_\mu$ it holds that $D(\mu,\zeta)= 0$. 

        \item \label{item:suitable-D-inK} If $D(\mu,\zeta)=0$ and $\zeta\in (\sK_\mu)^\co$, then $\zeta\in \sK_\mu$.
    \end{enumerate}
\end{defn}

A suitable distance map has some properties which follow immediately from its definition. These properties are the content of the following lemma.

\begin{lemma} \label{lemma:suitable-distance-map} 
    Let $D\in C ( \Theta \times \closure{B}_c(0) ; \R)$ be a suitable distance map. Then the following statement holds.
    \begin{enumerate} \setcounter{enumi}{3}
        \item \label{item:L-suitable-D-geq0} For all $\mu\in \Theta$ and all $\zeta\in (\sK_\mu)^\co$ it holds that $D(\mu,\zeta)\geq 0$. 
    \end{enumerate}
    As soon as the family $(\sK_\mu)_{\mu\in\Theta}$ is suitable in the sense of Defn.~\ref{defn:suitable-K}, then we also have
    \begin{enumerate} \setcounter{enumi}{4}
        \item \label{item:L-suitable-D->0} For all $\mu\in \Theta$ and all $\zeta\in \sU_\mu$ it holds that $D(\mu,\zeta)>0$.
    \end{enumerate}
\end{lemma}

\begin{proof}
Statement~\ref{item:L-suitable-D-geq0} follows immediately from items~\ref{item:suitable-D-concave} and \ref{item:suitable-D-=0} in Defn.~\ref{defn:suitable-distance-map}. The claim in item~\ref{item:L-suitable-D->0} can be deduced from Defn.~\ref{defn:suitable-K}~\ref{item:suitable-K-boundary} and Lemma~\ref{lemma:suitable-distance-map}~\ref{item:L-suitable-D-geq0}.
\end{proof}

Moreover we point out that a suitable distance map $D$ is even uniformly continuous on $\Theta \times \closure{B}_c(0)$ as the latter is a compact subset of $\R^m\times \R^M$. This allows to define the following.

\begin{defn} \label{defn:delta}
    Let $D$ a suitable distance map. We denote by $\delta_D:(0,\infty)\to (0,2c)$ a map with the following property. For all $\ep>0$ and all $(\mu_1,\zeta_1),(\mu_2,\zeta_2)\in \Theta \times \closure{B}_c(0)$ with $\big| (\mu_1,\zeta_1)-(\mu_2,\zeta_2)\big| \leq \delta_D(\ep)$ it holds that $\big|D(\mu_1,\zeta_1) - D(\mu_2,\zeta_2) \big|\leq \ep$.
\end{defn}

Finally we note that $D$ is bounded on $\Theta \times \closure{B}_c(0)$.

The following proposition is a version of \cite[Prop.~2.10]{Markfelder24}.

\begin{prop}[{Cf.~\cite[Prop.~2.10]{Markfelder24}}] \label{prop:geom-property-U} 
	Consider a family of constitutive sets $(\sK_\mu)_{\mu\in \Theta}$ which is suitable in the sense of Defn.~\ref{defn:suitable-K}, and assume that \eqref{eq:KbLambda=KbCo} holds. Moreover let $D$ a suitable distance map in the sense of Defn.~\ref{defn:suitable-distance-map}. Finally let $\mu\in \Theta$ fixed, $\zeta\in \sU_{\mu}$ and $\ep>0$. Then there exist $N\in \N$ with $N\geq 2$, and $(\tau_i,\zeta_i)\in \R^+\times \R^M$ for $i=1,...,N$ with the following properties.
	\begin{enumerate}
		\item \label{item:geom-propU-a} The family $\{(\tau_i,\zeta_i)\}_{i=1,...,N}$ satisfies the\footnote{We refer to Appendix~\ref{app:Lconvex} for the definition and basic properties of the $H_N$-condition and the barycenter.} $H_N$-condition;
		
		\item \label{item:geom-propU-b} All endpoints lie in $\sU_{\mu}$ and their distance to $\partial \sU_{\mu}$ can be estimated by the distance of $\zeta$ to $\partial \sU_{\mu}$, more precisely
		\begin{align*} 
			\zeta_i &\in \sU_{\mu} & &\text{ for all }i=1,...,N, \\
			\dist ( \zeta_i , \partial \sU_{\mu}) &\geq \frac{\delta_D(\ep)}{2c} \dist(\zeta,\partial \sU_{\mu}) & &\text{ for all }i=1,...,N, 
		\end{align*}
		where $c$ is the bound from Lemma~\ref{lemma:uniform-boundedness-KU} and $\delta_D$ is the map defined in Defn.~\ref{defn:delta} which corresponds to $D$.
		
		\item \label{item:geom-propU-c} All endpoints are close to $\sK_{\mu}$, i.e.
		$$
			D(\mu,\zeta_i)\leq \ep \qquad \text{ for all }i=1,...,N.
		$$ 
        
		\item \label{item:geom-propU-d} The barycenter of the family $\{(\tau_i,\zeta_i)\}_{i=1,...,N}$ is given by $\zeta$, i.e. 
		$$
			\sum_{i=1}^N \tau_i \zeta_i = \zeta .
		$$
	\end{enumerate} 
\end{prop} 

\begin{proof}
The proof works similar to \cite[Proof of Prop.~2.10]{Markfelder24}. We only replace the choice of $\tau$ by $\tau\in (0,1)$ with 
\begin{equation} \label{eq:002}
    \frac{\delta_D(\ep)}{2c} \leq 1-\tau \leq \frac{\delta_D(\ep)}{\max\limits_{i=1,...,N}|\zeta-\widehat{\zeta}_i|},
\end{equation}
where we use the same notation as in \cite[Proof of Prop.~2.10]{Markfelder24}. Then we proceed exactly as in \cite[Proof of Prop.~2.10]{Markfelder24} to see that items \ref{item:geom-propU-a}, \ref{item:geom-propU-b} and \ref{item:geom-propU-d} of Prop.~\ref{prop:geom-property-U} hold\footnote{Actually one assumes \eqref{eq:KbLambda=KbCo-noint} instead of \eqref{eq:KbLambda=KbCo} in \cite[Prop.~2.10]{Markfelder24}. It is however simple to verify that \cite[Proof of Prop.~2.10]{Markfelder24} still works if one replaces \eqref{eq:KbLambda=KbCo-noint} by \eqref{eq:KbLambda=KbCo}, see also Rem.~\ref{rem:interiorVSnoint}.}. 

Next we observe that \eqref{eq:002} ensures
$$
    \big|(\mu,\zeta_i) - (\mu,\widehat{\zeta}_i)\big| = |\zeta_i - \widehat{\zeta}_i| = (1-\tau) |\zeta - \widehat{\zeta}_i| \leq \delta_D(\ep)\qquad \text{ for all }i=1,...,N,
$$
again see \cite[Proof of Prop.~2.10]{Markfelder24} for more details. Together with Defn.~\ref{defn:suitable-distance-map}~\ref{item:suitable-D-=0} this implies item \ref{item:geom-propU-c} of Prop.~\ref{prop:geom-property-U}.
\end{proof}

\subsection{Functional setup} \label{subsec:general-functional}

\subsubsection{Convex integration theorem}

Next we state our convex integration theorem. The most important difference between Thm.~\ref{thm:conv-int} and \cite[Thm.~2.11]{Markfelder24} is that the solutions $\zeta$ lie in $\Cweak([0,T];L^r(\Omega;\R^M))$ rather than just $L^\infty((0,T)\times \Omega;\R^M)$, and that the statement in item~\ref{item:main-thm-sol2} holds for all $t\in (0,T)$ and a.e.~$x\in \Omega$ rather than for a.e.~$(t,x)\in (0,T)\times \Omega$. An analogous statement in the context of the incompressible Euler equations can be found in \cite[Prop.~2]{DelSze10}.

\begin{thm} \label{thm:conv-int} 
    Let $T\in (0,\infty)\cup\{\infty\}$, $\Omega\subset\R^n$ be a bounded\footnote{See Rem.~\ref{rem:unbounded-domains} for a way how to generalize Thm.~\ref{thm:conv-int} to unbounded domains $\Omega$.} Lipschitz domain, and $r\in (1,\infty)$ an integrability parameter. Let furthermore $a,b_1,b_2:\R^M\to \R^m$ and $A,B:\R^M\to \R^{m\times n}$ linear, $\theta\in \Cb([0,T]\times \closure{\Omega};\R^m)\cap C^1([0,T]\times \closure{\Omega};\R^m)$, and $(\sK_{\mu})_{\mu\in\Theta}$ a family of constitutive sets, where $\Theta:=\closure{\theta([0,T]\times \closure{\Omega})}$ and $\sK_\mu\subset \R^M$ for all $\mu\in \Theta$. \\ 
	Suppose that the following structural assumptions hold:
	\begin{itemize} 
		\item The family of constitutive sets $(\sK_{\mu})_{\mu\in\Theta}$ is suitable in the sense of Defn.~\ref{defn:suitable-K}.
		
		\item For any $\zeta\in \Lambda$ there exists $\ell\in \N$ and an $\ell$-th order homogeneous differential operator 
		$$
			\opL_\zeta : C^\infty(\R^{1+n}) \to C^\infty(\R^{1+n};\R^M) 
		$$ 
		which is suitable in the sense of Defn.~\ref{defn:suitable-operator}.
		
		\item For any $\mu\in \Theta$ the interior of the $\Lambda$-convex hull of $\sK_{\mu}$ coincides with the interior of its convex hull, i.e.\footnote{Note that this assumption is slightly weaker than the corresponding assumption in \cite[Thm.~2.11]{Markfelder24}. This fact is explained in more detail in Rem.~\ref{rem:interiorVSnoint}.} 
		\begin{equation} \label{eq:ass-KbLambda=KbCo} 
			\interior{\big((\sK_{\mu})^\Lambda\big)} = \interior{\big((\sK_{\mu})^\co\big)} \qquad \text{ for all } \mu\in \Theta.
		\end{equation}

        \item There exists a map $D\in C ( \Theta \times \closure{B}_c(0) ; \R)$ which is a suitable distance map in the sense of Defn.~\ref{defn:suitable-distance-map}, where $c$ is the bound from Lemma~\ref{lemma:uniform-boundedness-KU}.
    \end{itemize}
	Finally assume there exists 
    $$
        \ov{\zeta}\in \Cweak([0,T];L^r(\Omega;\R^M)) \cap C^1((0,T)\times \closure{\Omega};\R^M)
    $$ 
    with the following properties: 
    \begin{itemize}
		\item It satisfies the linear system of PDEs \eqref{eq:lin-eq}, i.e. 
		$$ 
            \partial_t a(\ov{\zeta}) + \Div A (\ov{\zeta}) = \partial_t b_1(\theta) + \Div B(\theta) + b_2(\theta) 
		$$
		holds pointwise for all $(t,x)\in (0,T)\times \Omega$.
        
		\item It takes values in $\sU_{\theta(t,x)}$ where $\sU_\mu:= \interior{\big((\sK_{\mu})^\Lambda\big)}$ for all $\mu\in \Theta$, i.e. 
		$$ 
			\ov{\zeta}(t,x) \in \sU_{\theta(t,x)}\qquad \text{ for all }(t,x)\in (0,T)\times \Omega . 
		$$ 
	\end{itemize}
	Then there exist infinitely many solutions 
    $$
        \zeta\in \Cweak([0,T];L^r(\Omega;\R^M)) \cap L^\infty((0,T)\times \Omega;\R^M)
    $$ 
    in the following sense:
    \begin{enumerate}
		\item \label{item:main-thm-sol1a} They satisfy \eqref{eq:lin-eq} in the sense of distributions with initial and boundary data given by $\ov{\zeta}$, more precisely 
		\begin{align*} 
			&-\int_0^T\int_\Omega \Big[ a(\zeta) \cdot \partial_t \varphi + A(\zeta) : \Grad \varphi \Big] \dx\dt + \int_0^T\int_{\partial \Omega} (A(\ov{\zeta})\cdot \nu)\cdot \varphi \dS\dt \\
            &\qquad + \int_\Omega \Big[ a(\ov{\zeta}(T,x)) \varphi(T,x) - a(\ov{\zeta}(0,x)) \varphi(0,x) \Big] \dx \\
            &= \int_0^T \int_\Omega \Big[ \partial_t b_1(\theta) + \Div B(\theta) + b_2(\theta) \Big] \cdot \varphi \dx\dt
		\end{align*}
		for all test functions $\varphi\in \Cc([0,T]\times \closure{\Omega};\R^m)$. Here $\nu$ denotes the outward pointing normal vector on $\partial \Omega$. 

        \item \label{item:main-thm-sol1b} They satisfy the initial condition
        \begin{equation*}
            \zeta(0,x) = \ov{\zeta}(0,x) \qquad \text{ for a.e. }x\in \Omega,
        \end{equation*}
        and moreover if $T<\infty$, the end condition
        \begin{equation*}
            \zeta(T,x) = \ov{\zeta}(T,x) \qquad \text{ for a.e. }x\in \Omega.
        \end{equation*}
		
		\item \label{item:main-thm-sol2} They take values in $\sK_{\theta(t,x)}$, i.e.
		\begin{equation*} 
			\zeta(t,x)\in \sK_{\theta(t,x)}\qquad \text{ for all }t\in (0,T)\text{ and a.e. }x\in \Omega. 
		\end{equation*}
	\end{enumerate}
\end{thm}

\begin{rem} 
    Note that if $T=\infty$, the term containing $a(\ov{\zeta}(T,x)) \varphi(T,x)$ in the equation in item~\ref{item:main-thm-sol1a} of Thm.~\ref{thm:conv-int} vanishes since in this case $\varphi$ is compactly supported in $[0,\infty)\times \closure{\Omega}$.
\end{rem}

The remainder of this subsection is devoted to the proof of Thm.~\ref{thm:conv-int}. As part of the proof we will generalize several of the ideas used in \cite{DelSze10}.

\subsubsection{The functionals $\I_{\sigma_0,\sigma_1}$ and their properties}

Let $T\in (0,\infty)\cup\{\infty\}$, $\Omega\subset\R^n$ bounded, $r\in (1,\infty)$, $a,b_1,b_2:\R^M\to \R^m$ and $A,B:\R^M\to \R^{m\times n}$ linear, $\theta\in \Cb([0,T]\times \closure{\Omega};\R^m)\cap C^1([0,T]\times \closure{\Omega};\R^m)$, $(\sK_{\mu})_{\mu\in \Theta}$, $D\in C(\Theta\times \ov{B}_c(0);\R)$ and $\ov{\zeta}\in \Cweak([0,T];L^r(\Omega;\R^M)) \cap C^1((0,T)\times \closure{\Omega};\R^M)$ be given such that the assumptions of Thm.~\ref{thm:conv-int} hold.  

First we define the set $X_0$ similar to \cite[Defn.~2.13]{Markfelder24}, see also \cite[Defn.~4]{DelSze10}.

\begin{defn} \label{defn:X0andX}
	We define the set $X_0$ by
	$$
		X_0 := \left\{ \zeta\in \Cweak([0,T];L^r(\Omega;\R^M)) \cap C^1((0,T)\times \closure{\Omega};\R^M) \, \Big|\, \text{Properties \eqref{eq:defn-X0-eq}-\eqref{eq:defn-X0-boundary-time-end} hold} \right\} , 
	$$
	where the properties \eqref{eq:defn-X0-eq}-\eqref{eq:defn-X0-boundary-time-end} read as follows:
	\begin{align}
		\partial_t a(\zeta) + \Div A (\zeta) &= \partial_t b_1(\theta) + \Div B(\theta) + b_2(\theta) & & \text{ pointwise for all }(t,x)\in (0,T)\times \Omega; \label{eq:defn-X0-eq} \\ 
		\zeta(t,x) &\in \sU_{\theta(t,x)} & & \text{ for all }(t,x)\in (0,T)\times \Omega ; \label{eq:defn-X0-subs} \\ 
		\zeta(t,x) &= \ov{\zeta}(t,x) & & \text{ for all }(t,x)\in (0,T)\times \partial\Omega, \label{eq:defn-X0-boundary-space} \\
        \zeta(0,x) &= \ov{\zeta}(0,x) & & \text{ for a.e. } x\in\Omega, \label{eq:defn-X0-boundary-time-init} \\ 
        \zeta(T,x) &= \ov{\zeta}(T,x) & & \text{ for a.e. } x\in\Omega \text{ if } T<\infty. \label{eq:defn-X0-boundary-time-end} 
	\end{align}
\end{defn}

So one may say that $\zeta\in X_0$ satisfies the linear system \eqref{eq:lin-eq}, takes values in $\sU_{\theta(t,x)}$ and coincides on the boundary in space-time with $\ov{\zeta}$.

Next we denote the topology on $\Cweak([0,T];L^r(\Omega;\R^M))$ by $\topology_1$, and the $L^\infty$ weak-$\ast$ topology on $L^\infty((0,T)\times \Omega;\R^M)$ by $\topology_2$. We endow the space 
\begin{equation} \label{eq:function-space}
    \Cweak([0,T];L^r(\Omega;\R^M)) \cap L^\infty((0,T)\times \Omega;\R^M)
\end{equation}
with the topology\footnote{To be precise, we set $\topology:= \widetilde{\topology}_1 \cap \widetilde{\topology}_2$, where $\widetilde{\topology}_1$ is the subspace topology of $\topology_1$ on \eqref{eq:function-space}, and analogously $\widetilde{\topology}_2$ is the subspace topology of $\topology_2$ on \eqref{eq:function-space}.} $\topology:= \topology_1 \cap \topology_2$. 

We define $X$ as the closure of $X_0$ with respect to the topology $\topology$. 

As in \cite[Prop.~2.14]{Markfelder24} and also in \cite{DelSze10}, we observe the following.

\begin{prop} \label{prop:X-metric-space}
	There exists a metric $d$ on $X$ which induces the topology $\topology$, and furthermore the metric space $(X,d)$ is complete. 
\end{prop}

\begin{proof} 
    In order to see that the topology $\topology_1$ is metrizable on $X$, we proceed as in \cite{DelSze10}. We obtain from Lemma~\ref{lemma:uniform-boundedness-KU}, assumption \eqref{eq:ass-KbLambda=KbCo} and condition \eqref{eq:defn-X0-subs} that for any $\zeta\in X_0$, 
    \begin{equation} \label{eq:006} 
        \int_\Omega |\zeta(t,x)|^r \dx \leq c^r |\Omega|\qquad \text{ for all }t\in(0,T).
    \end{equation}
    By weak continuity and the fact that $\zeta\mapsto |\zeta|^r$ is convex, this even holds for $t\in [0,T]$. So $\zeta:[0,T]\to L^r(\Omega;\R^M)$ takes values in a bounded subset $B$ of $L^r(\Omega;\R^M)$. On $B$ the weak topology of $L^r(\Omega;\R^M)$ is metrizable (see e.g.~\name{Megginson}~\cite[Cor.~2.6.20]{Megginson}), and we denote the corresponding metric by $d_B$. The metric $d_1$ defined by  
    $$
        d_1(\zeta_1,\zeta_2) := \sup_{t\in [0,T]} d_B (\zeta_1(t,\cdot) , \zeta_2(t,\cdot)),
    $$ 
    induces the topology $\topology_1$ on $C([0,T];(B,d_B))\subset \Cweak([0,T];L^r(\Omega;\R^M))$. Moreover, the space $\big(C([0,T];(B,d_B)),d_1\big)$ is complete and thus $X^1$, which we define to be closure of $X_0$ with respect to $d_1$, is complete with respect to $d_1$.

    Next we look at the topology $\topology_2$. Like in \cite[Proof of Prop.~2.14]{Markfelder24} we find that $X_0$ is bounded with respect to the $L^\infty$ norm by using Lemma~\ref{lemma:uniform-boundedness-KU}, assumption \eqref{eq:ass-KbLambda=KbCo} and condition \eqref{eq:defn-X0-subs}. Thus the weak-$\ast$ topology on $X^2$, which we define to be closure of $X_0$ with respect to $\topology_2$, is metrizable, and the resulting metric space $(X^2,d_2)$ is compact and complete, see e.g.~\cite[Proof of Prop.~5.1.6]{Markfelder} and references therein for more details. 

    Now, note that $d$ defined by 
    $$
        d(\zeta_1,\zeta_2):= \max\big\{ d_1(\zeta_1,\zeta_2),d_2(\zeta_1,\zeta_2)\big\},
    $$
    induces the topology $\topology$ on $X$, so it remains to show that $(X,d)$ is complete. To this end consider a Cauchy sequence $(\zeta_k)_{k\in \N}$ in $X$. Then $(\zeta_k)_{k}$ is a Cauchy sequence both with respect to $d_1$ and $d_2$. It is simple to see that $X\subset X^1\cap X^2$. Hence the completeness of $(X^1,d_1)$ and $(X^2,d_2)$ yield existence of $\ov{\zeta}^1\in X^1$ and $\ov{\zeta}^2\in X^2$ such that $\zeta_k \mathop{\to}\limits^{d_1} \ov{\zeta}^1$ and $\zeta_k \mathop{\to}\limits^{d_2} \ov{\zeta}^2$ as $k\to \infty$. 

    Note that convergence with respect to $d_1$ implies convergence $L^\infty$ strongly in time and $L^r$ weakly in space, which in turn implies convergence $L^\infty$ weakly-$\ast$ in time and $L^r$ weakly in space. Similarly convergence with respect to $d_2$ means convergence $L^\infty$ weakly-$\ast$ in time and space, which implies convergence $L^\infty$ weakly-$\ast$ in time and $L^r$ weakly in space. From this we infer that $(\zeta_k)_{k}$ converges to both $\ov{\zeta}^1$ and $\ov{\zeta}^2$ with respect to the same topology, namely $L^\infty$ weakly-$\ast$ in time and $L^r$ weakly in space. Thus $\ov{\zeta}^1=\ov{\zeta}^2=:\ov{\zeta}$ and we deduce that $\zeta_k \mathop{\to}\limits^{d} \ov{\zeta}$ as $k\to \infty$. Since $X$ is closed with respect to $d$, this implies $\ov{\zeta}\in X$. All in all this proves that $(X,d)$ is complete as desired.
\end{proof}

\begin{rem} \label{rem:unbounded-domains}
    In our framework we consider the spatial domain $\Omega$ to be bounded. One could also generalize this to unbounded domains. In the case of the incompressible Euler equations, this was done by \name{De~Lellis}-\name{Sz{\'e}kelyhidi}~\cite{DelSze10}. The issue is then, that we still need the left-hand side in \eqref{eq:006} to be bounded. In \cite{DelSze10} this achieved by estimating a power of $|\zeta|$ by $|\theta|$ on $(\sK_\theta)^\co$ and then requiring $\theta$ to be bounded in a suitable $L^p$ norm. Since an implementation of such an approach in an abstract framework makes the notation very complicated, we decided to restrict to bounded domains in order to keep the paper clearer.
\end{rem}

Next we define the functionals $\I_{\sigma_0,\sigma_1}$.

\begin{defn} \label{defn:I} 
	For any $0<\sigma_0\leq \sigma_1< T$ we define the functional $\I_{\sigma_0,\sigma_1}:X \to \R$ by 
	\begin{equation*} 
		\zeta \mapsto \I_{\sigma_0,\sigma_1} (\zeta) := \sup_{t\in [\sigma_0,\sigma_1]}\left[\int_{\Omega} D\big(\theta(t,x),\zeta(t,x)\big) \dx \right]. 
	\end{equation*}
\end{defn}

Similar to \cite[Lemma~2.17]{Markfelder24}, the following lemma summarizes some properties of the functionals $\I_{\sigma_0,\sigma_1}$. 

\begin{lemma} \label{lemma:properties-I}
	The following claims hold.
	\begin{enumerate} 
		\item \label{item:prop-I-Baire1} For all $0<\sigma_0\leq \sigma_1< T$ the map $\I_{\sigma_0,\sigma_1} : X\to \R$ is a Baire-1 function\footnote{Recall that a Baire-1 function $X\to \R$ is by definition the pointwise limit of a sequence of continuous functions. The important fact which we need in this paper (more precisely, for the proof of Lemma~2.24 below) is that the points of continuity of Baire-1 functions form a dense set as soon as $X$ is complete. This fact may be viewed as a version of the Baire Category Theorem.} with respect to the metric $d$. 
		
		\item \label{item:prop-I-X0} For all $0<\sigma_0< \sigma_1< T$ and all $\zeta\in X_0$ we have $\I_{\sigma_0,\sigma_1}(\zeta) > 0$. 
		
		\item \label{item:prop-I-sol} If $\zeta\in X$ with $\I_{\sigma_0,\sigma_1}(\zeta)=0$ for all $0<\sigma_0\leq \sigma_1< T$, then $\zeta$ is a solution in the sense specified in Thm.~\ref{thm:conv-int}, i.e.~items \ref{item:main-thm-sol1a}-\ref{item:main-thm-sol2} hold. 
	\end{enumerate}
\end{lemma}

\begin{proof} 
\begin{enumerate}
	\item Let $0<\sigma_0\leq \sigma_1< T$. We will show below that $\I_{\sigma_0,\sigma_1}:X\to \R$ is upper semi-continuous with respect to $d$. Together with the fact that $\I_{\sigma_0,\sigma_1}$ takes values in a bounded subset of $\R$, this already yields that $\I_{\sigma_0,\sigma_1}$ is Baire-1 with respect to $d$, see e.g.~\cite[Prop.~11 in Sect.~2.7 of Chap.~IX]{Bourbaki:Top2}. So let us prove that $\I_{\sigma_0,\sigma_1}$ is upper semi-continuous with respect to $d$.

    Assume there was a sequence $(\zeta_k)_{k\in\N}\subset X$ and $\zeta\in X$ such that $\zeta_k \mathop{\to}\limits^d \zeta$ as $k\to \infty$ but 
    $$
        \lim_{k\to \infty} \I_{\sigma_0,\sigma_1}(\zeta_k) > \I_{\sigma_0,\sigma_1}(\zeta).
    $$
    Thus there exist $(t_k)_{k\in \N}$ with $t_k\in [\sigma_0,\sigma_1]$ and 
    \begin{equation} \label{eq:001}
        \lim_{k\to \infty} \int_\Omega D\big(\theta(t_k,x),\zeta_k(t_k,x)\big)\dx > \I_{\sigma_0,\sigma_1}(\zeta).
    \end{equation}
    Since $[\sigma_0,\sigma_1]$ is compact, we may assume that $t_k\to t_0$ by considering a subsequence if necessary. By definition of the topology $\topology$, convergence with respect to $d$ implies $\zeta_k(t_k,\cdot)\rightharpoonup \zeta(t_0,\cdot)$ weakly in $L^r(\Omega;\R^M)$. Consequently Defn.~\ref{defn:suitable-distance-map}~\ref{item:suitable-D-concave} implies together with the uniform continuity of $D$ and $\theta$, that for all $\ep>0$
    \begin{align*}
        &\limsup_{k\to \infty} \int_\Omega D\big(\theta(t_k,x),\zeta_k(t_k,x)\big)\dx \\
        &\leq \limsup_{k\to \infty} \int_\Omega D\big(\theta(t_0,x),\zeta_k(t_k,x)\big)\dx \\
        &\qquad + \limsup_{k\to \infty} \int_\Omega \Big|D\big(\theta(t_k,x),\zeta_k(t_k,x)\big) - D\big(\theta(t_0,x),\zeta_k(t_k,x)\big)\Big|\dx \\
        &\leq \int_\Omega D\big(\theta(t_0,x),\zeta(t_0,x)\big)\dx + \ep \leq \I_{\sigma_0,\sigma_1}(\zeta) + \ep,
    \end{align*}
    a contradiction to \eqref{eq:001}. Thus $\I_{\sigma_0,\sigma_1}$ is upper semi-continuous with respect to $d$ as desired. 
	
	\item We infer from property \eqref{eq:defn-X0-subs} and Lemma~\ref{lemma:suitable-distance-map}~\ref{item:L-suitable-D->0}, that 
	$$
		D\big(\theta(t,x),\zeta(t,x)\big) > 0 \qquad \text{ for all } (t,x)\in (0,T)\times \Omega,
	$$
	which immediately leads to $\I_{\sigma_0,\sigma_1}(\zeta)> 0$.

	\item From the definition of the set $X_0$ (Defn.~\ref{defn:X0andX}) we see that any $\zeta\in X_0$ satisfies item \ref{item:main-thm-sol1a} in Thm.~\ref{thm:conv-int}. Now consider $\zeta\in X$. Then by definition of the set $X$, there is a sequence $(\zeta_k)_{k\in \N}\subset X_0$ with $\zeta_k \mathop{\to}\limits^d \zeta$, in particular $\zeta_k \mathop{\rightharpoonup}\limits^\ast \zeta$ in $L^\infty((0,T)\times \Omega;\R^M)$. As explained above item \ref{item:main-thm-sol1a} in Thm.~\ref{thm:conv-int} holds for any $\zeta_k$, $k\in \N$, and by taking the limit we obtain item \ref{item:main-thm-sol1a} even for $\zeta\in X$. 

    Notice that by \eqref{eq:defn-X0-boundary-time-init} we have $\zeta_k(0,x)=\ov{\zeta}(0,x)$ for all $k\in \N$ and a.e.~$x\in \Omega$. Moreover $\zeta_k$ converges to $\zeta$ in $\Cweak([0,T];L^r(\Omega;\R^M))$, and hence item \ref{item:main-thm-sol1b} in Thm.~\ref{thm:conv-int} holds. 
    
    Together with Lemma~\ref{lemma:suitable-distance-map}~\ref{item:L-suitable-D-geq0}, the fact that $\I_{\sigma_0,\sigma_1}(\zeta)=0$ for all $0<\sigma_0\leq \sigma_1<T$ implies 
    $$
		D\big(\theta(t,x),\zeta(t,x)\big) = 0 \qquad \text{ for all } t\in (0,T) \text{ and a.e. }x\in \Omega, 
	$$
    see \cite[Lemma~A.4.1]{Markfelder} for a more detailed proof of an analogous fact. This means by Defn.~\ref{defn:suitable-distance-map}~\ref{item:suitable-D-inK} that item \ref{item:main-thm-sol2} in Thm.~\ref{thm:conv-int} is true.
\end{enumerate}
\end{proof}

\subsubsection{Perturbation property and proof of the convex integration theorem (Thm.~\ref{thm:conv-int})}

Next we state the perturbation property, whose proof can be found in Sect.~\ref{subsec:general-pert-prop} below. 

\begin{prop}[The perturbation property] \label{prop:pert-prop}
	Let $\ep,\ov{\ep}>0$ and $0< \ov{\sigma}_0 < \sigma_0 < \sigma_1 < \ov{\sigma}_1 < T$. For all $\zeta \in X_0$ there exists a perturbation $\zeta_\pert\in X_0$ with the following properties.
	\begin{align}
		d(\zeta_\pert, \zeta) &\leq \ov{\ep} ; \label{eq:pert-prop1} \\
		\I_{\sigma_0,\sigma_1}(\zeta_\pert) &\leq \ep ; \label{eq:pert-prop2} \\
        \zeta_\pert(t,x) &= \zeta(t,x) \qquad \text{ for all }t\in (0,\ov{\sigma}_0)\cup (\ov{\sigma}_1,T)\text{ and all }x\in \closure{\Omega}. \label{eq:pert-prop3}
	\end{align}
\end{prop}

\begin{rem} 
    In the case of the incompressible Euler equations, Prop.~\ref{prop:pert-prop} corresponds to \cite[Prop.~3]{DelSze10}. Note however that in \cite{DelSze10} the estimate \eqref{eq:pert-prop2} is replaced by 
    $$
        \I_{\sigma_0,\sigma_1}(\zeta_\pert) \leq \I_{\sigma_0,\sigma_1}(\zeta) - \beta
    $$
    with a suitable constant $\beta>0$. In view of this we would like to point out that Prop.~\ref{prop:pert-prop} is not only a generalization of \cite[Prop.~3]{DelSze10} but also an improvement regarding the rate of convergence.
\end{rem}

In order to prove Thm.~\ref{thm:conv-int}, consider two sequences $(\sigma_{0,j})_{j\in \N},(\sigma_{1,j})_{j\in \N}\subset (0,T)$ with $\sigma_{0,j}\leq \sigma_{1,j}$ for all $j\in \N$ and $\sigma_{0,j}\searrow 0$, $\sigma_{1,j}\nearrow T$ as $j\to \infty$. Moreover define
\begin{align*}
	\Xi_j &:= \left\{\zeta \in X \, \Big| \, \I_{\sigma_{0,j},\sigma_{1,j}} \text{ is continuous at } \zeta \text{ with respect to } d\right\} , \\
	\Xi \ &:= \bigcap_{j\in \N} \Xi_j .
\end{align*}

The perturbation property allows to prove the following two lemmas.

\begin{lemma} \label{lemma:Xi-inf-elements}
	The sets $X_0$ and $\Xi$ contain infinitely many elements.
\end{lemma} 

\begin{lemma} \label{lemma:Xi-I=0}
	If $\zeta\in \Xi$, then $I_{\sigma_0,\sigma_1}(\zeta)=0$ for all $0<\sigma_0\leq \sigma_1< T$. 
\end{lemma} 

For the proofs of Lemmas~\ref{lemma:Xi-inf-elements} and \ref{lemma:Xi-I=0} we refer to the literature, e.g.~\cite[Lemmas~5.1.12, 5.1.13 and 5.1.14]{Markfelder}, \cite[Lemmas~2.19 and 2.20]{Markfelder24} or the corresponding statements in \cite{DelSze10}.

\begin{proof}[Proof of Thm.~\ref{thm:conv-int}] 
Now Thm.~\ref{thm:conv-int} follows from Lemmas~\ref{lemma:properties-I}~\ref{item:prop-I-sol}, \ref{lemma:Xi-inf-elements} and \ref{lemma:Xi-I=0}, see e.g.~\cite[Proof of Thm.~5.1.2]{Markfelder} for more details.
\end{proof}

\subsection{Proof of the perturbation property (Prop.~\ref{prop:pert-prop})} \label{subsec:general-pert-prop}

It remains to prove the perturbation property (Prop.~\ref{prop:pert-prop}). Note that compared to \cite{Markfelder24}, the functional $I_{\sigma_0,\sigma_1}(\zeta)$ defined in Defn.~\ref{defn:I} only contains a spatial integral while the corresponding definition in \cite{Markfelder24} contains a spatial and a temporal integral, and therefore the proof of the perturbation property is different in structure compared to \cite{Markfelder24}.

\subsubsection{Step 1: inductive lemma} 

One difference compared to \cite{Markfelder24} is that the estimate in Lemma~\ref{lemma:pert-prop-step1}~\ref{item:pert-prop-step1-e} below deals with integrals only in space and must hold for all times in a certain interval. To be able to achieve this, the sets $\Gamma^\ast$ must have a non-empty intersection with each time slice. We implement this properly by introducing the notion of a suitable union of convex polytopes as follows. 

\begin{defn} \label{defn:suitable-union} 
    A set $\Gamma^\ast\subset \R^{1+n}$ in space-time is called a \emph{suitable union of convex polytopes} if it has the following properties (see Fig.~\ref{fig:suitable-union} for an illustration):
    \begin{itemize}
        \item The set $\Gamma^\ast$ is open and bounded. 
        \item The set $\Gamma^\ast$ can be written as the union of finitely many open, convex and pairwise disjoint polytopes.
        \item It holds that
        $$
            \Gamma^\ast \cap \{t=s\} \neq \emptyset \qquad \text{ for all }s\in (\sigma^\ast_0, \sigma^\ast_1 ),
        $$
        where 
        \begin{equation} \label{eq:defn-sigma-ast}
            \sigma^\ast_0:= \inf_{(t,x)\in \Gamma^\ast} t, \qquad \text{ and } \qquad \sigma^\ast_1:= \sup_{(t,x)\in \Gamma^\ast} t.
        \end{equation}
    \end{itemize}
\end{defn}

\begin{figure}[htb] 
	\centering
	\subfloat[\label{fig:1a}]{
		\centering
		\includegraphics[width=0.25\textwidth]{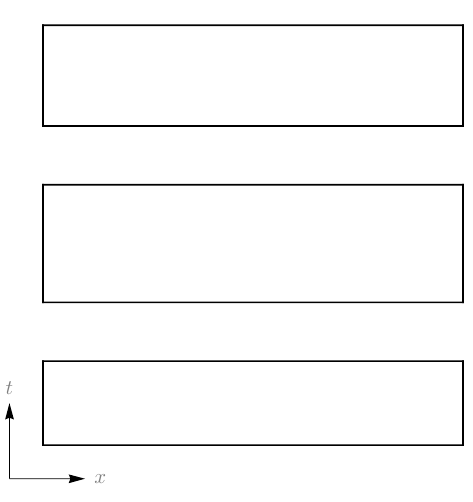}
	}
	\hspace{1cm}
	\subfloat[\label{fig:1b}]{
		\centering
		\includegraphics[width=0.25\textwidth]{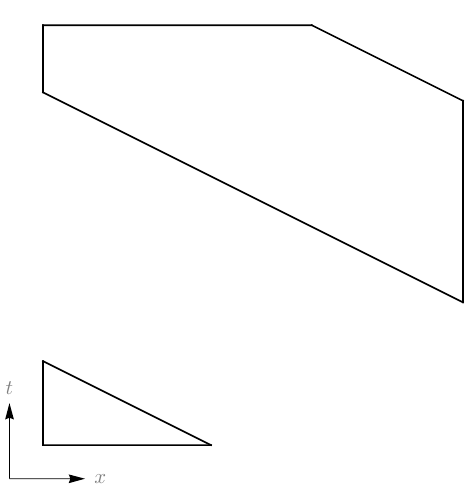}
	}
    \hspace{1cm}
	\subfloat[\label{fig:1c}]{
		\centering
		\includegraphics[width=0.25\textwidth]{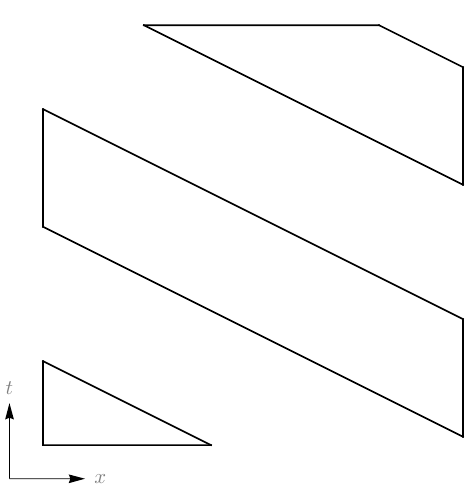}
	}
	\caption{Illustration of Defn.~\ref{defn:suitable-union}. While examples (a) and (b) are not suitable, (c) is.} 
	\label{fig:suitable-union}
\end{figure} 

We need the following version of \cite[Lemma~2.21]{Markfelder24}.

\begin{lemma}[{Cf.~\cite[Lemma~2.21]{Markfelder24}}] \label{lemma:pert-prop-step1} 
	Let $\Gamma^\ast\subset \R^{1+n}$ be a suitable union of convex polytopes (see Defn.~\ref{defn:suitable-union}), $\ep_1,\ep_2>0$, $\mu^\ast\in \Theta$ and $\zeta^\ast\in \sU_{\mu^\ast}$. Let furthermore $N\in\N$ with $N\geq 2$, and $(\tau_i,\zeta_i)\in \R^+ \times \sU_{\mu^\ast}$ for all $i=1,...,N$ such that 
	\begin{itemize} 
		\item the family $\{(\tau_i,\zeta_i)\}_{i=1,...,N}$ satisfies the $H_N$-condition and 
		\item its barycenter is given by $\zeta^\ast$, i.e.~$\sum_{i=1}^N \tau_i \zeta_i = \zeta^\ast$. 
	\end{itemize}
	Then there exists a sequence of oscillations $(\widetilde{\zeta}_k)_{k\in \N} \subset \Cc(\Gamma^\ast;\R^M)$ with the following properties:
	\begin{enumerate}
		\item \label{item:pert-prop-step1-a} The sequence $(\widetilde{\zeta}_k)_{k\in \N}$ converges to $0$ with respect to $d$, i.e.~$\widetilde{\zeta}_k \mathop{\to}\limits^d 0$ as $k\to \infty$. 
	\end{enumerate}	
	For each fixed $k\in \N$ the following statements hold:
	\begin{enumerate} \setcounter{enumi}{1}
		\item \label{item:pert-prop-step1-b} The linear system of PDEs $\ \partial_t a(\widetilde{\zeta}_k) + \Div A(\widetilde{\zeta}_k) = 0\ $ holds pointwise for all $(t,x)\in \Gamma^\ast$. 
		
		\item \label{item:pert-prop-step1-c} There exist open sets $\Gamma_i \subsetcomp \Gamma^\ast$ for all $i=1,...,N$ such that
		\begin{itemize} 
			\item the $\closure{\Gamma_i}$ are pairwise disjoint, 
			\item $\zeta^\ast+\widetilde{\zeta}_k(t,x) = \zeta_i$ for all $(t,x)\in \Gamma_i$ and
            \item each $\Gamma_i$ is a suitable union of convex polytopes in the sense of Defn.~\ref{defn:suitable-union}. 
		\end{itemize} 
		
		\item \label{item:pert-prop-step1-d} The sum $\zeta^\ast + \widetilde{\zeta}_k$ takes values in $\sU_{\mu^\ast}$ and its distance to $\partial\sU_{\mu^\ast}$ can be estimated by the distance of the $\zeta_i$ to $\partial\sU_{\mu^\ast}$, more precisely 
		\begin{align*}
			\zeta^\ast + \widetilde{\zeta}_k(t,x) &\in \sU_{\mu^\ast} & &\text{ for all }(t,x)\in \Gamma^\ast, \\ 
			\dist (\zeta^\ast + \widetilde{\zeta}_k(t,x), \partial\sU_{\mu^\ast}) &\geq \half \min_{i=1,...,N} \dist(\zeta_i,\partial\sU_{\mu^\ast}) & & \text{ for all }(t,x)\in \Gamma^\ast.  
		\end{align*}
		
		\item \label{item:pert-prop-step1-e} The estimate
        $$ 
			\int_{\Gamma^\ast\cap \{t=s\}} D\big(\mu^\ast, \zeta^\ast + \widetilde{\zeta}_k(s,x)\big) \dx < \ep_1 + \sum_{i=1}^N \int_{\Gamma_i\cap \{t=s\}} D(\mu^\ast,\zeta_i) \dx
		$$ 
        holds for all $s\in [\Sigma^\ast_0,\Sigma^\ast_1]$, where
        \begin{align*}
            \Sigma^\ast_0 &:= \left\{ \begin{array}{ll} \sigma^\ast_0 + \ep_2 & \text{ if there exists a face }F\text{ of }\Gamma^\ast \text{ with }t=\sigma^\ast_0 \text{ for all }(t,x)\in F, \\ \sigma^\ast_0 & \text{ else}, \end{array} \right. \\
            \Sigma^\ast_1 &:= \left\{ \begin{array}{ll} \sigma^\ast_1 - \ep_2 & \text{ if there exists a face }F\text{ of }\Gamma^\ast \text{ with }t=\sigma^\ast_1 \text{ for all }(t,x)\in F, \\ \sigma^\ast_1 & \text{ else}. \end{array} \right. 
        \end{align*}
        Keep in mind that $\sigma^\ast_0,\sigma^\ast_1$ are defined in \eqref{eq:defn-sigma-ast}.
	\end{enumerate} 
\end{lemma} 

\begin{proof} 
The only differences between Lemma~\ref{lemma:pert-prop-step1} and the corresponding \cite[Lemma~2.21]{Markfelder24} lie in items~\ref{item:pert-prop-step1-c} (in particular the third bullet point therein) and \ref{item:pert-prop-step1-e}. 

The proof of Lemma~\ref{lemma:pert-prop-step1} works exactly as \cite[Proof of Lemma~2.21]{Markfelder24}. So we only show here that the aforementioned differences hold. 

First we observe that each $\Gamma_i$ is indeed a suitable union of convex polytopes due to the fact that $\eta$ is not parallel to the time-axis, see Defn.~\ref{defn:wave-cone} and also Rem.~\ref{rem:lambdatilde-vs-lambda}, and because $k$ can be chosen suitably large. The former excludes the situation shown in Fig.~\ref{fig:1a}, the latter ensures that the situation depicted in Fig.~\ref{fig:1b} does not occur. Hence the third bullet point in item~\ref{item:pert-prop-step1-c} of Lemma~\ref{lemma:pert-prop-step1} holds. 

Secondly, we find that item~\ref{item:pert-prop-step1-e} is valid by adjusting the set $\widetilde{\Gamma}$ (which appears in \cite[Proof of Lemma~2.21]{Markfelder24}) to $\ep_1,\ep_2$ if necessary.
\end{proof}

\subsubsection{Step 2: perturbations around fixed vectors $\zeta^\ast$ and $\mu^\ast$} 

Similar to \cite[Lemma~2.22]{Markfelder24} we next combine Lemma~\ref{lemma:pert-prop-step1} with Prop.~\ref{prop:geom-property-U}.

\begin{lemma}[{Cf.~\cite[Lemma~2.22]{Markfelder24}}] \label{lemma:pert-prop-step2} 
	Let $\Gamma^\ast\subset \R^{1+n}$ be an open cube with edge length $h>0$. Let furthermore $\ep_1,\ep_2>0$, $\mu^\ast\in \Theta$ and $\zeta^\ast\in \sU_{\mu^\ast}$. Then there exists a sequence of oscillations $(\widetilde{\zeta}_k)_{k\in \N} \subset \Cc(\Gamma^\ast;\R^M)$ with the following properties:
	\begin{enumerate}
		\item \label{item:pert-prop-step2-a} The sequence $(\widetilde{\zeta}_k)_{k\in \N}$ converges to $0$ with respect to $d$, i.e.~$\widetilde{\zeta}_k \mathop{\to}\limits^d 0$ as $k\to \infty$. 
	\end{enumerate}	
	For each fixed $k\in \N$ the following statements hold:
	\begin{enumerate} \setcounter{enumi}{1}
		\item \label{item:pert-prop-step2-b} The linear system of PDEs $\ \partial_t a(\widetilde{\zeta}_k) + \Div A(\widetilde{\zeta}_k) = 0\ $ holds pointwise for all $(t,x)\in \Gamma^\ast$;
		
		\item \label{item:pert-prop-step2-c} The sum $\zeta^\ast + \widetilde{\zeta}_k$ takes values in $\sU_{\mu^\ast}$ and its distance to $\partial\sU_{\mu^\ast}$ can be estimated by the distance of the $\zeta_i$ to $\partial\sU_{\mu^\ast}$, more precisely 
		\begin{align*}
			\zeta^\ast + \widetilde{\zeta}_k(t,x) &\in \sU_{\mu^\ast} & &\text{ for all }(t,x)\in \Gamma^\ast, \\ 
			\dist (\zeta^\ast + \widetilde{\zeta}_k(t,x), \partial\sU_{\mu^\ast}) &\geq \frac{1}{4c} \,\delta_D\Big(\frac{\ep_1}{2 h^n}\Big) \,\dist(\zeta^\ast, \partial\sU_{\mu^\ast}) & & \text{ for all }(t,x)\in \Gamma^\ast, 
		\end{align*} 
		where $c$ is the bound from Lemma~\ref{lemma:uniform-boundedness-KU} and $\delta_D$ is the map defined in Defn.~\ref{defn:delta}; 
		
		\item \label{item:pert-prop-step2-d} The estimate 
		$$ 
			\int_{\Gamma^\ast \cap \{t=s\}} D\big(\mu^\ast,\zeta^\ast + \widetilde{\zeta}_k(s,x)\big) \dx < \ep_1 
		$$
        holds for all $s\in [\sigma^\ast_0 + \ep_2 , \sigma^\ast_1 - \ep_2]$, where $\sigma^\ast_0,\sigma^\ast_1$ are defined in \eqref{eq:defn-sigma-ast}. 
	\end{enumerate}
\end{lemma}

\begin{proof} 
We first apply Prop.~\ref{prop:geom-property-U} where $\mu^\ast$, $\zeta^\ast$ and $\frac{\ep_1}{2h^n}$ play the role of $\mu$, $\zeta$ and $\ep$ in Prop.~\ref{prop:geom-property-U}, respectively. Next we apply Lemma~\ref{lemma:pert-prop-step1}, where $\half\ep_1$ and $\ep_2$ play the role of $\ep_1$ and $\ep_2$ therein. Then the claims can be readily checked, see \cite[Proof of Lemma~2.22]{Markfelder24} for more details.
\end{proof}

\subsubsection{Step 3: conclusion} 

With the help of Lemma~\ref{lemma:pert-prop-step2} we are now ready to prove the perturbation property (Prop.~\ref{prop:pert-prop}). Note however that since our functionals $I_{\sigma_0,\sigma_1}$ differ from the ones used in \cite{Markfelder24}, a slightly different approach in proving Prop.~\ref{prop:pert-prop} is required. In particular the grid needs to be shifted (see below) as first done in \cite{DelSze10}.

\begin{proof}[Proof of the perturbation property (Prop.~\ref{prop:pert-prop})]
Let us first note that we may assume that $\sigma_0<\sigma_1$ as the case $\sigma_0=\sigma_1$ is trivial. 

We begin by fixing $\ov{D}>0$ such that $\ov{D}\geq \max\limits_{(\mu',\zeta')\in \Theta \times \closure{B}_c(0)} D(\mu',\zeta')$. Note that this is possible since $D$ is bounded on $\Theta \times \closure{B}_c(0)$.

Next we introduce a spatial grid in $\R^n$ with size $h>0$: Consider open cubes $Q_{\alpha_x,h}\subset\R^n$ for $\alpha_x\in \Z^n$ with center $h\alpha_x$ and edge length $h$, i.e.
$$
    Q_{\alpha_x,h}:= h\alpha_x + \left(-\half h, \half h\right)^n .
$$
Fix $\ov{h}>0$ such that the volume of the frame $\Omega_\tf:= \Omega \setminus \widetilde{\Omega}$ satisfies 
\begin{equation} \label{eq:step2-est1}
    |\Omega_\tf| \ov{D} \leq \frac{\ep}{4},
\end{equation}
where 
$$
    \widetilde{\Omega}:= \interior{\left( \bigcup_{\alpha_x\in \Z^n \text{ with }Q_{\alpha_x,\ov{h}} \subsetcomp \Omega} \closure{Q_{\alpha_x,\ov{h}}} \right) }.
$$

As explained in \cite[Proof of Prop.~2.18]{Markfelder24}, there exists $\beta>0$ such that 
\begin{equation*} 
	\dist\big( \zeta(t,x), \partial\sU_{\theta(t,x)}\big) > \beta \qquad \text{ for all } (t,x)\in (\ov{\sigma}_0,\ov{\sigma}_1)\times \widetilde{\Omega}. 
\end{equation*}
We also fix a ``shifting parameter'' $S\in \N$ such that 
\begin{equation} \label{eq:step2-est2}
    \frac{|\Omega| \ov{D}}{S} \leq \frac{\ep}{4}. 
\end{equation}

We notice that the functions $(t,x)\mapsto \zeta(t,x)$ and $(t,x)\mapsto \theta(t,x)$ are uniformly continuous on $[\ov{\sigma}_0,\ov{\sigma}_1]\times \closure{\Omega}$ as the latter is compact. Together with the uniform continuity of the maps 
$$
	(\zeta,\mu)\mapsto D(\mu,\zeta) \qquad \text{ and } \qquad (\zeta,\mu)\mapsto \dist(\zeta,\partial\sU_\mu),
$$
see Lemma~\ref{lemma:properties-dist}, we infer that there exists $q\in \N$ such that for $h:= \frac{\ov{h}}{q\cdot S}$ the following holds: If $x_1$ and $x_2$ lie in the same cube $Q_{\alpha_x,h}$ and $|t_1-t_2|<h$, then 
\begin{align}
	\Big| D\big(\theta(t_1,x_1), \zeta(t_1,x_1) + \widehat{\zeta} \big) - D\big(\theta(t_2,x_2) , \zeta(t_2,x_2) + \widehat{\zeta} \big) \Big| &< \frac{\ep}{4|\widetilde{\Omega}|}, \label{eq:unif-cont-K} \\ 
	\Big| \dist\big(\zeta(t_1,x_1) + \widehat{\zeta} , \partial\sU_{\theta(t_1,x_1)} \big) - \dist\big(\zeta(t_2,x_2) + \widehat{\zeta} , \partial\sU_{\theta(t_2,x_2)} \big) \Big| &< \frac{1}{8c} \delta_D \Big(\frac{\ep}{8 |\widetilde{\Omega}| }\Big)\beta , \label{eq:unif-cont-boundaryU} 
\end{align}
for any $\widehat{\zeta}\in \R^M$. In addition to that we may assume that 
\begin{equation} \label{eq:003}
    3h< \ov{\sigma}_1 - \sigma_1 \qquad \text{ and } \qquad 3h< \sigma_0 - \ov{\sigma}_0
\end{equation}
by increasing $q$ if necessary. 

We set 
$$
	\mathcal{J}_x := \left\{ \alpha_x\in \Z^n \,\Big|\, Q_{\alpha_x,h} \subset \widetilde{\Omega} \right\}.
$$
and observe that $\mathcal{J}_x$ contains only finitely many points $\alpha_x$. Moreover we have 
$$
	\widetilde{\Omega} = \interior{ \left(\bigcup_{\alpha_x\in \mathcal{J}_x} \closure{Q_{\alpha_x,h}} \right)}
$$
because $\frac{\ov{h}}{h}\in \N$, and consequently
\begin{equation} \label{eq:aux06}
	|\widetilde{\Omega}| = |\mathcal{J}_x| \cdot |Q_{\alpha_x,h}| = |\mathcal{J}_x| h^n .
\end{equation} 

Next we extend our grid with respect to time. To this end, we set 
$$
    \mathcal{J}_t:= \left\{ 0,1,...,\left\lceil \frac{\sigma_1-\sigma_0}{h} + 1\right\rceil \right\} ,
$$
and introduce the cubes in space-time\footnote{Here we use the usual notation $|\alpha|= |\alpha_1| + ... + |\alpha_n|$ for $\alpha=(\alpha_1,...,\alpha_n)\in \Z^n$.} 
\begin{align*}
    C_{\alpha_t,\alpha_x,h} &= \left( \sigma_0 + h \alpha_t - h \frac{|\alpha_x|\mod S}{S} + \left(-\half h, \half h \right) \right) \times Q_{\alpha_x,h} \\
    &=\left( \sigma_0 + h \alpha_t - h \frac{|\alpha_x|\mod S}{S} , h \alpha_x \right) + \left(-\half h, \half h \right)^{1+n} . 
\end{align*} 
One can readily check using \eqref{eq:003} that 
$$
    [\sigma_0,\sigma_1] \subset \interior{\left(\bigcup_{\alpha_t\in \mathcal{J}_t} \left( \sigma_0 + h \alpha_t - h \frac{|\alpha_x|\mod S}{S} + \left[-\half h, \half h \right] \right) \right)}\subset (\ov{\sigma}_0,\ov{\sigma}_1) \qquad \text{ for all }\alpha_x\in \mathcal{J}_x,
$$ 
and hence 
\begin{equation} \label{eq:cubes-in-sets}
    [\sigma_0,\sigma_1] \times \widetilde{\Omega} \subset \interior{\left( \bigcup_{(\alpha_t,\alpha_x)\in \mathcal{J}_t \times \mathcal{J}_x} \closure{C_{\alpha_t,\alpha_x,h}}\right)} \subset (\ov{\sigma}_0,\ov{\sigma}_1) \times \widetilde{\Omega}.
\end{equation} 
Let us also introduce the abbreviation 
$$
    ( t_{\alpha_t,\alpha_x,h}, x_{\alpha_t,\alpha_x,h} ) := \left( \sigma_0 + h \alpha_t - h \frac{|\alpha_x|\mod S}{S} , h \alpha_x \right)
$$
for the centers of the cubes $C_{\alpha_t,\alpha_x,h}$. 

Now we are ready to apply Lemma~\ref{lemma:pert-prop-step2} for each $(\alpha_t,\alpha_x)\in \mathcal{J}_t\times \mathcal{J}_x$ where we will choose
\begin{align*}
    &\Gamma^\ast = C_{\alpha_t,\alpha_x,h}, \quad \ep_1 = \frac{\ep}{4|\mathcal{J}_x|}, \quad \ep_2 = \frac{h}{4S}, \quad \\
    &\mu^\ast = \theta(t_{\alpha_t,\alpha_x,h}, x_{\alpha_t,\alpha_x,h})\quad\text{ and } \quad \zeta^\ast = \zeta(t_{\alpha_t,\alpha_x,h}, x_{\alpha_t,\alpha_x,h}).
\end{align*} 

This yields for each $(\alpha_t,\alpha_x)\in \mathcal{J}_t\times \mathcal{J}_x$ a sequence of oscillations 
$$
    (\widetilde{\zeta}_{\alpha_t,\alpha_x,k})_{k\in \N} \subset \Cc(C_{\alpha_t,\alpha_x,h};\R^M)
$$ 
with the properties stated in Lemma~\ref{lemma:pert-prop-step2}.

We define 
$$
	\zeta_{\pert} := \zeta + \sum_{(\alpha_t,\alpha_x)\in \mathcal{J}_t\times \mathcal{J}_x} \widetilde{\zeta}_{\alpha_t,\alpha_x,k},
$$
where $k\in\N$ is chosen sufficiently large, such that \eqref{eq:pert-prop1} holds. Here we have used property~\ref{item:pert-prop-step2-a} in Lemma~\ref{lemma:pert-prop-step2}.

It remains to prove that $\zeta_\pert\in X_0$ and that \eqref{eq:pert-prop2}, \eqref{eq:pert-prop3} are satisfied. Let us begin with the former. From Lemma~\ref{lemma:pert-prop-step2}~\ref{item:pert-prop-step2-b} and \eqref{eq:cubes-in-sets} we observe that $\zeta_\pert$ fulfills \eqref{eq:defn-X0-eq}, \eqref{eq:defn-X0-boundary-space}-\eqref{eq:defn-X0-boundary-time-end}, and also \eqref{eq:pert-prop3}. So we only have to check \eqref{eq:defn-X0-subs}. Since this can be done exactly as in \cite[Proof of Prop.~2.18]{Markfelder24}, we skip the details here.

Finally we show \eqref{eq:pert-prop2}. To this end let $s\in [\sigma_0,\sigma_1]$. We estimate
\begin{align}
    &\int_{\Omega} D\big(\theta(s,x),\zeta_\pert(s,x)\big) \dx \notag\\
    &= \int_{\Omega_\tf} D\big(\theta(s,x),\zeta_\pert(s,x)\big) \dx + \sum_{(\alpha_t,\alpha_x)\in \mathcal{J}_t\times \mathcal{J}_x} \int_{C_{\alpha_t,\alpha_x,h} \cap \{t=s\}} D\big(\theta(s,x),\zeta_\pert(s,x)\big) \dx \notag \\
    &\leq  \int_{\Omega_\tf} D\big(\theta(s,x),\zeta_\pert(s,x)\big) \dx \label{eq:004}\\
    &\qquad + \sum_{(\alpha_t,\alpha_x)\in \mathcal{J}_t\times \mathcal{J}_x} \int_{C_{\alpha_t,\alpha_x,h} \cap \{t=s\}} D\big(\theta( t_{\alpha_t,\alpha_x,h}, x_{\alpha_t,\alpha_x,h} ),\zeta( t_{\alpha_t,\alpha_x,h}, x_{\alpha_t,\alpha_x,h} ) + \widetilde{\zeta}_{\alpha_t,\alpha_x,h}(s,x)\big) \dx \notag \\
    &\qquad + \sum_{(\alpha_t,\alpha_x)\in \mathcal{J}_t\times \mathcal{J}_x} \int_{C_{\alpha_t,\alpha_x,h} \cap \{t=s\}} \Big| D\big(\theta(s,x),\zeta(s,x) + \widetilde{\zeta}_{\alpha_t,\alpha_x,h} (s,x)\big) - \notag \\
    &\qquad \qquad \qquad \qquad \quad \qquad \qquad D\big(\theta( t_{\alpha_t,\alpha_x,h}, x_{\alpha_t,\alpha_x,h} ),\zeta( t_{\alpha_t,\alpha_x,h}, x_{\alpha_t,\alpha_x,h} ) + \widetilde{\zeta}_{\alpha_t,\alpha_x,h}(s,x)\big) \Big| \dx. \notag
\end{align}
We consider each summand in \eqref{eq:004} separately. 

Using \eqref{eq:step2-est1} we find that 
\begin{equation} \label{eq:004-est1}
    \int_{\Omega_\tf} D\big(\theta(s,x),\zeta_\pert(s,x)\big) \dx \leq \ov{D} |\Omega_\tf| \leq \frac{\ep}{4}.
\end{equation}

By virtue of \eqref{eq:unif-cont-K} we have 
\begin{align}
    &\sum_{(\alpha_t,\alpha_x)\in \mathcal{J}_t\times \mathcal{J}_x} \int_{C_{\alpha_t,\alpha_x,h} \cap \{t=s\}} \Big| D\big(\theta(s,x),\zeta(s,x) + \widetilde{\zeta}_{\alpha_t,\alpha_x,h} (s,x)\big) - \notag \\
    &\qquad \qquad \qquad \qquad \qquad D\big(\theta( t_{\alpha_t,\alpha_x,h}, x_{\alpha_t,\alpha_x,h} ),\zeta( t_{\alpha_t,\alpha_x,h}, x_{\alpha_t,\alpha_x,h} ) + \widetilde{\zeta}_{\alpha_t,\alpha_x,h}(s,x)\big) \Big| \dx \notag \\
    &\leq \frac{\ep}{4|\widetilde{\Omega}|} \sum_{(\alpha_t,\alpha_x)\in \mathcal{J}_t\times \mathcal{J}_x} \int_{C_{\alpha_t,\alpha_x,h} \cap \{t=s\}} 1 \dx \ \leq \ \frac{\ep}{4}, \label{eq:004-est2}
\end{align}
where we have used that 
\begin{equation} \label{eq:005}
    \sum_{(\alpha_t,\alpha_x)\in \mathcal{J}_t\times \mathcal{J}_x} \int_{C_{\alpha_t,\alpha_x,h} \cap \{t=s\}} 1 \dx \leq |\widetilde{\Omega}|
\end{equation}
which results from \eqref{eq:aux06} and the fact that for each $\alpha_x\in\mathcal{J}_x$ there exists at most one $\alpha_t\in \mathcal{J}_t$ such that $C_{\alpha_t,\alpha_x,h} \cap \{t=s\}\neq \emptyset$. 

To estimate the remaining term, we first observe that by construction there are at most\footnote{Note that $\frac{|\mathcal{J}_x|}{S}$ is indeed a natural number due to the above choice of $h$.} $\frac{|\mathcal{J}_x|}{S}$ many pairs $(\alpha_t,\alpha_x)\in \mathcal{J}_t \times \mathcal{J}_x$ for the fixed time $s$ for which the estimate in Lemma~\ref{lemma:pert-prop-step2}~\ref{item:pert-prop-step2-d} does not hold. For the other pairs $(\alpha_t,\alpha_x)$ we can apply the latter estimate. This argumentation leads to 
\begin{align}
    &\sum_{(\alpha_t,\alpha_x)\in \mathcal{J}_t\times \mathcal{J}_x} \int_{C_{\alpha_t,\alpha_x,h} \cap \{t=s\}} D\big(\theta( t_{\alpha_t,\alpha_x,h}, x_{\alpha_t,\alpha_x,h} ),\zeta( t_{\alpha_t,\alpha_x,h}, x_{\alpha_t,\alpha_x,h} ) + \widetilde{\zeta}_{\alpha_t,\alpha_x,h}(s,x)\big) \dx \notag\\
    &\leq \frac{|\mathcal{J}_x|}{S} \ov{D} h^n + |\mathcal{J}_x| \frac{\ep}{4|\mathcal{J}_x|} \ \leq \ \frac{\ep}{2}, \label{eq:004-est3}
\end{align}
where we have used a similar argument as in \eqref{eq:005}, and \eqref{eq:aux06} and \eqref{eq:step2-est2}.

Plugging \eqref{eq:004-est1}, \eqref{eq:004-est2} and \eqref{eq:004-est3} into \eqref{eq:004} we find \eqref{eq:pert-prop2} as desired. 
\end{proof}

\subsection{Construction of a subsolution with initial data taking values in $\sK_\theta$} \label{subsec:general-idiK}

This section is devoted to the construction of a subsolution which takes values in $\sK_\theta$ at initial time. This can be viewed as a generalization of \cite[Section~5]{DelSze10} to the general system \eqref{eq:lin-eq}. To proceed, we have to require that the suitable distance map $D$ has the following additional property:
\begin{enumerate} \setcounter{enumi}{5}
    \item \label{item:suitable-D-convergence} If $\widehat{\theta}\in C^1(\closure{\Omega};\Theta)$, $\widehat{\zeta}_k,\widehat{\zeta}\in L^r(\Omega;\closure{B}_c(0))$ for $k\in \N$, and $\widehat{\zeta}_k\to \widehat{\zeta}$ strongly in $L^r$, then 
    $$
        \lim_{k\to \infty} \int_\Omega D\big(\widehat{\theta}(x), \widehat{\zeta}_k(x) \big) \dx = \int_\Omega D\big(\widehat{\theta}(x), \widehat{\zeta}(x) \big) \dx.
    $$
\end{enumerate}

The following theorem modifies a given subsolution $\ov{\zeta}$, which does not need to take values in $\sK_\theta$ at initial time, into a new subsolution $\zeta$, which takes values in $\sK_\theta$ at some time $t=T_0$. Here $T_0\in (0,T)$ may be chosen arbitrarily and will play the role of the new initial time. Thm.~\ref{thm:idiK} will be used to prove the results of type Thm.~\ref{thm:general-wilddata} below.

\begin{thm} \label{thm:idiK} 
    Let $T\in (0,\infty)\cup\{\infty\}$, $\Omega\subset\R^n$ bounded, $r\in (1,\infty)$, $a,b_1,b_2:\R^M\to \R^m$ and $A,B:\R^M\to \R^{m\times n}$ linear, $\theta\in \Cb([0,T]\times \closure{\Omega};\R^m)\cap C^1([0,T]\times \closure{\Omega};\R^m)$, $(\sK_{\mu})_{\mu\in \Theta}$, $D\in C(\Theta\times \ov{B}_c(0);\R)$ and $\ov{\zeta}\in \Cweak([0,T];L^r(\Omega;\R^M)) \cap C^1((0,T)\times \closure{\Omega};\R^M)$ be given such that the assumptions of Thm.~\ref{thm:conv-int} hold. We assume that $D$ satisfies the additional property \ref{item:suitable-D-convergence} above. Let furthermore $T_0\in (0,T)$ and $\beta>0$. 
    
    Then there exists
    $$
        \zeta\in \Cweak([0,T];L^r(\Omega;\R^M)) \cap C^1(((0,T)\setminus\{T_0\})\times \closure{\Omega};\R^M)
    $$ 
    with the following properties: 
    \begin{enumerate}
		\item \label{item:idiK-pde} It satisfies the linear system of PDEs \eqref{eq:lin-eq}, i.e. 
		$$ 
            \partial_t a(\zeta) + \Div A (\zeta) = \partial_t b_1(\theta) + \Div B(\theta) + b_2(\theta) 
		$$
		holds pointwise for all $t\in (0,T)\setminus\{T_0\}$ and all $x \in \Omega$.
        
		\item \label{item:idiK-valuesinUK} It takes values in $\sU_{\theta(t,x)}$ if $t\neq T_0$ and in $\sK_{\theta(t,x)}$ for $t=T_0$, i.e. 
		\begin{align}
		    \zeta(t,x) &\in \sU_{\theta(t,x)}\qquad \text{ for all }t\in (0,T)\setminus\{T_0\} \text{ and all }x\in \Omega , \label{eq:idiK-valuesinU} \\
            \zeta(T_0,x) &\in \sK_{\theta(T_0,x)}\qquad \text{ for a.e. }x\in \Omega . \label{eq:idiK-valuesinK}
		\end{align} 

        \item \label{item:idiK-boundary+outside} On the boundary $\partial\Omega$ and outside the temporal interval $(T_0-\beta,T_0+\beta)$, it coincides with $\ov{\zeta}$, i.e. 
        \begin{align}
            \zeta(t,x) &= \ov{\zeta}(t,x) \qquad \text{ for all }t\in (0,T)\setminus\{T_0\} \text{ and all }x\in \partial\Omega , \label{eq:idiK-boundary} \\ 
            \zeta(t,x) &= \ov{\zeta}(t,x) \qquad \text{ for all }t\in (0,T)\setminus(T_0-\beta,T_0+\beta) \text{ and all }x\in \Omega. \label{eq:idiK-outside}
        \end{align} 
	\end{enumerate}
\end{thm}

The proof of Thm.~\ref{thm:idiK} follows \cite[Section~5]{DelSze10}.

\begin{proof}
We may assume without loss of generality that $0<T_0-\beta < T_0+\beta < T$ (by shrinking $\beta$ if necessary). We set $\beta_k:= \frac{\beta}{2^k}$ for $k\in \N$. 

Next we construct a sequence $(\zeta_k)_{k\in \N}$ in $X_0$ as follows. We define $\zeta_1:= \ov{\zeta}$. Having fixed $\zeta_1,...,\zeta_k$ for some $k\in \N$, we will now construct $\zeta_{k+1}$. To this end, we first define $0<\widehat{\ep}_k < \left( \half \right)^k$ such that 
\begin{equation} \label{eq:idiK-est-1}
    \| \zeta_k (T_0,\cdot) - \zeta_k (T_0,\cdot) \ast \phi_{\widehat{\ep}_k} \|_{L^r(\Omega;\R^M)} \leq \left( \half \right)^k , 
\end{equation}
where $\phi\in C^\infty(\R^n;\R^+_0)$ is the standard $n$-dimensional Friedrichs mollifier. To obtain $\zeta_{k+1}$, we apply Prop.~\ref{prop:pert-prop} with 
\begin{align*}
    \ep &= \left( \half \right)^k, \\
    \sigma_0 &= T_0 - \beta_{k+1}, & \sigma_1 &= T_0 + \beta_{k+1}, \\
    \ov{\sigma}_0 &= T_0 - \beta_{k}, & \ov{\sigma}_1 &= T_0 + \beta_{k},
\end{align*}
and $\ov{\ep}\leq \left( \half \right)^k$ small enough such that \eqref{eq:pert-prop1} implies 
\begin{equation} \label{eq:idiK-est-2}
    \left\| \big(\zeta_k (T_0,\cdot) - \zeta_{k+1} (T_0,\cdot)\big) \ast \phi_{\widehat{\ep}_\ell} \right\|_{L^r(\Omega;\R^M)} \leq \left( \half \right)^k \qquad \text{ for all }\ell\leq k.
\end{equation}
Hence, according to Prop.~\ref{prop:pert-prop} we have 
\begin{align}
    \zeta_{k+1} &\in X_0, \label{eq:idiK-constseq-X0} \\
    d( \zeta_{k+1}, \zeta_k) &\leq \left( \half \right)^k, \label{eq:idiK-constseq-d} \\
    \I_{T_0 - \beta_{k+1}, T_0 + \beta_{k+1}}(\zeta_{k+1} ) & \leq \left( \half \right)^k, \label{eq:idiK-constseq-I} \\
    \zeta_{k+1} &= \zeta_k \qquad \text{ for all }t\in (0,T_0 - \beta_{k})\cup (T_0 + \beta_{k}, T) \text{ and all } x\in \Omega. \label{eq:idiK-constseq-outside}
\end{align}
So we have constructed a sequence $(\zeta_k)_{k\in \N}$, which lies in $X_0$ according to \eqref{eq:idiK-constseq-X0}.

From \eqref{eq:idiK-constseq-d} we deduce that $(\zeta_k)_{k\in \N}$ is a Cauchy sequence in the complete metric space $(X,d)$ (cf.~Lemma~\ref{prop:X-metric-space}). Thus there exists $\zeta\in X$ with $\zeta_k \mathop{\to}\limits^d \zeta$. 

Let $t\in (0,T)\setminus\{T_0\}$ and $x \in \closure{\Omega}$. Due to \eqref{eq:idiK-constseq-outside}, there exists $\ov{k}\in \N$ such that $\zeta \equiv \zeta_{\ov{k}}$ on a small neighbourhood of $(t,x)$. Consequently $\zeta\in C^1(((0,T)\setminus\{T_0\})\times \closure{\Omega};\R^M)$, and item~\ref{item:idiK-pde} of Thm.~\ref{thm:idiK} and \eqref{eq:idiK-valuesinU} hold. Property \eqref{eq:idiK-outside} is shown in a similar fashion. In order to prove \eqref{eq:idiK-boundary}, we notice that $\zeta_{\ov{k}}\equiv\ov{\zeta}$ on $\partial\Omega$ due to the fact that $\zeta_{\ov{k}}\in X_0$. So we have shown item~\ref{item:idiK-boundary+outside} of Thm.~\ref{thm:idiK}.

It remains to prove \eqref{eq:idiK-valuesinK}. We obtain from \eqref{eq:idiK-est-2} that 
\begin{align*}
    \left\| \big(\zeta_k (T_0,\cdot) - \zeta (T_0,\cdot) \big) \ast \phi_{\widehat{\ep}_k} \right\|_{L^r(\Omega;\R^M)} &\leq \sum_{\ell=0}^\infty \left\| \big(\zeta_{k+\ell} (T_0,\cdot) - \zeta_{k+\ell+1} (T_0,\cdot) \big) \ast \phi_{\widehat{\ep}_k} \right\|_{L^r(\Omega;\R^M)} \\ 
    &\leq \sum_{\ell=0}^\infty \left( \half \right)^{k+\ell} \ = \ 2 \left( \half \right)^k.
\end{align*}
Combining the latter with \eqref{eq:idiK-est-1} and 
\begin{align*}
    \| \zeta_k (T_0,\cdot) - \zeta (T_0,\cdot) \|_{L^r(\Omega;\R^M)} &\leq \| \zeta_k (T_0,\cdot) - \zeta_k (T_0,\cdot) \ast \phi_{\widehat{\ep}_k} \|_{L^r(\Omega;\R^M)} \\
    & \qquad + \left\| \big(\zeta_k (T_0,\cdot) - \zeta (T_0,\cdot) \big) \ast \phi_{\widehat{\ep}_k} \right\|_{L^r(\Omega;\R^M)} \\
    & \qquad + \| \zeta (T_0,\cdot) \ast \phi_{\widehat{\ep}_k} - \zeta (T_0,\cdot) \|_{L^r(\Omega;\R^M)},
\end{align*}
we find that $\zeta_k (T_0,\cdot) \to \zeta (T_0,\cdot)$ strongly in $L^r$. Thus property \ref{item:suitable-D-convergence} of the suitable distance map $D$ shows that 
$$
    \lim_{k\to \infty} \int_\Omega D\big(\theta(T_0,x), \zeta_k(T_0,x) \big) \dx = \int_\Omega D\big(\theta(T_0,x), \zeta(T_0,x) \big) \dx.
$$

On the other hand, \eqref{eq:idiK-constseq-I} implies 
$$
    \int_\Omega D\big(\theta(T_0,x), \zeta_{k+1}(T_0,x) \big) \dx \leq \left( \half \right)^k, 
$$
and consequently\footnote{To be precise, we also use Lemma~\ref{lemma:suitable-distance-map}~\ref{item:L-suitable-D-geq0} here.}
$$
    \int_\Omega D\big(\theta(T_0,x), \zeta(T_0,x) \big) \dx = 0 .
$$
Arguing as in the proof of Lemma~\ref{lemma:properties-I}~\ref{item:prop-I-sol}, we deduce that 
$$
    D\big(\theta(T_0,x), \zeta(T_0,x) \big) = 0 \qquad \text{ for a.e. }x\in \Omega,
$$
and, by Defn.~\ref{defn:suitable-distance-map}~\ref{item:suitable-D-inK}, we conclude with \eqref{eq:idiK-valuesinK}.
\end{proof}

\section{Incompressible Euler equations} \label{sec:ex-incomp-euler}

As a first example, we consider the incompressible Euler equations \eqref{eq:incomp-euler-mass}, \eqref{eq:incomp-euler-mom}. The goal of this section is to show that the general framework established in Sect.~\ref{sec:general} covers the convex integration approach by \name{De~Lellis}-\name{Sz{\'e}kelyhidi}~\cite{DelSze09,DelSze10}. In other words we will show that the structural assumptions of Thm.~\ref{thm:conv-int} hold in the context of the incompressible Euler equations \eqref{eq:incomp-euler-mass}, \eqref{eq:incomp-euler-mom}, see Sect.~\ref{subsec:ex-incomp-euler-ass} below. Thus Thm.~\ref{thm:conv-int} turns into a version of \cite[Prop.~2]{DelSze10} for bounded domains\footnote{Again see Rem.~\ref{rem:unbounded-domains} for a way how to generalize Thm.~\ref{thm:conv-int} to unbounded domains.} $\Omega\subset \R^n$. 

This will then immediately lead to a proof of Thms.~\ref{thm:incomp-euler-exforalldata} and \ref{thm:incomp-euler-wilddata}, see Sect.~\ref{subsec:ex-incomp-euler-pf} below. In particular our framework established in Sect.~\ref{sec:general} is able to recover and extend the results which were found by \name{De~Lellis}-\name{Sz{\'e}kelyhidi}~\cite{DelSze10} and \name{Wiedemann}~\cite{Wiedemann11}. Other results in the literature which are based on \cite[Prop.~2]{DelSze10} can be recovered as well (e.g.~\cite{SzeWie12}, see Rem.~\ref{rem:incomp-euler-furtherresults}), but we will not present any details here. The main new result (that is not the literature) that we will prove in this section is the global existence and non-uniqueness of weak solutions for all $L^\infty (\mathbb{T}^3)$ data, which implies the ill-posedness of the incompressible Euler equations in that setting.

\subsection{Preliminaries and reformulation of the problem} \label{subsec:ex-incomp-euler-prel}

As originally carried out by \name{De~Lellis}-\name{Sz{\'e}kelyhidi}~\cite{DelSze09}, we rewrite the non-linear system \eqref{eq:incomp-euler-mass}, \eqref{eq:incomp-euler-mom} in terms of \name{Tartar}~\cite{Tartar79}, see also Rem.~\ref{rem:tartars-framework}. We introduce new unknowns $U$ and $q$ which take values in\footnote{Here $\symz{n}$ denotes the space of all symmetric $n\times n$-matrices which have zero trace.} $\symz{n}$ and $\R$, respectively. This way the non-linear system \eqref{eq:incomp-euler-mass}, \eqref{eq:incomp-euler-mom} turns into the linear one\footnote{For $n\in \N$, $\id_n$ denotes the $n\times n$ identity matrix.}
\begin{align}
    \Div v &= 0, \label{eq:lin1-incomp-euler-mass} \\
    \partial_t v + \Div (U+q\id_n) &= 0, \label{eq:lin1-incomp-euler-mom}
\end{align}
and we define the constitutive set
\begin{equation} \label{eq:const1-incomp-euler}
    \sK := \left\{ (v,p,U,q)\in \R^n \times \R \times \symz{n} \times \R \,\Big|\, v\otimes v - U + (p-q)\id_n = 0 \right\}.
\end{equation}
Instead of solving the non-linear system \eqref{eq:incomp-euler-mass}, \eqref{eq:incomp-euler-mom}, we can equivalently look for solutions of the linear system \eqref{eq:lin1-incomp-euler-mass}, \eqref{eq:lin1-incomp-euler-mom} which take values in $\sK$ as defined in \eqref{eq:const1-incomp-euler}.

As was done in the literature (see e.g.~\cite{DelSze10}), we treat $p$ and $q$ as parameters. With the notation of Sect.~\ref{sec:general} this means that we have $\theta=(p,q)$, and consequently $\zeta=(v,U)$. Therefore we rewrite \eqref{eq:lin1-incomp-euler-mass}, \eqref{eq:lin1-incomp-euler-mom} in the form \eqref{eq:lin-eq}, i.e. 
\begin{align}
    \Div v &= 0, \label{eq:lin-incomp-euler-mass} \\
    \partial_t v + \Div U &= - \Grad q. \label{eq:lin-incomp-euler-mom}
\end{align}
Moreover, the set in \eqref{eq:const1-incomp-euler} turns into the family of constitutive sets 
\begin{equation} \label{eq:const-incomp-euler}
    \sK_{p,q} := \left\{ (v,U)\in \R^n \times \symz{n} \,\Big|\, v\otimes v - U + (p-q)\id_n = 0 \right\}.
\end{equation}

Assuming that $\theta=(p,q)$ has the regularity required in Thm.~\ref{thm:conv-int}, we may suppose that $(p,q)$ takes values in a compact set $\Theta$. Moreover, without loss of generality we have $q\geq p$ on $\Theta$ as otherwise $\sK_{p,q}=\emptyset$ and thus existence of $\ov{\zeta}$ as required in Thm.~\ref{thm:conv-int} is impossible. 

In the setting we just introduced, the wave cone (see Defn.~\ref{defn:wave-cone}) reads 
\begin{align*}
    \Lambda = \bigg\{ (v,U)\in \R^n\times \symz{n}\, \Big|\, &\exists\eta=(\eta_t,\eta_x)\in \R^{1+n}\ \text{ with }\ \eta_x \neq 0 \ \text{ and } \ \\
    & \qquad \qquad \left(\begin{array}{cc} 0 & v^\trans \\ v & U \end{array}\right)\cdot\eta = 0 \quad\bigg\} .
\end{align*}

\subsection{Structural assumptions of Thm.~\ref{thm:conv-int}} \label{subsec:ex-incomp-euler-ass}

Let us next check the structural assumptions of Thm.~\ref{thm:conv-int}. We summarize them in the following lemma.

\begin{lemma} \label{lemma:ass-incomp-euler}
    \begin{enumerate}
        \item \label{item:ass1-incomp-euler} The family of constitutive sets $(\sK_{p,q})_{(p,q)\in \Theta}$ defined in \eqref{eq:const-incomp-euler} is suitable in the sense of Defn.~\ref{defn:suitable-K}. 

        \item \label{item:ass2-incomp-euler} Let $(v,U)\in\Lambda$. Then there exists a third order homogeneous differential operator
        $$
		      \opL_{(v,U)} : C^\infty(\R^{1+n}) \to C^\infty(\R^{1+n};\R^n\times \symz{n}) 
        $$
        which is suitable in the sense of Defn.~\ref{defn:suitable-operator}. 

        \item \label{item:ass3-incomp-euler} It holds that 
        $$
            (\sK_{p,q})^\Lambda = (\sK_{p,q})^\co \qquad \text{ for all } (p,q) \in \Theta.
        $$ 

        \item \label{item:ass4-incomp-euler} There exists a map $D\in C(\Theta\times \ov{B}_c(0);\R)$ which is a suitable distance map in the sense of Defn.~\ref{defn:suitable-distance-map}. Moreover if $r=2$, then $D$ satisfies the additional property~\ref{item:suitable-D-convergence} introduced in Sect.~\ref{subsec:general-idiK}.
    \end{enumerate}
\end{lemma}

\begin{proof} 
\begin{enumerate}
    \item We have to check the properties in Defn.~\ref{defn:suitable-K} (denoted in sequel by ``Prop.~(a)'', ``Prop.~(b)'' and ``Prop.~(c)'').
    \begin{itemize}
        \item[Prop.~(a)] It is obvious that the $\sK_{p,q}$ are closed. Their boundedness was shown in \cite[Lemma~3]{DelSze10}.

        \item[Prop.~(b)] Let $\ep>0$ given. Note that on $\sK_{p,q}$ it holds that $|v|^2 = n (q-p)$, which can be simply verified by taking the trace of the equation in \eqref{eq:const-incomp-euler}. This implies together with the fact that $\Theta$ is compact, that there exists $c_v>1$ such that $|v|\leq c_v$ for all $(v,U)\in \sK_{p,q}$ and all $(p,q)\in \Theta$.
    
        Now choose $\delta>0$ such that $\delta<\frac{\ep}{3n}$ and 
        $$
            \Big| \sqrt{q_1-p_1} - \sqrt{q_2-p_2} \Big| < \frac{\ep}{6 \sqrt{n} c_v}
        $$
        for all $(p_1,q_1),(p_2,q_2)\in \Theta$ with $|(p_1,q_1)-(p_2,q_2)| < \delta$. Note that the latter is possible by uniform continuity of the square root. 

        Next consider $|(p_1,q_1)-(p_2,q_2)| < \delta$ and $(v_1,U_1)\in \sK_{p_1,q_1}$. If $v_1=0$, simply choose $v_2\in \R^n$ such that $|v_2|^2 = n (q_2-p_2)$. In the case $v_1\neq 0$, set 
        $$
            v_2 := \sqrt{ n (q_2-p_2)} \frac{v_1}{|v_1|}.
        $$
        In both cases define 
        $$
            U_2 := v_2 \otimes v_2 + (p_2-q_2) \id_n. 
        $$
        Then $(u_2,U_2)\in \sK_{p_2,q_2}$ by construction. It remains to prove that $|(u_1,U_1)-(u_2,U_2)| <\ep$.

        We find that in both cases 
        \begin{align*}
            |v_1-v_2| &= \sqrt{n} \Big| \sqrt{q_1-p_1} - \sqrt{q_2-p_2} \Big| < \frac{\ep}{6 c_v} < \ep, \\
            |U_1-U_2| &= \Big| v_1 \otimes v_1 - v_2 \otimes v_2 + (p_1-p_2-q_1+q_2) \id_n \Big| \\
            &\leq |v_1-v_2| \Big( |v_1| + |v_2| \Big) + n |p_1-p_2| + n|q_1- q_2| \\
            &< \frac{\ep}{6 c_v} 2c_v + \frac{\ep}{3} + \frac{\ep}{3} = \ep. 
        \end{align*}

        \item[Prop.~(c)] As shown in \cite[Lemma~3]{DelSze10} (see also \cite[Lemma~4.3.6]{Markfelder}), it holds that\footnote{We write $\lambda_{\max}(A)$ for the largest eigenvalue of a symmetric matrix $A\in \R^{n\times n}$.} 
        \begin{equation} \label{eq:convhull-incomp-euler}
            (\sK_{p,q})^\Lambda = (\sK_{p,q})^\co = \left\{ (v,U)\in \R^n \times \symz{n} \,\Big|\, \lambda_{\max} \big(v\otimes v - U + (p-q)\id_n\big) \leq 0 \right\} 
        \end{equation}
        for all $(p,q) \in \Theta$. From that it is not difficult to see that 
        \begin{equation} \label{eq:interior-convhull-incomp-euler}
            \interior{\big((\sK_{p,q})^\co\big)} = \left\{ (v,U)\in \R^n \times \symz{n} \,\Big|\, \lambda_{\max} \big(v\otimes v - U + (p-q)\id_n\big) < 0 \right\},
        \end{equation}
        see e.g.~\cite[Lemma~5.1.1]{Markfelder} for a detailed proof. Thus we have 
        $$
            \interior{\big((\sK_{p,q})^\co\big)}\cap \sK_{p,q} = \emptyset \qquad \text{ for all }(p,q) \in \Theta
        $$    
        as desired.
    \end{itemize}

    \item The claim was proven in \cite[Prop.~4.4.1~(c)]{Markfelder}.

    \item As mentioned above (see equ.~\eqref{eq:convhull-incomp-euler}), the claim was shown in \cite[Lemma~3]{DelSze10}, see also \cite[Lemma~4.3.6]{Markfelder}.

    \item Inspired by \cite{DelSze10} we choose $D((p,q),(v,U)):= n(q-p) - |v|^2$. Then obviously item~\ref{item:suitable-D-concave} of Defn.~\ref{defn:suitable-distance-map} holds. From the explicit form of $\sK_{p,q}$ (see \eqref{eq:const-incomp-euler}) we infer property \ref{item:suitable-D-=0}. For item~\ref{item:suitable-D-inK} we again refer to \cite[Lemma~3]{DelSze10}. Finally we note that property~\ref{item:suitable-D-convergence} follows immediately from the definition of $D$. 
\end{enumerate}
\end{proof}

\subsection{Proof of Thms.~\ref{thm:incomp-euler-exforalldata} and \ref{thm:incomp-euler-wilddata}} \label{subsec:ex-incomp-euler-pf}

As observed in Sect.~\ref{subsec:ex-incomp-euler-ass}, the structural assumptions of Thm.~\ref{thm:conv-int} hold in the context of the incompressible Euler equations \eqref{eq:incomp-euler-mass}, \eqref{eq:incomp-euler-mom}, i.e.~Thm.~\ref{thm:conv-int} can be applied to these equations. Moreover, Thm.~\ref{thm:idiK} can be applied to the to the incompressible Euler system too, which was originally carried out in \cite[Sect.~5]{DelSze10}. 

Now, with Thms.~\ref{thm:conv-int} and \ref{thm:idiK} at hand, it is straightforward to prove Thms.~\ref{thm:incomp-euler-exforalldata} and \ref{thm:incomp-euler-wilddata}:

\begin{proof}[Proof of Thm.~\ref{thm:incomp-euler-exforalldata}]
The goal is to find a subsolution for the initial data $v_0\in L^\infty(\T^3;\R^3)$ and then apply Thm.~\ref{thm:conv-int}. Such a subsolution $(\ov{v},\ov{U})$ is generated in Prop.~\ref{prop:app-sfad}. Indeed, the boundedness of $(\ov{v},\ov{U})$ (see Prop.~\ref{prop:app-sfad}~\ref{item:sfad-proposition-c}) ensures that that there are constants $p,q$ such that $(\ov{v},\ov{U})$ takes values in $\sU_{p,q}=\interior{\big( (\sK_{p,q})^\Lambda\big)}$, see \eqref{eq:convhull-incomp-euler} and \eqref{eq:interior-convhull-incomp-euler}. As $p,q$ are constant, property \ref{item:sfad-proposition-a} from Prop.~\ref{prop:app-sfad} shows that $(\ov{v},\ov{U})$ solves the required PDEs \eqref{eq:lin-incomp-euler-mass}, \eqref{eq:lin-incomp-euler-mom}. Finally, Thm.~\ref{thm:conv-int} yields infinitely many weak solutions as stated in Thm.~\ref{thm:incomp-euler-exforalldata}, and property \ref{item:sfad-proposition-b} from Prop.~\ref{prop:app-sfad} proves that these solutions satisfy the right initial condition. 
\end{proof}

\begin{proof}[Proof of Thm.~\ref{thm:incomp-euler-wilddata}]
Choose $(\ov{v},\ov{U})\equiv (0,0)$, $p,q\in\R$ constant such that $q>p$. Then $(\ov{v},\ov{U})$ is a subsolution, i.e.~it satisfies the required PDEs \eqref{eq:lin-incomp-euler-mass}, \eqref{eq:lin-incomp-euler-mom}, and it takes values in $\sU_{p,q}=\interior{\big( (\sK_{p,q})^\Lambda\big)}$. Hence we may apply Thm.~\ref{thm:idiK} with $T_0>0$ arbitrary to obtain a new subsolution $(\ov{\ov{v}},\ov{\ov{U}})$ which (after shifting in time by $T_0$) takes values in $\sK_{p,q}$ for $t=0$ and in $\sU_{p,q}$ for $t>0$. 

Now we apply Thm.~\ref{thm:conv-int} to obtain infinitely many weak solutions which not only take values in $\sK_{p,q}$ for $t>0$, but even at $t=0$. Thus they satisfy 
$$
    \half |v(t,x)|^2 = \tfrac{n}{2}(q-p) \qquad \text{ for all }t\in [0,T)\text{ and a.e. }x\in \Omega,
$$
which can be verified by taking the trace in \eqref{eq:const-incomp-euler}. Consequently the solutions are admissible in the sense of Defn.~\ref{defn:incomp-euler-sol}~\ref{item:incomp-euler-sol-adm} as desired. 

It remains to prove that $\ov{\ov{v}}(T_0,\cdot)$ (before shifting in time) satisfies the compatibility condition \eqref{eq:incomp-euler-compatibility-initialdata} such that we may use it as initial datum. According to Thm.~\ref{thm:idiK} we have 
$$
    \Div \ov{\ov{v}} = 0 \qquad \text{ pointwise on }\big((0,T)\setminus\{T_0\}\big) \times \Omega.
$$
Hence for all test functions $\phi\in\Cc(\closure{\Omega})$ and $\psi\in \Cc((0,T))$ we have 
\begin{align}
    0 &= -\int_0^T \int_{\Omega} \Div \ov{\ov{v}}(t,x) \, \phi(x) \, \psi(t) \dx\dt \notag \\
    &= \int_0^T \int_{\Omega} \ov{\ov{v}}(t,x) \cdot \Grad \phi(x) \, \psi(t) \dx\dt - \int_0^T \int_{\partial\Omega} \ov{\ov{v}} \cdot \nu \, \phi \, \psi \dS\dt \notag \\
    &= \int_0^T \left( \int_{\Omega} \ov{\ov{v}}(t,x) \cdot \Grad \phi(x) \dx \right) \psi(t) \dt, \label{eq:div-equation-continous}
\end{align}
where we used integration by parts, and the fact that $\ov{\ov{v}}$ coincides with $\ov{v}$ on $\partial\Omega$ and thus the boundary term cancels. Since the map $t\mapsto \int_{\Omega} \ov{\ov{v}} \cdot \Grad \phi\dx$ is continuous on $(0,T)$ because $\ov{\ov{v}}\in \Cweak((0,T);L^2(\Omega;\R^n))$, we deduce from \eqref{eq:div-equation-continous} that %
$$
    \int_{\Omega} \ov{\ov{v}}(t,x) \cdot \Grad \phi(x) \dx = 0 \qquad \text{ for all }t\in (0,T),
$$
in particular for $t=T_0$.
\end{proof}

\section{Barotropic compressible Euler equations} \label{sec:ex-compr-euler} 

Next we apply our general convex-integration-framework to the barotropic compressible Euler equations \eqref{eq:compr-euler-mass}, \eqref{eq:compr-euler-mom} and the shallow water equations \eqref{eq:sw-mass}, \eqref{eq:sw-mom}. The goal is to recover \name{Feireisl}'s result Thm.~\ref{thm:compr-euler-wilddata} (see \cite[Thm.~1.3]{Feireisl14}) and Cor.~\ref{cor:sw-wilddata}.

As already seen in the introduction (see Sect.~\ref{subsubsec:intro-sw}), Cor.~\ref{cor:sw-wilddata} follows from Thm.~\ref{thm:compr-euler-wilddata}. Thus we will only deal with Thm.~\ref{thm:compr-euler-wilddata} in this section in order to cover both cases of the barotropic compressible Euler equations \eqref{eq:compr-euler-mass}, \eqref{eq:compr-euler-mom} and the shallow water equations \eqref{eq:sw-mass}, \eqref{eq:sw-mom}.

\subsection{Preliminaries and reformulation of the problem} \label{subsec:ex-compr-euler-prel}

We proceed as in \name{Feireisl}~\cite{Feireisl14}, see also \cite{Chiodaroli14} and note that this process was already mentioned in \cite{DelSze10}. More precisely we rewrite the barotropic compressible Euler equations \eqref{eq:compr-euler-mass}, \eqref{eq:compr-euler-mom} as 
\begin{align}
    \Div m &= - \partial_t \rho, \label{eq:lin-compr-euler-mass} \\
    \partial_t m + \Div U &= - \Grad q, \label{eq:lin-compr-euler-mom}
\end{align}
and we define the family of constitutive sets 
\begin{equation} \label{eq:const-compr-euler}  
    \sK_{\rho,q} := \left\{ (m,U)\in \R^n \times \symz{n} \,\Big|\, \frac{m\otimes m}{\rho} - U + (p(\rho)-q)\id_n = 0 \right\}.
\end{equation}
With the notation used in Sect.~\ref{sec:general} we may set $\theta=(\rho,q)$ and $\zeta=(m,U)$. Note furthermore, that we work with the momentum $m$ as an unknown rather than the velocity $u$. As mentioned in Sect.~\ref{subsubsec:intro-compr-euler}, we exclude vacuum, i.e.~$\rho$ will always be positive and consequently there is no difficulty when computing the velocity from the momentum via $u=\frac{m}{\rho}$ and vice versa. 

Finally, we observe that the wave cone in the setting of equations \eqref{eq:lin-compr-euler-mass}, \eqref{eq:lin-compr-euler-mom} is given by 
\begin{align*} 
    \Lambda = \bigg\{ (m,U)\in \R^n\times \symz{n}\, \Big|\, &\exists\eta=(\eta_t,\eta_x)\in \R^{1+n}\ \text{ with }\ \eta_x \neq 0 \ \text{ and } \ \\
    & \qquad \qquad \left(\begin{array}{cc} 0 & m^\trans \\ m & U \end{array}\right)\cdot\eta = 0 \quad\bigg\} .
\end{align*}
and thus it coincides with the wave cone in the context of the incompressible Euler equations, see Sect.~\ref{subsec:ex-incomp-euler-prel}.

\subsection{Structural assumptions of Thm.~\ref{thm:conv-int}} \label{subsec:ex-compr-euler-ass}

Also in the context of the barotropic compressible Euler equations, the structural assumptions of Thm.~\ref{thm:conv-int} hold, see the following lemma. 

\begin{lemma} \label{lemma:ass-compr-euler}
    \begin{enumerate}
        \item \label{item:ass1-compr-euler} The family of constitutive sets $(\sK_{\rho,q})_{(\rho,q)\in \Theta}$ defined in \eqref{eq:const-compr-euler} is suitable in the sense of Defn.~\ref{defn:suitable-K}. 

        \item \label{item:ass2-compr-euler} Let $(m,U)\in\Lambda$. Then there exists a third order homogeneous differential operator
        $$
		      \opL_{(m,U)} : C^\infty(\R^{1+n}) \to C^\infty(\R^{1+n};\R^n\times \symz{n}) 
        $$
        which is suitable in the sense of Defn.~\ref{defn:suitable-operator}. 

        \item \label{item:ass3-compr-euler} It holds that 
        $$
            (\sK_{\rho,q})^\Lambda = (\sK_{\rho,q})^\co \qquad \text{ for all } (\rho,q) \in \Theta.
        $$ 

        \item \label{item:ass4-compr-euler} There exists a map $D\in C(\Theta\times \ov{B}_c(0);\R)$ which is a suitable distance map in the sense of Defn.~\ref{defn:suitable-distance-map}. Moreover if $r=2$, then $D$ satisfies the additional property~\ref{item:suitable-D-convergence} introduced in Sect.~\ref{subsec:general-idiK}.
    \end{enumerate}
\end{lemma}

\begin{proof} 
The proof of Lemma~\ref{lemma:ass-compr-euler} can be achieved by modifying the proof of the corresponding lemma in the incompressible case (see Lemma~\ref{lemma:ass-incomp-euler}). We also refer to \cite[Sect.~3.2]{Markfelder24} and \cite[Chap.~4]{Markfelder}. Similar to \eqref{eq:convhull-incomp-euler} and \eqref{eq:interior-convhull-incomp-euler} we find 
\begin{align}
    (\sK_{\rho,q})^\Lambda = (\sK_{\rho,q})^\co &= \left\{ (m,U)\in \R^n \times \symz{n} \,\Big|\, \lambda_{\max} \left(\frac{m\otimes m}{\rho} - U + (p(\rho)-q)\id_n\right) \leq 0 \right\}, \label{eq:convhull-compr-euler} \\ 
    \interior{\big((\sK_{\rho,q})^\co\big)} &= \left\{ (m,U)\in \R^n \times \symz{n} \,\Big|\, \lambda_{\max} \left(\frac{m\otimes m}{\rho} - U + (p(\rho)-q)\id_n\right) < 0 \right\}, \label{eq:interior-convhull-compr-euler} 
\end{align}
for all $(\rho,q) \in \Theta$. Moreover, the suitable distance map in the compressible case reads 
$$
    D((\rho,q),(m,U))= n(q-p(\rho)) - \frac{|m|^2}{\rho}.
$$
\end{proof}

\subsection{Proof of Thm.~\ref{thm:compr-euler-wilddata}} \label{subsec:ex-compr-euler-pf}

The proof of Thm.~\ref{thm:compr-euler-wilddata} follows the same lines as the proof of the corresponding theorem in the incompressible setting (see Thm.~\ref{thm:incomp-euler-wilddata}). It was originally done in\footnote{In contrast to \name{Feireisl}'s proof in \cite{Feireisl14}, the explicit application of the Helmholtz decomposition is not necessary in our proof.} \cite{Feireisl14}.

\begin{proof}[Proof of Thm.~\ref{thm:compr-euler-wilddata}]
We first fix a time $T_0>0$ and a cut-off function $h\in C^\infty([0,\infty);\R^+_0)$ with 
\begin{equation*}
    h(t) = \left\{ \begin{array}{ll} 1 & \text{ for }t\in [0,2T_0] , \\ 0 & \text{ for }t\in [3T_0,\infty) , \end{array} \right.
\end{equation*}
and $h'(t)\leq 0$ for all $t\in [0,\infty)$. Next we set the density to be
$$
    \rho(t,x) := h(t) \rho_0(x) + (1-h(t)) \widetilde{\rho},
$$
where $\widetilde{\rho}:= \int_{\Omega} \rho_0 (x) \dx$. 

Now we choose $\ov{m}:= \Grad \Psi$, where $\Psi$ is the solution of $\Lap \Psi = - \partial_t\rho$. Note that $\Psi \in C^2([0,\infty)\times \Omega)$ according to Schauder estimates, in particular $\ov{m} \in C^1([0,\infty)\times \Omega;\R^n)$. Moreover we set $\ov{U}\equiv 0$ and $q(t,x):= \chi(t) - \partial_t \Psi(t,x)$ with a function $\chi\in \Cb([0,\infty)) \cap C^\infty([0,\infty))$ to be chosen below. 

Next we show that $(\ov{m},\ov{U})$ indeed form a subsolution for the compressible Euler equations with parameters $(\rho,q)$. First we note that $(\rho,q)\in \Cb([0,\infty)\times \Omega)\cap C^1 ([0,\infty)\times \Omega)$. The linear system \eqref{eq:lin-compr-euler-mass}, \eqref{eq:lin-compr-euler-mom} holds obviously. Finally there is a constant $C_1>0$ such that $(\ov{m},\ov{U})$ takes values in $\sU_{\rho,q}=\interior{\big( (\sK_{\rho,q})^\Lambda\big)}$ (see \eqref{eq:interior-convhull-compr-euler}) as soon as we ensure that $\chi(t)\geq C_1$ for all $t\in [0,\infty)$. 

Since the structural assumptions of Thm.~\ref{thm:conv-int} hold due to Lemma~\ref{lemma:ass-compr-euler}, we can apply Thm.~\ref{thm:idiK} to get a new subsolution which (after shifting in time by $T_0$) takes values in $\sK_{\rho,q}$ for $t=0$ and in $\sU_{\rho,q}$ for $t>0$. Then Thm.~\ref{thm:conv-int} yields infinitely many weak solutions which take values in $\sK_{\rho,q}$ even for $t\in [0,\infty)$. 

It remains to show that these solutions are admissible in the sense of Defn.~\ref{defn:compr-euler-sol}~\ref{item:compr-euler-sol-adm}. Taking the trace in \eqref{eq:const-compr-euler}, we find that the solutions satisfy 
\begin{equation} \label{eq:compr-1000-trace}
    \frac{|m(t,x)|^2}{2\rho(t,x)} = \frac{n}{2}\Big(q(t,x)-p(\rho(t,x))\Big) \qquad \text{ for all }t\in [0,T)\text{ and a.e. }x\in \Omega. 
\end{equation}
With this at hand, we check the validity of the energy inequality: 
\begin{align*}
    &\partial_t \left( \frac{|m|^2}{2\rho} + P(\rho) \right) + \Div \left[ \left( \frac{|m|^2}{2\rho} + P(\rho) + p(\rho) \right) \frac{m}{\rho} \right] \\
    &= \partial_t \left( \frac{n}{2}\big(q-p(\rho)\big) + P(\rho) \right) - \left[ \left( \frac{n}{2}\big(q-p(\rho)\big) + P(\rho) + p(\rho) \right) \frac{\partial_t \rho}{\rho} \right] \\
    &\qquad + m\cdot \Grad \left[ \left( \frac{n}{2}\big(q-p(\rho)\big) + P(\rho) + p(\rho) \right) \frac{1}{\rho} \right] \\
    &\leq \frac{n}{2} \partial_t\chi + \ep C_2 \chi^2 + C_3\chi + C_4,
\end{align*}
which suitable constants $C_2,C_3,C_4>0$. Here we made use of the inequality $a\cdot b \leq \half \big( |a|^2 + |b|^2 \big)$ (which holds for all $a,b \in \R^n$) together with \eqref{eq:compr-1000-trace}. Notice that the left-hand side of the above energy inequality has to be understood in the sense of distributions. We also observe that $C_i$ ($i=1,...,4$) only depend on $\un{\rho}$, $\ov{\rho}$ and $p$. 

For $t\geq 3T_0$ (even for $t\geq 2T_0$ due to the aforementioned time shift) we obtain an even better estimate:
$$
    \partial_t \left( \frac{|m|^2}{2\rho} + P(\rho) \right) + \Div \left[ \left( \frac{|m|^2}{2\rho} + P(\rho) + p(\rho) \right) \frac{m}{\rho} \right] \leq \frac{n}{2} \partial_t\chi.
$$

In order to find a suitable $\chi$, we choose $\widetilde{C}>\frac{2}{n}\max\{C_2,C_3/2,C_4\}$, $\ov{T}:=4T_0 \widetilde{C}$ and $\ov{C}:= C_1$, and apply Lemma~\ref{lemma:app-constr-phi} to obtain $\ov{\chi}\in C^\infty([0,\ov{T}])$ with properties \eqref{eq:app-constr-phi-p1}, \eqref{eq:app-constr-phi-p2}. Setting 
$$
    \widetilde{\chi}(t) := \left\{ \begin{array}{rl} \ov{\chi}\Big(\widetilde{C} (4T_0 - t)\Big) & \text{ if } t\leq 4 T_0, \\ C_1 & \text{ if } t\geq 4 T_0, \end{array} \right.
$$
and mollifying the latter with mollifying parameter smaller than $T_0$ yields a desired $\chi$. Indeed, $\chi\in \Cb([0,\infty)) \cap C^\infty([0,\infty))$ and $\chi(t)\geq C_1$ for all $t\in [0,\infty)$. Moreover, for $t\geq 3T_0$ it holds $\partial_t \chi \leq 0$, and for $t\in [0,3T_0]$ we have $\chi\equiv \widetilde{\chi}$ and thus
\begin{align*}
    \frac{n}{2} \partial_t\chi + \ep C_2 \chi^2 + C_3\chi + C_4 &\leq \frac{n}{2} \left( \partial_t\chi + \widetilde{C} \Big(\ep \chi^2 + 2 \chi + 1 \Big)\right) \\
    &\leq \frac{n}{2} \widetilde{C} \Big( -\partial_t\ov{\chi} + \ep \ov{\chi}^2 + 2 \ov{\chi} + 1 \Big) = 0.
\end{align*}
\end{proof}

\section{Lake equations} \label{sec:ex-lake} 

In this section we will apply the general framework from Sect.~\ref{sec:general} to the lake equations \eqref{eq:lake-mass}, \eqref{eq:lake-mom}. Our aim is to prove Thms.~\ref{thm:lake-exforalldata} and \ref{thm:lake-wilddata}.

\subsection{Preliminaries and reformulation of the problem} \label{subsec:ex-lake-prel}

As we did above for the Euler equations, we rewrite the lake equations \eqref{eq:lake-mass}, \eqref{eq:lake-mom} as 
\begin{align}
    \Div m &= 0, \label{eq:lin-lake-mass} \\
    \partial_t m + \Div U &= - \Grad q - b\Grad p, \label{eq:lin-lake-mom} 
\end{align}
and we define the family of constitutive sets 
\begin{equation} \label{eq:const-lake}  
    \sK_{b,p,q} := \left\{ (m,U)\in \R^2 \times \symz{2} \,\Big|\, \frac{m\otimes m}{\rho} - U -q\id_2 = 0 \right\}.
\end{equation}
Using the notation from Sect.~\ref{sec:general}, we may write $\theta=(b,p,q)$ and $\zeta=(m,U)$. Like in the case of the compressible Euler equations, we work with the momentum $m$ as an unknown rather than the velocity $u$, which can be transformed into each other via $m=bu$ (note that we require $b$ to be strictly positive, see Sect.~\ref{subsubsec:intro-lake}).

It turns out that the wave cone for the above setting coincides with the wave cone in the context of the incompressible Euler equations, see Sect.~\ref{subsec:ex-incomp-euler-prel}.

\subsection{Structural assumptions of Thm.~\ref{thm:conv-int}} \label{subsec:ex-lake-ass} 

\begin{lemma} \label{lemma:ass-lake}
    \begin{enumerate}
        \item \label{item:ass1-ilake} The family of constitutive sets $(\sK_{b,p,q})_{(b,p,q)\in \Theta}$ defined in \eqref{eq:const-lake} is suitable in the sense of Defn.~\ref{defn:suitable-K}. 

        \item \label{item:ass2-lake} Let $(m,U)\in\Lambda$. Then there exists a third order homogeneous differential operator
        $$
		      \opL_{(m,U)} : C^\infty(\R^{3}) \to C^\infty(\R^{3};\R^2\times \symz{2}) 
        $$
        which is suitable in the sense of Defn.~\ref{defn:suitable-operator}. 

        \item \label{item:ass3-lake} It holds that 
        $$
            (\sK_{b,p,q})^\Lambda = (\sK_{b,p,q})^\co \qquad \text{ for all } (b,p,q) \in \Theta.
        $$ 

        \item \label{item:ass4-lake} There exists a map $D\in C(\Theta\times \ov{B}_c(0);\R)$ which is a suitable distance map in the sense of Defn.~\ref{defn:suitable-distance-map}. Moreover if $r=2$, then $D$ satisfies the additional property~\ref{item:suitable-D-convergence} introduced in Sect.~\ref{subsec:general-idiK}.
    \end{enumerate}
\end{lemma}

We skip a detailed proof of Lemma~\ref{lemma:ass-lake} because it follows the same lines as the proofs of Lemmas~\ref{lemma:ass-incomp-euler} and \ref{lemma:ass-compr-euler}.

\subsection{Proof of Thms.~\ref{thm:lake-exforalldata} and \ref{thm:lake-wilddata}} \label{subsec:ex-lake-pf} 

Since the structural assumptions of Thm.~\ref{thm:conv-int} hold, we may apply Thms.~\ref{thm:conv-int} and \ref{thm:idiK} in the context of the lake equations, which allows to prove Thms.~\ref{thm:lake-exforalldata} and \ref{thm:lake-wilddata}.

\begin{proof}[Proof of Thm.~\ref{thm:lake-exforalldata}]
Since the initial data $u_0$ are even $C^1$, the proof of Thm.~\ref{thm:lake-exforalldata} is much simpler than the proof of the corresponding theorem in the case of the incompressible Euler equations (see Thm.~\ref{thm:incomp-euler-exforalldata}). 

For the given topography $b$, we simply choose $p$ and $q$ constant, $\ov{m}(t,x)=b(x) u_0(x)$ and $\ov{U}\equiv 0$. If $q$ is taken sufficiently large, then $(\ov{m},\ov{U})$ takes values in $\sU_{b,p,q}=\interior{\big( (\sK_{b,p,q})^\Lambda\big)}$. Hence $(\ov{m},\ov{U})$ is indeed a desired subsolution (the PDEs \eqref{eq:lin-lake-mass}, \eqref{eq:lin-lake-mom} hold trivially) and we can apply Thm.~\ref{thm:conv-int} to obtain infinitely many solutions as stated in Thm.~\ref{thm:lake-exforalldata}. 
\end{proof}

\begin{proof}[Proof of Thm.~\ref{thm:lake-wilddata}]
Again we choose $p$ constant, and set $(\ov{m},\ov{U})\equiv (0,0)$. We will fix $q=q(t)>0$ later. It is then obvious that $(\ov{m},\ov{U})$ is a subsolution and we can proceed as in the proofs of Thms.~\ref{thm:incomp-euler-wilddata} and \ref{thm:compr-euler-wilddata}: We apply Thm.~\ref{thm:idiK} with $T_0>0$ arbitrary to obtain a new subsolution which (after shifting in time by $T_0$) takes values in $\sK_{b,p,q}$ for $t=0$ and in $\sU_{b,p,q}$ for $t>0$. Afterwards we apply Thm.~\ref{thm:conv-int} to obtain infinitely many weak solutions which take values in $\sK_{b,p,q}$ for $t\geq 0$. The corresponding initial data satisfy the compatibility condition \ref{eq:hydrost-euler-compatibility-initialdata} which can be shown by proceeding like in the proof of Thm.~\ref{thm:incomp-euler-wilddata}.

Next we estimate the energy inequality:
\begin{align*}
    \partial_t \left( \frac{|m|^2}{2b}\right) + \Div \left[ \left( \frac{|m|^2}{2b} + b p \right) \frac{m}{b} \right] &= \partial_t q + m\cdot \Grad \left( \frac{q}{b} + p \right) \\
    &\leq \partial_t q + C_1 q^2 + C_2 q,
\end{align*}
with $C_1,C_2>0$ depending only on $b$. It is not difficult (e.g.~using the Picard-Lindel{\"o}f-theorem) to show that there exists a $q$ making the energy inequality \eqref{eq:lake-enineq} valid and taking positive values at least on a short time interval $[0,T]$. The time $T$ only depends on $C_1$ and $C_2$ and thus on the topography $b$. 
\end{proof}

\section{Hydrostatic Euler equations}  \label{sec:ex-hydrost-euler}

In this section we cover one of the main aims of this paper, namely the application of the framework established in Sect.~\ref{sec:general} to the hydrostatic Euler equations \eqref{eq:hydrost-euler-mass}-\eqref{eq:hydrost-euler-p}. Compared to Sects.~\ref{sec:ex-incomp-euler}-\ref{sec:ex-lake} above, we now need to show the structural assumptions of Thm.~\ref{thm:conv-int} in the context of the hydrostatic Euler equations \eqref{eq:hydrost-euler-mass}-\eqref{eq:hydrost-euler-p} comprehensively and in detail (see Sect.~\ref{subsec:ex-hydrost-euler-ass} below). This is due to the fact that the authors are not aware of any work in the literature which can be used for our purposes, even though the proof of the structural assumptions is similar to the corresponding proof for the incompressible Euler equations (see Sect.~\ref{subsec:ex-incomp-euler-ass} above). In addition to that, in the context of the hydrostatic Euler equations there seems to be no simple explicit formula to characterize the set $\sU = \interior{(\sK^\Lambda)}$. For this reason, we will work with a sufficiently large subset $\sW\subset \sU$, see Sect.~\ref{subsec:ex-hydrost-euler-subset} below. Also in \cite{Markfelder24,GebKol22}, where convex integration was used in a different context, a subset $\sW\subset \sU$ was considered rather than the whole $\sU$.

Note that all these difficulties do not appear in \cite{ChiMic17} which is due to the special approach used therein (by considering the 3D incompressible Euler equations with an additional requirement on the pressure rather than the hydrostatic Euler equations directly). However this way one cannot achieve admissibility with respect to the right energy, see Rem.~\ref{rem:hydrost-euler-differences-ChiMic}.

\subsection{Preliminaries and reformulation of the problem} \label{subsec:ex-hydrost-euler-prel} 

Like in Sect.~\ref{subsec:ex-incomp-euler-prel} we first rewrite the equations \eqref{eq:hydrost-euler-mass}-\eqref{eq:hydrost-euler-p} as the linear system 
\begin{align}
	\Divh u + \partial_z w &= 0, \label{eq:lin1-hydrost-euler-mass} \\
	\partial_t u + \Divh (U+q\id_2) + \partial_z W &=0, \label{eq:lin1-hydrost-euler-mom} \\
	\partial_z p &=0. \label{eq:lin1-hydrost-euler-p} 
\end{align}
Here we have introduced the additional unknowns $U$, $q$ and $W$, which take values in $\symz{2}$, $\R$ and $\R^2$, respectively. Moreover the constitutive set reads 
\begin{align}
	\sK = \Big\{ (u,w,U,W,p,q)\in \R^2\times \R &\times \symz{2} \times \R^2 \times\R \times\R \,\Big|\, \label{eq:const1-hydrost-euler} \\
    & u\otimes u - U + (p-q)\id_2 = 0, \ W=uw \Big\}. \notag
\end{align} 

Also in the context of the hydrostatic Euler system, we treat $p$ and $q$ as parameters. Consequently we rewrite the linear system \eqref{eq:lin1-hydrost-euler-mass}-\eqref{eq:lin1-hydrost-euler-p} in the form \eqref{eq:lin-eq}, namely 
\begin{align}
	\Divh u + \partial_z w &= 0, \label{eq:lin-hydrost-euler-mass} \\
	\partial_t u + \Divh U + \partial_z W &= - \Gradh q, \label{eq:lin-hydrost-euler-mom} 
\end{align}
where equation \eqref{eq:lin1-hydrost-euler-p} shall be considered as an additional requirement on the given function $p$. 

In order to make the corresponding constitutive sets bounded, we introduce a bound $c\in \R$ and look for solutions which satisfy $|w|=c$. This also ensures that the constructed solutions will be truly hydrostatic in nature, rather than solutions of the 2D Euler equations. Therefore, instead of the set \eqref{eq:const1-hydrost-euler}, we obtain the family of constitutive sets 
\begin{align}
	\sK_{p,q,c} = \Big\{ (u,w,U,W)\in \R^2 &\times \R \times \symz{2} \times \R^2 \,\Big|\, \label{eq:const-hydrost-euler} \\
    & u\otimes u - U + (p-q)\id_2 = 0, \ W=uw , \ |w|=c \Big\}. \notag
\end{align}
Moreover we have $\theta=(p,q,c)$ and $\zeta=(u,w,U,W)$. Similar to Sect.~\ref{subsec:ex-incomp-euler-prel} we may assume that $\theta=(p,q,c)$ takes values in a compact set $\Theta$, and that $q\geq p$ and $c\geq 0$ on $\Theta$ (otherwise $\sK_{p,q,c}=\emptyset$ and existence of $\ov{\zeta}$ as required in Thm.~\ref{thm:conv-int} is impossible).

Finally we note the wave cone for the setting under consideration, which reads 
\begin{align*}
    \Lambda =  \bigg\{ (u,w,U,W)\in \R^2 \times \R \times \symz{2} \times \R^2 \, \Big|\, &\exists\eta=(\eta_t,\eta_x)\in \R^{4}\ \text{ with }\ \eta_x \neq 0 \ \text{ and }\  \\ 
    & \qquad \qquad \left(\begin{array}{ccc} 0 & u^\trans & w\\ u & U & W\end{array}\right)\cdot\eta = 0 \quad\bigg\}.
\end{align*}

\subsection{Structural assumptions of Thm.~\ref{thm:conv-int}} \label{subsec:ex-hydrost-euler-ass}

Like in the case of the incompressible Euler equations (see Sect.~\ref{subsec:ex-incomp-euler-ass}), we check in this subsection if the structural assumptions of Thm.~\ref{thm:conv-int} hold.

\begin{lemma} \label{lemma:ass-hydrost-euler} 
    \begin{enumerate}
        \item \label{item:ass1-hydrost-euler} The family of constitutive sets $(\sK_{p,q,c})_{(p,q,c)\in \Theta}$ defined in \eqref{eq:const-hydrost-euler} is suitable in the sense of Defn.~\ref{defn:suitable-K}.  

        \item \label{item:ass2-hydrost-euler} Let $(u,w,U,W)\in\Lambda$. Then there exists third order homogeneous differential operator
	    $$
		      \opL_{(u,w,U,W)} : C^\infty(\R^{4}) \to C^\infty(\R^{4};\R^2\times \R\times \symz{2}\times \R^2) 
        $$
        which is suitable in the sense of Defn.~\ref{defn:suitable-operator}.

        \item \label{item:ass3-hydrost-euler} It holds that 
        $$
            (\sK_{p,q,c})^\Lambda = (\sK_{p,q,c})^\co \qquad \text{ for all } (p,q,c) \in \Theta.
        $$ 

        \item \label{item:ass4-hydrost-euler} There exists a map\footnote{We write ``$\ov{c}$'' here for the bound denoted by ``$c$'' in Sect.~\ref{sec:general}. The reader should notice that in the current section the symbol ``$c$'' is used for something else. Still the authors believe that this notation does not cause confusion because ``$\ov{c}$'' will appear very rarely.} $D\in C(\Theta\times \ov{B}_{\ov{c}}(0);\R)$ which is a suitable distance map in the sense of Defn.~\ref{defn:suitable-distance-map}. Moreover if $r=2$, then $D$ satisfies the additional property~\ref{item:suitable-D-convergence} introduced in Sect.~\ref{subsec:general-idiK}.
    \end{enumerate}
\end{lemma}

\begin{proof}[Proof of Lemma~\ref{lemma:ass-hydrost-euler}~\ref{item:ass1-hydrost-euler}] 
We have to show that the properties in Defn.~\ref{defn:suitable-K} hold (denoted in sequel by ``Prop.~(a)'', ``Prop.~(b)'' and ``Prop.~(c)''). 
\begin{itemize}
    \item[Prop.~(a)] We see immediately from the definition of the constitutive sets $\sK_{p,q,c}$ (see \eqref{eq:const-hydrost-euler}) that they are closed. The boundedness of the constitutive sets for the incompressible Euler equations (see Lemma~\ref{lemma:ass-incomp-euler}~\ref{item:ass1-incomp-euler} and \cite[Lemma~3]{DelSze10}) yields bounds on $u$ and $U$. Moreover \eqref{eq:const-hydrost-euler} gives a bound on $w$ and consequently even $W$ is bounded on $\sK_{p,q,c}$.
    
    \item[Prop.~(b)] We proceed similar to the case of the incompressible Euler equations, see proof of Lemma~\ref{lemma:ass-incomp-euler}~\ref{item:ass1-incomp-euler}. 
    
    Let $\ep>0$ given. We may choose $c_u>1$ such that $|u|<c_u$ and $c<c_u$ for all $(u,w,U,W)\in \sK_{p,q,c}$ and all $(p,q,c)\in \Theta$. 

    Like in the proof of Lemma~\ref{lemma:ass-incomp-euler}~\ref{item:ass1-incomp-euler}, we fix $\delta>0$ such that $\delta< \frac{\ep}{6}$, $\delta< \frac{\ep}{2 c_u}$, and 
    $$
        \Big| \sqrt{q_1-p_1} - \sqrt{q_2-p_2} \Big| < \frac{\ep}{6 \sqrt{2} c_u}
    $$
    for all $(p_1,q_1,c_1),(p_2,q_2,c_2)\in \Theta$ with $|(p_1,q_1,c_1)-(p_2,q_2,c_2)| < \delta$.

    Now consider $|(p_1,q_1,c_1)-(p_2,q_2,c_2)| < \delta$ and $(u_1,w_1,U_1,W_1)\in \sK_{p_1,q_1,c_1}$. We choose $u_2$ and $U_2$ as in the proof of Lemma~\ref{lemma:ass-incomp-euler}~\ref{item:ass1-incomp-euler}, i.e. 
    $$
        u_2 := \sqrt{ 2 (q_2-p_2)} \frac{u_1}{|u_1|},
    $$
    if $u_1\neq 0$, and $u_2\in \R^2$ such that $|u_2|^2= 2 (q_2-p_2)$ otherwise. Moreover 
    $$
        U_2 := u_2 \otimes u_2 + (p_2-q_2) \id_2. 
    $$
    In addition to that we set 
    $$
        w_2 := w_1 \frac{c_2}{c_1}
    $$
    if $c_1\neq 0$, and $w_2\in \R$ such that $|w_2|=c_2$ otherwise. Finally define
    $$
        W_2 := u_2 w_2.
    $$
    
    By construction we have $(u_2,w_2,U_2,W_2)\in \sK_{p_2,q_2,c_2}$. As shown in the proof of Lemma~\ref{lemma:ass-incomp-euler}~\ref{item:ass1-incomp-euler} it holds that 
    $$
        |u_1 - u_2| < \ep \qquad \text{ and } \qquad |U_1 - U_2| < \ep. 
    $$
    Furthermore we find 
    \begin{align*}
        |w_1 - w_2 | &\leq |c_1 - c_2| < \delta < \ep, \\
        |W_1 - W_2 | &\leq |u_1-u_2| |w_1| + |u_2| |w_1 - w_2| < \frac{\ep}{6c_u} c_u + c_u \frac{\ep}{2c_u} <\ep.
    \end{align*}

    \item[Prop.~(c)] Proceeding as in \cite[Proof of Lemma~3.2]{Markfelder24}, one may show that 
    $$
        \interior{\big((\sK_{p,q,c})^\co\big)} \subset \Big\{ (u,w,U,W)\in \R^2 \times \R \times \symz{2} \times \R^2 \,\Big|\, \lambda_{\max} \big( u\otimes u - U + (p-q)\id_2 \big) < 0 \Big\}. 
    $$
    Hence 
    $$
        \interior{\big((\sK_{p,q,c})^\co\big)} \cap \sK_{p,q,c} = \emptyset \qquad \text{ for all }(p,q,c)\in \Theta.
    $$
\end{itemize}
\end{proof}

\begin{proof}[Proof of Lemma~\ref{lemma:ass-hydrost-euler}~\ref{item:ass2-hydrost-euler}] 
In this proof we will use the notation $\partial_i:= \partial_{x_i}$ for $i=1,2$, and accordingly $\partial_3:= \partial_z$. Moreover we write $\parthree{i}{j}{k}:=\partial_i \partial_j\partial_k$ where $i,j,k\in \{t,1,2,3\}$.

Let $(u,w,U,W)\in \Lambda$, i.e.~there exists $\eta=(\eta_t,\eta_1,\eta_2,\eta_3)\in \R^4$ with $\eta_1^2 + \eta_2^2 + \eta_3^2\neq 0$ and 
\begin{equation} \label{eq:op-lambda}
    \left(\begin{array}{ccc} 0 & u^\trans & w\\ u & U & W\end{array}\right)\cdot\eta = 0 .
\end{equation}
	
Define 
\begin{align*}
	\opL_{[u]_1}[g] &:= \alpha \parthree{3}{3}{3}g + \epsilon \Big( \parthree{1}{1}{2}g + \parthree{2}{2}{2}g \Big) - \zeta \parthree{1}{1}{3} g ,\\
	\opL_{[u]_2}[g] &:= \beta \parthree{3}{3}{3} g - \epsilon \Big( \parthree{1}{1}{1}g + \parthree{1}{2}{2}g \Big) -\theta \parthree{2}{2}{3} g ,\\
	\opL_w[g] &:= -\alpha\parthree{1}{3}{3}g - \beta\parthree{2}{3}{3}g + \zeta\parthree{1}{1}{1} g +\theta \parthree{2}{2}{2} g ,\\
	\opL_{[U]_{11}}[g] &:= \gamma\parthree{3}{3}{3}g -2\epsilon \parthree{t}{1}{2}g + \lambda \Big(\parthree{1}{1}{3}g + \parthree{2}{2}{3}g \Big) ,\\
	\opL_{[U]_{12}}[g] &:= \delta\parthree{3}{3}{3}g + \epsilon \Big( \parthree{t}{1}{1}g - \parthree{t}{2}{2}g \Big) - \kappa \Big(\parthree{1}{1}{3} g + \parthree{2}{2}{3}g \Big) ,\\
	\opL_{[W]_1}[g] &:= -\alpha\parthree{t}{3}{3}g - \gamma\parthree{1}{3}{3}g -\delta\parthree{2}{3}{3}g + \zeta \parthree{t}{1}{1} g +\kappa \Big(\parthree{1}{1}{2} g + \parthree{2}{2}{2} g\Big) - \lambda\Big(\parthree{1}{1}{1} g + \parthree{1}{2}{2} g \Big), \\
	\opL_{[W]_2}[g] &:= -\beta\parthree{t}{3}{3}g + \gamma\parthree{2}{3}{3}g -\delta\parthree{1}{3}{3}g + \theta\parthree{t}{2}{2}g +\kappa \Big(\parthree{1}{1}{1} g + \parthree{1}{2}{2} g \Big) + \lambda \Big(\parthree{1}{1}{2} g + \parthree{2}{2}{2} g\Big), 
\end{align*}
where\footnote{Note that $\eta_1^2 + \eta_2^2 + \eta_3^2 \neq 0$, and accordingly if $\eta_3=0$, then $\eta_1^2 + \eta_2^2\neq 0$. For the same reason if $\eta_3=\eta_1=0$, then $\eta_2\neq 0$.} 
\renewcommand{\arraystretch}{2.5}
\begin{align*}
	&(\alpha,\beta,\gamma,\delta,\epsilon,\zeta,\theta,\kappa,\lambda) \\
	&:= \left\{ \begin{array}{ll}
		\left(\frac{[u]_1}{\eta_3^3},\frac{[u]_2}{\eta_3^3},\frac{[U]_{11}}{\eta_3^3},\frac{[U]_{12}}{\eta_3^3},0,0,0,0,0\right), & \text{ if }\eta_3\neq 0 , \\
		\left(0,0,0,0,-\frac{[u]_2}{\eta_1 (\eta_1^2 + \eta_2^2)},\frac{w}{\eta_1^3},0,\frac{\eta_2 [W]_1 + \eta_1 [W]_2 - \frac{\eta_t \eta_2 w}{\eta_1}}{(\eta_1^2+\eta_2^2)^2},\frac{-\eta_1 [W]_1 + \eta_2 [W]_2 + \eta_t w}{(\eta_1^2+\eta_2^2)^2}\right), & \text{ if }\eta_3 = 0 \text{ and }\eta_1 \neq 0 , \\ 
		\left(0,0,0,0,\frac{[u]_1}{\eta_2^3},0,\frac{w}{\eta_2^3},\frac{[W]_1}{\eta_2^3},\frac{[W]_2}{\eta_2^3}-\frac{\eta_t w}{\eta_2^4 }\right), & \text{ if }\eta_3 = 0 \text{ and }\eta_1 = 0. 
	\end{array}\right. 
\end{align*}
\renewcommand{\arraystretch}{1}
It is then straightforward to check that the statement in item~\ref{item:suitable-operator-a} of Defn.~\ref{defn:suitable-operator} holds. 
	
Now let $g(t,x):= h((t,x)\cdot \eta)= h(\eta_t t + \eta_1 x_1 + \eta_2 x_2 + \eta_3 z)$. We first consider the case $\eta_3\neq 0$. We obtain 
\begin{align*}
	\opL_{[u]_1}[g] &= \alpha \eta_3^3 \,h'''((t,x)\cdot \eta) = [u]_1 \,h'''((t,x)\cdot \eta) ,\\
	\opL_{[u]_2}[g] &= \beta \eta_3^3 \,h'''((t,x)\cdot \eta) = [u]_2 \,h'''((t,x)\cdot \eta),\\
	\opL_w[g] &= \Big[-\alpha \eta_1 \eta_3^2 -\beta \eta_2 \eta_3^2\Big]\,h'''((t,x)\cdot \eta) = \left[-\frac{\eta_1 [u]_1}{\eta_3} - \frac{\eta_2 [u]_2}{\eta_3}\right] \,h'''((t,x)\cdot \eta)\\
	\opL_{[U]_{11}}[g] &= \gamma \eta_3^3 \,h'''((t,x)\cdot \eta) = [U]_{11} \,h'''((t,x)\cdot \eta) ,\\
	\opL_{[U]_{12}}[g] &= \delta \eta_3^3 \,h'''((t,x)\cdot \eta) = [U]_{12} \,h'''((t,x)\cdot \eta) ,\\
	\opL_{[W]_1}[g] &= \Big[-\alpha \eta_t \eta_3^2 - \gamma \eta_1 \eta_3^2 - \delta \eta_2 \eta_3^2\Big] \,h'''((t,x)\cdot \eta) = \left[-\frac{\eta_t [u]_1}{\eta_3} - \frac{\eta_1 [U]_{11}}{\eta_3} - \frac{\eta_2 [U]_{12}}{\eta_3}\right] \,h'''((t,x)\cdot \eta) , \\
	\opL_{[W]_2}[g] &= \Big[-\beta \eta_t \eta_3^2 + \gamma \eta_2 \eta_3^2 - \delta \eta_1 \eta_3^2\Big] \,h'''((t,x)\cdot \eta) = \left[- \frac{\eta_t [u]_2}{\eta_3} + \frac{\eta_2 [U]_{11}}{\eta_3} - \frac{\eta_1 [U]_{12}}{\eta_3} \right] \,h'''((t,x)\cdot \eta). 
\end{align*}
According to \eqref{eq:op-lambda} we have
\begin{align*}
	\eta_1 [u]_1 + \eta_2 [u]_2 + \eta_3 w &=0, \\
	\eta_t [u]_1 + \eta_1 [U]_{11} + \eta_2 [U]_{12} + \eta_3 [W]_1 &=0, \\
	\eta_t [u]_2 + \eta_1 [U]_{12} - \eta_2 [U]_{11} + \eta_3 [W]_2 &=0,
\end{align*}
and hence item~\ref{item:suitable-operator-b} of Defn.~\ref{defn:suitable-operator} holds. 

Next we consider $\eta_3 = 0$, $\eta_1\neq 0$. We obtain in a similar fashion 
\begin{align*} 
	\opL_{[u]_1}[g] &= \epsilon \Big[ \eta_1^2 \eta_2 + \eta_2^3\Big] \,h'''((t,x)\cdot \eta) = \frac{-\eta_2 [u]_2}{\eta_1} \,h'''((t,x)\cdot \eta) ,\\
	\opL_{[u]_2}[g] &= -\epsilon \Big[ \eta_1^3 + \eta_1 \eta_2^2\Big] \,h'''((t,x)\cdot \eta) = [u]_2 \,h'''((t,x)\cdot \eta),\\
	\opL_w[g] &= \zeta \eta_1^3 \,h'''((t,x)\cdot \eta) = w \,h'''((t,x)\cdot \eta)\\
	\opL_{[U]_{11}}[g] &= -2\epsilon \eta_t \eta_1 \eta_2 \,h'''((t,x)\cdot \eta) = \frac{2 \eta_t \eta_2 [u]_2}{\eta_1^2 + \eta_2^2} \,h'''((t,x)\cdot \eta) ,\\
	\opL_{[U]_{12}}[g] &= \epsilon\Big[ \eta_t \eta_1^2 - \eta_t \eta_2^2\Big] \,h'''((t,x)\cdot \eta) = -\frac{\eta_t(\eta_1^2 - \eta_2^2)[u]_2}{\eta_1(\eta_1^2 + \eta_2^2)} \,h'''((t,x)\cdot \eta) ,\\
	\opL_{[W]_1}[g] &= \Big[\zeta \eta_t \eta_1^2 + \kappa \eta_1^2 \eta_2 + \kappa \eta_2^3 - \lambda \eta_1^3 - \lambda \eta_1 \eta_2^2 \Big] \,h'''((t,x)\cdot \eta) = [W]_1 \,h'''((t,x)\cdot \eta) , \\
	\opL_{[W]_2}[g] &= \Big[\kappa \eta_1^3 + \kappa \eta_1 \eta_2^2 + \lambda \eta_1^2 \eta_2 + \lambda \eta_2^3 \Big] \,h'''((t,x)\cdot \eta) = [W]_2 \,h'''((t,x)\cdot \eta).
\end{align*}
In this case \eqref{eq:op-lambda} turns into 
\begin{align*}
	\eta_1 [u]_1 + \eta_2 [u]_2 &=0, \\
	\eta_t [u]_1 + \eta_1 [U]_{11} + \eta_2 [U]_{12} &=0, \\
	\eta_t [u]_2 + \eta_1 [U]_{12} - \eta_2 [U]_{11} &=0,
\end{align*} 
which yields by elementary manipulations 
\begin{align*}
	-2 \eta_t \eta_2 [u]_2 + (\eta_1^2 + \eta_2^2) [U]_{11} &=0, \\
	\frac{\eta_t(\eta_1^2 - \eta_2^2) [u]_2}{\eta_1} + (\eta_1^2 + \eta_2^2) [U]_{12} &=0.
\end{align*}
Thus again item~\ref{item:suitable-operator-b} of Defn.~\ref{defn:suitable-operator} is satisfied. 

Finally we suppose $\eta_3=\eta_1= 0$. We get
\begin{align*} 
	\opL_{[u]_1}[g] &= \epsilon \eta_2^3 \,h'''((t,x)\cdot \eta) = [u]_1 \,h'''((t,x)\cdot \eta) ,\\
	\opL_{[u]_2}[g] &= 0,\\
	\opL_w[g] &= \theta \eta_2^3\,h'''((t,x)\cdot \eta) = w \,h'''((t,x)\cdot \eta)\\
	\opL_{[U]_{11}}[g] &= 0,\\
	\opL_{[U]_{12}}[g] &= -\epsilon \eta_t \eta_2^2 \,h'''((t,x)\cdot \eta) = -\frac{\eta_t [u]_1}{\eta_2} \,h'''((t,x)\cdot \eta) ,\\ 
	\opL_{[W]_1}[g] &= \kappa \eta_2^3 \,h'''((t,x)\cdot \eta) = [W]_1 \,h'''((t,x)\cdot \eta) , \\
	\opL_{[W]_2}[g] &= \Big[\theta \eta_t \eta_2^2 + \lambda \eta_2^3 \Big] \,h'''((t,x)\cdot \eta) = [W]_2 \,h'''((t,x)\cdot \eta). 
\end{align*}
As $\eta_3=\eta_1=0$, \eqref{eq:op-lambda} yields 
\begin{align*}
	\eta_2 [u]_2 &=0, \\
	\eta_t [u]_1 + \eta_2 [U]_{12} &=0, \\
	\eta_t [u]_2 - \eta_2 [U]_{11} &=0,
\end{align*} 
and since $\eta_2\neq 0$, we infer $[u]_2=0$ and $[U]_{11}=0$. So we deduce that item~\ref{item:suitable-operator-b} of Defn.~\ref{defn:suitable-operator} holds in this case too. 
\end{proof}

\begin{proof}[Proof of Lemma~\ref{lemma:ass-hydrost-euler}~\ref{item:ass3-hydrost-euler}]
It turns out that the proof of Lemma~\ref{lemma:ass-hydrost-euler}~\ref{item:ass3-hydrost-euler} is very similar to the corresponding claim in the case of the incompressible Euler equations, i.e.~Lemma~\ref{lemma:ass-incomp-euler}~\ref{item:ass3-incomp-euler}, see \cite[Lemma~3]{DelSze10}.

We show that $\Lambda$ is complete with respect to $\sK_{p,q,c}$. Then the claim immediately follows from Prop.~\ref{prop:app-complete-wc}. So let $(u_1,w_1,U_1,W_1),(u_2,w_2,U_2,W_2)\in \sK_{p,q,c}$. Choose $\eta=(\eta_t,\eta_x)\in \R^4$ with $\eta_x\neq 0$, 
$$
    \eta_x \perp \left( \begin{array}{c} u_2-u_1 \\ w_2-w_1 \end{array} \right) \qquad \text{ and } \qquad\eta_t= -\left( \begin{array}{c} u_2 \\ w_2 \end{array} \right) \cdot \eta_x .
$$
Then a straightforward computation shows
\begin{align*}
    &\left(\begin{array}{ccc} 0 & u_2^\trans - u_1^\trans & w_2-w_1 \\ u_2-u_1 & U_2-U_1 & W_2 - W_1\end{array}\right)\eta \\
    &= \left(\begin{array}{ccc} 0 & u_2^\trans - u_1^\trans & w_2-w_1 \\ u_2-u_1 & u_2\otimes u_2 - u_1\otimes u_1 & u_2 w_2 - u_1 w_1 \end{array}\right)\eta \\ 
    &= \left(\begin{array}{ccc} 0 & u_2^\trans - u_1^\trans & w_2-w_1 \\ u_2-u_1 & u_2 u_2^\trans - u_1 u_2^\trans + u_1 u_2^\trans - u_1 u_1^\trans & u_2 w_2 - u_1 w_2 + u_1 w_2 - u_1 w_1 \end{array}\right)\eta \\
    &= \left(\begin{array}{ccc} 0 & u_2^\trans - u_1^\trans & w_2-w_1 \\ 0 & u_1 (u_2^\trans - u_1^\trans) & u_1 (w_2 - w_1) \end{array}\right)\eta + \left(\begin{array}{ccc} 0 & 0 & 0 \\ u_2-u_1 & (u_2 - u_1 ) u_2^\trans & (u_2 - u_1) w_2 \end{array}\right)\eta \\\\
    &= 0,
\end{align*} 
and thus $(u_2,w_2,U_2,W_2) - (u_1,w_1,U_1,W_1) \in \Lambda$.
\end{proof}

\begin{proof}[Proof of Lemma~\ref{lemma:ass-hydrost-euler}~\ref{item:ass4-hydrost-euler}]
Define $D((p,q,c),(u,w,U,W)):= 2(q-p) - |u|^2 + c^2 - w^2$. Obviously item~\ref{item:suitable-D-concave} of Defn.~\ref{defn:suitable-distance-map} holds. From \eqref{eq:const-hydrost-euler} we deduce property \ref{item:suitable-D-=0}. 

In order to prove item~\ref{item:suitable-D-inK}, let $(u,w,U,W)\in (\sK_{p,q,c})^\co$ with $D((p,q,c),(u,w,U,W))=0$. Since the maps $(u,w,U,W)\mapsto 2(q-p) - |u|^2$ and $(u,w,U,W)\mapsto c^2 - w^2$ are concave, we have $2(q-p) - |u|^2 \geq 0$ and $c^2 - w^2\geq 0$. So we deduce that both 
$$
    2(q-p) - |u|^2 =0 \qquad \text{ and }\qquad c^2 - w^2 = 0.
$$
As shown in \cite[Lemma~3]{DelSze10} for the case of the incompressible Euler equations, the former implies that 
$$
    u\otimes u - U + (p-q)\id_2 = 0,
$$
while the latter obviously leads to $|w|=c$. It remains to prove that $W=uw$. According to Prop.~\ref{prop:app-caratheodory} there exist $N\in \N$, $\tau_i\in \R^+$ with $\sum_{i=1}^N \tau_i = 1$, and $(u_i,w_i,U_i,W_i)\in \sK_{p,q,c}$ such that 
$$
    (u,w,U,W) = \sum_{i=1}^N \tau_i (u_i,w_i,U_i,W_i).
$$
In particular 
$$
    \sum_{i=1}^N \tau_i |w_i| = \sum_{i=1}^N \tau_i c = c=|w| = \left| \sum_{i=1}^N \tau_i w_i \right|. 
$$
This implies that $w_i=w$ for all $i=1,...,N$. Thus 
$$
    W = \sum_{i=1}^N \tau_i W_i = \sum_{i=1}^N \tau_i u_i w_i = \sum_{i=1}^N \tau_i u_i w = uw. 
$$
So all in all we have proven that $(u,w,U,W)\in \sK_{p,q,c}$, i.e.~item~\ref{item:suitable-D-inK} of Defn.~\ref{defn:suitable-distance-map} is satisfied.

Finally we observe that property~\ref{item:suitable-D-convergence} follows immediately from the definition of $D$. 
\end{proof}

\subsection{A subset of $\sU_{p,q,c}$} \label{subsec:ex-hydrost-euler-subset}

In contrast to the preceding sections (see the proofs of Lemmas~\ref{lemma:ass-incomp-euler}, \ref{lemma:ass-compr-euler} and \ref{lemma:ass-lake}), we have not found an explicit expression of $\sU_{p,q,c}$ while proving Lemma \ref{lemma:ass-hydrost-euler}. Such an explicit form would be helpful in order to find a suitable subsolution. However, instead we will construct an explicit subset $\sW_{p,q,c}\subset \sU_{p,q,c}$ which is large enough for our purposes. The construction of this subset $\sW_{p,q,c}$ will be the content of the current subsection. 

\begin{rem} \label{rem:zero-in-U-incomp}
    In this paper we only need that 
    \begin{equation} \label{eq:zero-in-U-incomp}
        (0,0,0,0)\in \sU_{p,q,c} \qquad \text{ for all }q>p,\ c>0.
    \end{equation}
    The explicit form of $\sW_{p,q,c}$ will immediately allow us to show that $(0,0,0,0)\in \sW_{p,q,c}$ which implies \eqref{eq:zero-in-U-incomp}. However there might be simpler way to prove \eqref{eq:zero-in-U-incomp} than the approach presented here. On the other hand our approach (i.e.~the explicit construction of a subset $\sW_{p,q,c}$ of $\sU_{p,q,c}$) might be quite powerful in view of possible future works, see also Rem.~\ref{rem:hydrost-euler-playwithc}.
\end{rem}

\subsubsection{Definition of the subset $\sW_{p,q,c}$}

In order to define $\sW_{p,q,c}$, we introduce the set 
\begin{equation*} 
    \sM_{p,q} := \left\{ (u,U)\in \R^2\times \symz{2} \,\Big|\, \lambda_{\max} ( u\otimes u - U - (q-p)\id_2 ) \leq 0\right\} ,
\end{equation*}
which coincides with the convex hull of the constitutive sets in the context of the incompressible Euler equations, see the proof of Lemma~\ref{lemma:ass-incomp-euler}. In the same proof we noted that 
\begin{equation*} 
    \interior{(\sM_{p,q})} = \left\{ (u,U)\in \R^2\times \symz{2} \,\Big|\, \lambda_{\max} ( u\otimes u - U - (q-p)\id_2 ) < 0\right\} .
\end{equation*}
Next we define four functions $f_{p,q}^{j,s}: \interior{(\sM_{p,q})}\to \R$ for $j\in \{1,2\}$ and $s\in \{-,+\}$ by 
\begin{align*} 
    f_{p,q}^{1,-}(u,U) &:= \frac{[u]_2 [U]_{12} - \sqrt{(q-p-[U]_{11} - [u]_2^2) \det((q-p)\id_2 + U)}}{q-p-[U]_{11}}, \\
    f_{p,q}^{1,+}(u,U) &:= \frac{[u]_2 [U]_{12} + \sqrt{(q-p-[U]_{11} - [u]_2^2) \det((q-p)\id_2 + U)}}{q-p-[U]_{11}}, \\
    f_{p,q}^{2,-}(u,U) &:= \frac{[u]_1 [U]_{12} - \sqrt{(q-p+[U]_{11} - [u]_1^2) \det((q-p)\id_2 + U)}}{q-p+[U]_{11}}, \\
    f_{p,q}^{2,+}(u,U) &:= \frac{[u]_1 [U]_{12} + \sqrt{(q-p+[U]_{11} - [u]_1^2) \det((q-p)\id_2 + U)}}{q-p+[U]_{11}}.
\end{align*}

\begin{rem}
    The inequality $\lambda_{\max} ( u\otimes u - U - (q-p)\id_2 ) < 0$, which constitutes the set $\interior{(\sM_{p,q})}$, may be translated into bounds for the components of $u$. As shown in Lemma~\ref{lemma:geom-properties-f}~\ref{item:prop-f-c} below, the functions $f_{p,q}^{j,s}$ exactly give these lower (for $s=-$) and upper (for $s=+$) bounds for $[u]_j$.
\end{rem} 

The following lemma summarizes some important properties of the functions $f_{p,q}^{j,s}$.

\begin{lemma} \label{lemma:geom-properties-f}
    The functions $f_{p,q}^{j,s}$ ($j\in \{1,2\}$, $s\in \{-,+\}$) have the following properties.
    \begin{enumerate}
        \item \label{item:prop-f-a} The functions $f_{p,q}^{j,s}$ are well-defined and continuous on $\interior{(\sM_{p,q})}$ for $j\in \{1,2\}$ and $s\in \{-,+\}$. 

        \item \label{item:prop-f-b} The functions $f_{p,q}^{1,-}$ and $f_{p,q}^{1,+}$ do not depend on $[u]_1$, while $f_{p,q}^{2,-}$ and $f_{p,q}^{2,+}$ do not depend on $[u]_2$. Therefore we will sometimes abuse notation and treat $f_{p,q}^{1,\pm}$ as functions from $\sM_{p,q}^1 \to \R$ and $f_{p,q}^{2,\pm}:\sM_{p,q}^2 \to \R$, where
        \begin{align*} 
            \sM_{p,q}^1 &:= \left\{ ([u]_2,U)\in \R\times \sz \,\Big|\, \exists [u]_1 \in \R \text{ such that } (u,U)\in \interior{(\sM_{p,q})} \right\} , \\
            \sM_{p,q}^2 &:= \left\{ ([u]_1,U)\in \R\times \sz \,\Big|\, \exists [u]_2 \in \R \text{ such that } (u,U)\in \interior{(\sM_{p,q})} \right\}  .
        \end{align*}
        
        \item \label{item:prop-f-c} It holds that 
        \begin{align}
            f_{p,q}^{1,-}(u,U) &< [u]_1 < f_{p,q}^{1,+}(u,U) , \label{eq:bound-u1-f} \\
            f_{p,q}^{2,-}(u,U) &< [u]_2 < f_{p,q}^{2,+}(u,U) , \label{eq:bound-u2-f}
        \end{align}
        for all $(u,U)\in \interior{(\sM_{p,q})}$. 

        \item \label{item:prop-f-d} 
        For all $([u]_2, U)\in \sM_{p,q}^1$ we have 
        \begin{equation} \label{eq:bound-in-M-1}
            \left( \left( \begin{array}{c} f_{p,q}^{1,-}([u]_2,U) \\ {[u]_2} \end{array} \right), U \right), \left( \left( \begin{array}{c} f_{p,q}^{1,+}([u]_2,U) \\ {[u]_2} \end{array} \right), U \right) \in \sM_{p,q} ,
        \end{equation}
        and similarly for all $([u]_1, U)\in \sM_{p,q}^2$ it holds that
        \begin{equation} \label{eq:bound-in-M-2}
            \left( \left( \begin{array}{c} {[u]_1} \\ f_{p,q}^{2,-}([u]_1,U) \end{array} \right), U \right), \left( \left( \begin{array}{c} {[u]_1} \\ f_{p,q}^{2,+}([u]_1,U) \end{array} \right), U \right) \in \sM_{p,q} .
        \end{equation}
    \end{enumerate}
\end{lemma}

\begin{proof} 
\begin{enumerate}
    \item First we observe from Lemma~\ref{lemma:app-eigenvalues}, that if $(u,U)\in \interior{(\sM_{p,q})}$, then the inequalities \eqref{eq:app-eigenvalues3} hold strictly. From the second inequality in \eqref{eq:app-eigenvalues3} we deduce that 
    $$
        \Big( [u]_1^2 - [U]_{11} - (q-p)\Big) \Big( [u]_2^2 + [U]_{11} - (q-p) \Big) > 0,
    $$
    which yields together with the first inequality in \eqref{eq:app-eigenvalues3} that 
    \begin{equation} \label{eq:lf-terms-nonneg}
        q-p+[U]_{11} - [u]_1^2 > 0, \qquad \text{ and } \qquad q-p-[U]_{11} - [u]_2^2 > 0 . 
    \end{equation}

    Next we show that 
    \begin{equation} \label{eq:lf-det>0}
        \det((q-p)\id_2 + U) > 0. 
    \end{equation}
    We compute using \eqref{eq:app-eigenvalues3}
    \begin{align}
        \det((q-p)\id_2 +U) &= \big(q-p+[U]_{11}\big) \big(q-p-[U]_{11}\big) - \big([U]_{12}\big)^2 \notag \\
        &= \Big(q-p+[U]_{11} - [u]_1^2 \Big) \Big(q-p-[U]_{11} - [u]_2^2 \Big) - \Big([U]_{12} - [u]_1 [u]_2 \Big)^2 \notag \\
        &\qquad + \big(q-p+[U]_{11}\big) [u]_2^2 + \big(q-p-[U]_{11}\big) [u]_1^2 - 2 [U]_{12} [u]_1 [u]_2 \notag \\
        &> \Big([u]_2, -[u]_1\Big) \cdot \Big( (q-p) \id_2 + U\Big) \cdot \left( \begin{array}{c} [u]_2 \\ -[u]_1 \end{array} \right) . \label{eq:lf-computation-det}
    \end{align}
    Since obviously 
    $$
        \Big([u]_2, -[u]_1\Big) \cdot u\otimes u \cdot \left( \begin{array}{c} [u]_2 \\ -[u]_1 \end{array} \right) = 0,
    $$
    we obtain from \eqref{eq:lf-computation-det} that 
    \begin{equation} \label{eq:lf-computation-det2}
        \det((q-p)\id_2 +U) > -\Big([u]_2, -[u]_1\Big) \cdot \Big(u\otimes u - U -(q-p) \id_2 \Big) \cdot \left( \begin{array}{c} [u]_2 \\ -[u]_1 \end{array} \right) .
    \end{equation}
    Because $\lambda_{\max} (u\otimes u - U - (q-p)\id_2) < 0$ implies that the matrix $u\otimes u - U - (q-p)\id_2$ is negative definite, the desired relation \eqref{eq:lf-det>0} follows from \eqref{eq:lf-computation-det2}. 

    Finally we observe that \eqref{eq:lf-terms-nonneg} leads to 
    \begin{equation} \label{eq:lf-denom-nonzero}
        q-p+[U]_{11} > 0\qquad \text{ and } \qquad q-p-[U]_{11} > 0.
    \end{equation}

    From \eqref{eq:lf-terms-nonneg}, \eqref{eq:lf-det>0} and \eqref{eq:lf-denom-nonzero} we deduce that the functions $f_{p,q}^{j,s}$ are well-defined for all $j\in\{1,2\}$ and $s\in \{-,+\}$. The continuity of the maps $f_{p,q}^{j,s}$ follows immediately from their definition.

    \item It is obvious that $f_{p,q}^{1,\pm}$ and $f_{p,q}^{2,\pm}$ are independent of $[u]_1$ and $[u]_2$, respectively. 

    \item Let us rewrite the second inequality in \eqref{eq:app-eigenvalues3} as 
    \begin{equation} \label{eq:lf-det-u1}
        - \big(q-p-[U]_{11}\big) [u]_1^2 + 2 [U]_{12} [u]_1 [u]_2 - \big(q-p+[U]_{11}\big) [u]_2^2 + \det((q-p)\id_2 +U) > 0 .
    \end{equation}
    Using \eqref{eq:lf-denom-nonzero} we deduce from \eqref{eq:lf-det-u1} that
    \begin{align*}
        \frac{[u]_2 [U]_{12} - \sqrt{(q-p-[U]_{11} - [u]_2^2) \det((q-p)\id_2 + U)}}{q-p-[U]_{11}} &< [u]_1 , \\
        \frac{[u]_2 [U]_{12} + \sqrt{(q-p-[U]_{11} - [u]_2^2) \det((q-p)\id_2 + U)}}{q-p-[U]_{11}} &> [u]_1 ,
    \end{align*}
    i.e.~\eqref{eq:bound-u1-f} holds. Relation \eqref{eq:bound-u2-f} follows in a similar fashion. 
    
    \item Let $([u]_2,U)\in \sM_{p,q}^1$, and define a function $A:\R\to \R$, 
    $$
        A(v) := - (q-p-[U]_{11}) v^2 + 2 [U]_{12} v [u]_2 - (q-p+[U]_{11}) [u]_2^2 + \det((q-p)\id_2 +U). 
    $$
    According to the definition of $\sM_{p,q}^1$ and Lemma~\ref{lemma:app-eigenvalues} there exists $[u]_1\in \R$ such that the inequalities in \eqref{eq:app-eigenvalues3} hold (strictly). We have seen already (see proof of item~\ref{item:prop-f-c} above) that 
    $$
        A\big(f_{p,q}^{1,-}([u]_2,U)\big) = 0, \qquad \text{ and } \qquad A \big( f_{p,q}^{1,+}([u]_2,U)\big) =0.
    $$
    Thus the points in \eqref{eq:bound-in-M-1} satisfy the second inequality in \eqref{eq:app-eigenvalues3} (as equality). Next we compute
    \begin{align*}
        A\left( \pm \sqrt{2(q-p) - [u]_2^2} \right) &= -\left(q-p - [U]_{11} - [u]_2^2 \right)^2 - \left([U]_{12} \mp u_2 \sqrt{2(q-p) - [u]_2^2} \right)^2 \\
        &\leq 0,
    \end{align*}
    which leads to 
    $$
        -\sqrt{2(q-p) -[u]_2^2} \leq f_{p,q}^{1,-}([u]_2,U) \leq f_{p,q}^{1,+}([u]_2,U) \leq  \sqrt{2(q-p) -[u]_2^2} .
    $$
    Hence the points in \eqref{eq:bound-in-M-1} also satisfy the first inequality in \eqref{eq:app-eigenvalues3}, and with Lemma~\ref{lemma:app-eigenvalues} we infer that \eqref{eq:bound-in-M-1} holds. 

    The fact that \eqref{eq:bound-in-M-2} is valid can be shown in the same way.
\end{enumerate}
\end{proof}

Now we are ready to define $\sW_{p,q,c}$ by 
\begin{align}
    \sW_{p,q,c} := \bigg\{ &(u,w,U,W)\in \R^2\times \R\times \symz{2}\times \R^2 \,\Big|\, \label{eq:defn-W} \\ 
        &\ \bullet\ \lambda_{\max}\left( u\otimes u - U - (q-p)\id_2\right) < 0, \notag \\ 
        &\ \bullet\ |w| < c , \notag \\
        &\ \bullet\ \exists \tau \in (0,1) \text{ such that } \notag \\
        &\ \quad \bullet\ [W]_1 < [u]_1 w + \tau \min\Big\{ \big( [u]_1 - f_{p,q}^{1,-}(u,U) \big) \big( c - w \big) , \big( f_{p,q}^{1,+}(u,U) - [u]_1 \big) \big( w + c \big) \Big\}, \notag \\
        &\ \quad \bullet\ [W]_1 > [u]_1 w - \tau \min\Big\{ \big( [u]_1 - f_{p,q}^{1,-}(u,U) \big) \big( w + c \big) , \big( f_{p,q}^{1,+}(u,U) - [u]_1 \big) \big( c - w \big) \Big\}, \notag \\
        &\ \quad \bullet\ [W]_2 < [u]_2 w + (1-\tau) \min\Big\{ \big( [u]_2 - f_{p,q}^{2,-}(u,U) \big) \big( c - w \big) , \big( f_{p,q}^{2,+}(u,U) - [u]_2 \big) \big( w + c \big) \Big\}, \notag \\
        &\ \quad \bullet\ [W]_2 > [u]_2 w - (1-\tau) \min\Big\{ \big( [u]_2 - f_{p,q}^{2,-}(u,U) \big) \big( w + c \big) , \big( f_{p,q}^{2,+}(u,U) - [u]_2 \big) \big( c - w \big) \Big\} \bigg\} . \notag
\end{align}

In order to show $\sW_{p,q,c}\subset \sU_{p,q,c}$, we proceed in several steps which are the content of the following Lemmas~\ref{lemma:geom-step1}, \ref{lemma:geom-step2a} and \ref{lemma:geom-step2b}. Afterwards we conclude in Cor.~\ref{cor:geom-WsubsetU}.

\subsubsection{Step 1: rigid $w$ and $W$}

\begin{lemma} \label{lemma:geom-step1} 
    Let $(u,w,U,W)\in \R^2\times \R\times \symz{2}\times \R^2$ such that 
    \begin{align*}
        \lambda_{\max}\left( u\otimes u - U - (q-p)\id_2\right) &\leq 0, \\ 
        |w| &= c , \\
        W &= uw .  
    \end{align*}
    Then $(u,w,U,W)\in (\sK_{p,q,c})^\co$.
\end{lemma}

\begin{proof} 
From \eqref{eq:convhull-incomp-euler} and Prop.~\ref{prop:app-caratheodory} we know that there exist $N\in \N$ and $\big(\tau_i,(u_i,U_i)\big) \in \R^+\times \R^2\times \sz$ for $i=1,...,N$ such that 
\begin{align*}
    &\bullet\ \sum_{i=1}^N \tau_i = 1 , \\
    &\bullet\ u_i\otimes u_i - U_i - (q-p)\id_2 =0 \quad \text{ for all }i=1,...,N, \\
    &\bullet\ \sum_{i=1}^N \tau_i (u_i,U_i) = (u,U) .
\end{align*}
Setting 
$$
    (w_i,W_i) :=(w, u_i w) \quad \text{ for all }i=1,...,N,
$$
we obviously have 
\begin{align*}
    &\bullet\ (u_i,w_i,U_i,W_i)\in \sK_{p,q,c} \quad \text{ for all }i=1,...,N, \\
    &\bullet\ \sum_{i=1}^N \tau_i (u_i,w_i,U_i,W_i) = (u,w,U,W) ,
\end{align*}
and again by Prop.~\ref{prop:app-caratheodory} this means $(u,w,U,W)\in (\sK_{p,q,c})^\co$ as desired.
\end{proof}

\subsubsection{Step 2: one component of $W$ is rigid}

\begin{lemma} \label{lemma:geom-step2a} 
    Let $(u,w,U,W)\in \R^2\times \R\times \symz{2}\times \R^2$ such that 
    \begin{align*} 
        \lambda_{\max}\left( u\otimes u - U - (q-p)\id_2\right) &< 0, \\ 
        |w| &< c , \\
        [W]_1 &< [u]_1 w + \min\Big\{ \big( [u]_1 - f_{p,q}^{1,-}(u,U) \big) \big( c - w \big) , \big( f_{p,q}^{1,+}(u,U) - [u]_1 \big) \big( w + c \big) \Big\}, \notag \\
        [W]_1 &> [u]_1 w - \min\Big\{ \big( [u]_1 - f_{p,q}^{1,-}(u,U) \big) \big( w + c \big) , \big( f_{p,q}^{1,+}(u,U) - [u]_1 \big) \big( c - w \big) \Big\}, \notag \\
        [W]_2 &= [u]_2 w .  
    \end{align*}
    Then $(u,w,U,W)\in (\sK_{p,q,c})^\co$.
\end{lemma}

\begin{proof} 
According to Lemma~\ref{lemma:geom-properties-f}~\ref{item:prop-f-c} we have $f_{p,q}^{1,-}([u]_2,U) < [u]_1 < f_{p,q}^{1,+}([u]_2,U)$. Hence Lemma~\ref{lemma:app-tetrahedron} yields existence of $\tau_1,\tau_2,\tau_3,\tau_4\in \R^+$ with 
\begin{align*} 
    \left( \begin{array}{c} [u]_1 \\ w \\ {[W]_1} \end{array} \right) &= \tau_1 \left( \begin{array}{c} f_{p,q}^{1,-}([u]_2,U) \\ -c \\ - f_{p,q}^{1,-}([u]_2,U) c \end{array} \right) + \tau_2 \left( \begin{array}{c} f_{p,q}^{1,-}([u]_2,U) \\ c \\ f_{p,q}^{1,-}([u]_2,U) c \end{array} \right) \\
    &\qquad + \tau_3 \left( \begin{array}{c} f_{p,q}^{1,+}([u]_2,U) \\ -c \\ -f_{p,q}^{1,+}([u]_2,U) c \end{array} \right) + \tau_4 \left( \begin{array}{c} f_{p,q}^{1,+}([u]_2,U) \\ c \\ f_{p,q}^{1,+}([u]_2,U) c \end{array} \right), \\
    \sum_{i=1}^4 \tau_i &= 1.
\end{align*}
Let us set 
\begin{align*}
    (u_1,w_1,U_1,W_1) &:= \left( \left( \begin{array}{c} f_{p,q}^{1,-}([u]_2,U) \\ {[u]_2} \end{array} \right), -c , U, \left( \begin{array}{c} -f_{p,q}^{1,-}([u]_2,U) c \\ -{[u]_2}c \end{array} \right) \right) , \\
    (u_2,w_2,U_2,W_2) &:= \left( \left( \begin{array}{c} f_{p,q}^{1,-}([u]_2,U) \\ {[u]_2} \end{array} \right), c , U, \left( \begin{array}{c} f_{p,q}^{1,-}([u]_2,U) c \\ {[u]_2}c \end{array} \right) \right) , \\
    (u_3,w_3,U_3,W_3) &:= \left( \left( \begin{array}{c} f_{p,q}^{1,+}([u]_2,U) \\ {[u]_2} \end{array} \right), -c , U, \left( \begin{array}{c} -f_{p,q}^{1,+}([u]_2,U) c \\ -{[u]_2}c \end{array} \right) \right) , \\
    (u_4,w_4,U_4,W_4) &:= \left( \left( \begin{array}{c} f_{p,q}^{1,+}([u]_2,U) \\ {[u]_2} \end{array} \right), c , U, \left( \begin{array}{c} f_{p,q}^{1,+}([u]_2,U) c \\ {[u]_2}c \end{array} \right) \right) .
\end{align*}
Then we compute 
$$
    (u,w,U,W) = \sum_{i=1}^4 \tau_i (u_i,w_i,U_i,W_i) .
$$

Next we show that $(u_i,w_i,U_i,W_i) \in (\sK_{p,q,c})^\co$ for all $i=1,...,4$. We obviously have $|w_i|=c$ and $W_i= u_i w_i$ for $i=1,..,4$. Moreover Lemma~\ref{lemma:geom-properties-f}~\ref{item:prop-f-d} ensures that 
$$
    \lambda_{\max}(u_i \otimes u_i - U_i - (q-p)\id_2) \leq 0 \quad\text{ for all } i=1,..,4. 
$$
Hence by Lemma~\ref{lemma:geom-step1} we have $(u_i,w_i,U_i,W_i) \in (\sK_{p,q,c})^\co$ for all $i=1,...,4$ and thus $(u,w,U,W)\in (\sK_{p,q,c})^\co$.
\end{proof}

Similar to Lemma~\ref{lemma:geom-step2a} we obtain the following.

\begin{lemma} \label{lemma:geom-step2b} 
    Let $(u,w,U,W)\in \R^2\times \R\times \symz{2}\times \R^2$ such that 
    \begin{align*} 
        \lambda_{\max}\left( u\otimes u - U - (q-p)\id_2\right) &< 0, \\ 
        |w| &< c , \\
        [W]_1 &= [u]_1 w , \\
        [W]_2 &< [u]_2 w + \min\Big\{ \big( [u]_2 - f_{p,q}^{2,-}(u,U) \big) \big( c - w \big) , \big( f_{p,q}^{2,+}(u,U) - [u]_2 \big) \big( w + c \big) \Big\}, \\
        [W]_2 &> [u]_2 w - \min\Big\{ \big( [u]_2 - f_{p,q}^{2,-}(u,U) \big) \big( w + c \big) , \big( f_{p,q}^{2,+}(u,U) - [u]_2 \big) \big( c - w \big) \Big\} .
    \end{align*}
    Then $(u,w,U,W)\in (\sK_{p,q,c})^\co$.
\end{lemma}

The proof of Lemma~\ref{lemma:geom-step2b} works exactly as the proof of Lemma~\ref{lemma:geom-step2a} with the only difference that the components of $u$ and $W$ are interchanged. Hence we skip the details of the proof of Lemma~\ref{lemma:geom-step2b}.

\subsubsection{Step 3: conclusion}

Now we are ready to conclude with the following corollary.

\begin{cor} \label{cor:geom-WsubsetU}
    It holds that $\sW_{p,q,c}\subset \sU_{p,q,c}$.
\end{cor}

\begin{proof} 
The mappings $(u,U)\mapsto \lambda_{\max}(u\otimes u- U -(q-p)\id_2)$ and $f_{p,q}^{j,s}$ are continuous on $\interior{(\sM_{p,q})}$, see \cite[Lemma~4.3.4]{Markfelder} and Lemma~\ref{lemma:geom-properties-f}~\ref{item:prop-f-a}, respectively. Hence $\sW_{p,q,c}$ is open. 

We show that $\sW_{p,q,c}\subset (\sK_{p,q,c})^\co$. Then by taking the interior and using the fact that $\sW_{p,q,c}$ is open, we obtain the desired relation $\sW_{p,q,c}\subset \sU_{p,q,c}$.

Let $(u,w,U,W)\in \sW_{p,q,c}$. Set 
\begin{align*}
    (u_1,w_1,U_1,W_1) &:= \left(u,w,U, \left( \begin{array}{c} \frac{[W]_1 - [u]_1 w}{\tau} + [u]_1 w \\ {[u]_2} w \end{array} \right) \right) , \\
    (u_2,w_2,U_2,W_2) &:= \left(u,w,U, \left( \begin{array}{c} {[u]_1} w \\ \frac{[W]_2 - [u]_2 w}{1-\tau} + {[u]_2} w \end{array} \right) \right).
\end{align*}
Then $(u_1,w_1,U_1,W_1),(u_2,w_2,U_2,W_2)\in (\sK_{p,q,c})^\co$ by Lemmas~\ref{lemma:geom-step2a} and \ref{lemma:geom-step2b}, respectively. Since 
$$
    (u,w,U,W) = \tau (u_1,w_1,U_1,W_1) + (1-\tau) (u_2,w_2,U_2,W_2) ,
$$
we deduce $(u,w,U,W)\in (\sK_{p,q,c})^\co$. 
\end{proof}

\begin{rem} \label{rem:hydrost-euler-playwithc} 
    We would like to remark, that the knowledge of the subset $\sW_{p,q,c}$ is quite powerful in the following sense. It is obvious that if $q,p\in \R$ and $(u,w,U,W)\in \R^2\times \R\times \symz{2}\times \R^2$ are given such that $\lambda_{\max}\left( u\otimes u - U - (q-p)\id_2\right)<0$, then there exists $c>0$ such that $(u,w,U,W)\in \sW_{p,q,c}\subset \sU_{p,q,c}$. This means that if one is able to choose $c$ arbitrarily, there is no restriction on $(w,W)$ for a point $(u,w,U,W)$ lying in $\sU_{p,q,c}$. This fact is never used in this paper, but it might be helpful in future works. 
\end{rem}

\subsection{Proof of Thms.~\ref{thm:hydrost-euler-exforalldata} and \ref{thm:hydrost-euler-wilddata}} \label{subsec:ex-hydrost-euler-pf} 

As we have already pointed out at the beginning of Sect.~\ref{sec:ex-hydrost-euler}, the application of convex integration to the hydrostatic Euler equations \eqref{eq:hydrost-euler-mass}-\eqref{eq:hydrost-euler-p} involves several new ideas compared to the preceding applications (i.e.~Sects.~\ref{sec:ex-incomp-euler}-\ref{sec:ex-lake}). This holds especially for the proof of Thm.~\ref{thm:hydrost-euler-wilddata}, where convex integration is applied directly to the hydrostatic Euler equations \eqref{eq:hydrost-euler-mass}-\eqref{eq:hydrost-euler-p}.

On the other hand, in order to prove Thm.~\ref{thm:hydrost-euler-exforalldata}, we will actually apply convex integration to the \emph{incompressible} Euler equations \eqref{eq:incomp-euler-mass}, \eqref{eq:incomp-euler-mom} (and not directly to the hydrostatic Euler equations \eqref{eq:hydrost-euler-mass}-\eqref{eq:hydrost-euler-p}). A related approach was already used in \cite{ChiMic17}. This also means that while the preceding Subsects.~\ref{subsec:ex-hydrost-euler-prel}-\ref{subsec:ex-hydrost-euler-subset} are essential for the proof of Thm.~\ref{thm:hydrost-euler-wilddata}, they are not required for the proof of Thm.~\ref{thm:hydrost-euler-exforalldata}. 

\begin{proof}[Proof of Thm.~\ref{thm:hydrost-euler-exforalldata}] 
We recall that the hydrostatic Euler system \eqref{eq:hydrost-euler-mass}-\eqref{eq:hydrost-euler-p} on the channel $\T^2\times [0,1]$ can be equivalently interpreted on the three-dimensional torus $\T^3$ with even symmetry of $u$ in $z$ and odd symmetry of $w$ in $z$. We refer to \cite{GILT22} for more details on this equivalence. Consequently we may consider the initial data $(u_0,w_0)$ as an object in $L^\infty(\T^3;\R^3)$, where the aforementioned symmetry conditions are satisfied. Hence we can apply Prop.~\ref{prop:app-sfad} to obtain a subsolution for the incompressible Euler equations \eqref{eq:incomp-euler-mass}, \eqref{eq:incomp-euler-mom}. Looking into the proof of Prop.~\ref{prop:app-sfad}, we observe that the subsolution satisfies the symmetry conditions mentioned above. Transforming back to the channel and proceeding as in the proof of Thm.~\ref{thm:incomp-euler-exforalldata}, we obtain infinitely many solutions to the incompressible Euler equations on the channel with initial data $(u_0,w_0)$. Keeping in mind, that the pressure $p$ of these solutions is constant (see the proof of Thm.~\ref{thm:incomp-euler-exforalldata} in Sect.~\ref{subsec:ex-incomp-euler-pf}), we deduce that these solutions (of the incompressible Euler equations) solve the hydrostatic Euler system too.
\end{proof}

\begin{proof}[Proof of Thm.~\ref{thm:hydrost-euler-wilddata}]
Choose $(\ov{u},\ov{w},\ov{U},\ov{W})=(0,0,0,0)$, $p,q,c\in \R$ constant such that $q>p$, $c>0$. Then $(\ov{u},\ov{w},\ov{U},\ov{W})$ satisfies the PDEs \eqref{eq:lin-hydrost-euler-mass}, \eqref{eq:lin-hydrost-euler-mom} and $\partial_z p=0$, and it takes values in $\sW_{p,q,c}\subset \sU_{p,q,c}$. Thus $(\ov{u},\ov{w},\ov{U},\ov{W})$ is a desired subsolution, and -- since the structural assumptions hold according to Lemma~\ref{lemma:ass-hydrost-euler} -- we can proceed like in the proofs of Thms.~\ref{thm:incomp-euler-wilddata}, \ref{thm:compr-euler-wilddata} and \ref{thm:lake-wilddata}. For the horizontal kinetic energy balance, we compute
\begin{align*}
    &\partial_t \left(\half |u|^2\right) + \Divh \left[\left(\half |u|^2 + p\right)u\right] + \partial_z \left[\left(\half |u|^2 + p\right)w\right] \\
    &= \partial_t (q-p) + \Divh (q u) + \partial_z (q w) = 0,
\end{align*}
which must be understood in the sense of distributions. Here we have used that $p,q$ are constant, and \eqref{eq:hydrost-euler-mass}. This finishes the proof of Thm.~\ref{thm:hydrost-euler-wilddata}.

In order to address to Rem.~\ref{rem:hydrost-euler-decreasing-energy}, one can show conservation of vertical kinetic energy in exactly the same way. By choosing $q=q(t)$ and $c=c(t)$ (which still must satisfy $q(t)>p$ and $c(t)>0$ for all $t\in [0,T)$) suitably, one can achieve that the horizontal and the vertical kinetic energy are strictly decreasing. 
\end{proof}

\section{Compressible inviscid primitive equations} \label{sec:ex-compr-prim}

Now we turn our attention towards the compressible inviscid primitive equations \eqref{eq:compr-prim-mass}-\eqref{eq:compr-prim-p}. In particular we will prove Thm.~\ref{thm:compr-prim-wilddata} in this section.

\subsection{Preliminaries and reformulation of the problem} \label{subsec:ex-compr-prim-prel} 

We combine the ideas which we used in the case of the compressible Euler equations (see Sect.~\ref{subsec:ex-compr-euler-prel}) with the approach for the hydrostatic Euler equations (see Sect.~\ref{subsec:ex-hydrost-euler-prel}). So, we rewrite the compressible inviscid primitive equations \eqref{eq:compr-prim-mass}-\eqref{eq:compr-prim-p} as
\begin{align} 
	\Divh m_h + \partial_z m_v &= - \partial_t \rho , \label{eq:lin-compr-prim-mass} \\
	\partial_t m_h + \Divh U + \partial_z W &= - \Gradh q, \label{eq:lin-compr-prim-mom}
\end{align}
and we consider the family of constitutive sets
\begin{align}
	\sK_{\rho,q,c} = \bigg\{ (m_h,m_v,U,W)\in \R^2 &\times \R \times \symz{2} \times \R^2 \,\Big|\, \label{eq:const-compr-prim} \\
    & \frac{m_h\otimes m_h}{\rho} - U + (p(\rho)-q)\id_2 = 0, \ W=\frac{m_h m_v}{\rho} ,\ |m_v| = c\bigg\}. \notag 
\end{align} 
Similar to Sect.~\ref{subsec:ex-compr-euler-prel}, we work with the horizontal and vertical momentum $m_h=\rho u$ and $m_v= \rho w$, respectively, rather than with the velocities $u$ and $w$. As in Sect.~\ref{subsec:ex-hydrost-euler-prel}, equation \eqref{eq:compr-prim-p} is understood as an additional requirement on the given function $\rho$. 

In the setting of the compressible inviscid primitive equations, we have $\theta=(\rho,q,c)$ and $\zeta=(m_h,m_v,U,W)$. Moreover the wave cone is given by
\begin{align*} 
    \Lambda =  \bigg\{ (m_h,m_v,U,W)\in \R^2 \times \R \times \symz{2} \times \R^2 \, \Big|\, &\exists\eta=(\eta_t,\eta_x)\in \R^{4}\ \text{ with }\ \eta_x \neq 0 \ \text{ and }\  \\ 
    & \qquad \qquad \left(\begin{array}{ccc} 0 & m_h^\trans & m_v\\ m_h & U & W\end{array}\right)\cdot\eta = 0 \quad\bigg\}.
\end{align*}

\subsection{Structural assumptions of Thm.~\ref{thm:conv-int}} \label{subsec:ex-compr-prim-ass}

Next we check the structural assumptions of Thm.~\ref{thm:conv-int}.

\begin{lemma} \label{lemma:ass-compr-prim}
    \begin{enumerate}
        \item \label{item:ass1-compr-prim} The family of constitutive sets $(\sK_{\rho,q,c})_{(\rho,q,c)\in \Theta}$ defined in \eqref{eq:const-compr-prim} is suitable in the sense of Defn.~\ref{defn:suitable-K}. 

        \item \label{item:ass2-compr-prim} Let $(m_h,m_v,U,W)\in\Lambda$. Then there exists a third order homogeneous differential operator
        $$
		      \opL_{(m_h,m_v,U,W)} : C^\infty(\R^{4}) \to C^\infty(\R^{4};\R^2 \times \R \times \symz{2} \times \R^2) 
        $$
        which is suitable in the sense of Defn.~\ref{defn:suitable-operator}. 

        \item \label{item:ass3-compr-prim} It holds that 
        $$
            (\sK_{\rho,q,c})^\Lambda = (\sK_{\rho,q,c})^\co \qquad \text{ for all } (\rho,q,c) \in \Theta.
        $$ 

        \item \label{item:ass4-compr-prim} There exists a map $D\in C(\Theta\times \ov{B}_{\ov{c}}(0);\R)$ which is a suitable distance map in the sense of Defn.~\ref{defn:suitable-distance-map}. Moreover if $r=2$, then $D$ satisfies the additional property~\ref{item:suitable-D-convergence} introduced in Sect.~\ref{subsec:general-idiK}.
    \end{enumerate}
\end{lemma}

\begin{proof} 
Lemma~\ref{lemma:ass-compr-prim} can be proved by modifying the proof of Lemma~\ref{lemma:ass-hydrost-euler}. We leave the details to the reader.
\end{proof}

\subsection{Some points in $\sU_{\rho,q,c}$} \label{subsec:ex-compr-prim-point}

Like in the incompressible case (see Rem.~\ref{rem:zero-in-U-incomp}) we will need an analogue of equation \eqref{eq:zero-in-U-incomp} for the compressible inviscid primitive equations, where we have to consider some more points rather than just $(0,0,0,0)$. This is the content of the following proposition.

\begin{prop} \label{prop:zero-in-U-compr} 
    It holds that 
    $$
        (m_h,0,0,0) \in \sU_{\rho,q,c} \qquad \text{ for all }m_h\in \R^2,\ \rho>0,\ q>p(\rho) + \frac{|m_h|^2}{\rho},\ c>0. 
    $$
\end{prop}

In order to prove Prop.~\ref{prop:zero-in-U-compr} we will make use of the set $\sW_{p,q,c}$ for the incompressible case, see Sect.~\ref{subsec:ex-hydrost-euler-subset}.

\begin{proof} 
In this proof we will write $\sU^\incomp$, $\sW^\incomp$ and $\sK^\incomp$ for the ``$\sU$'', ``$\sW$'' and ``$\sK$'' for the incompressible hydrostatic Euler equations \eqref{eq:hydrost-euler-mass}-\eqref{eq:hydrost-euler-p} which we treated in Sect.~\ref{sec:ex-hydrost-euler}.

It is straightforward to verify that 
$$
    \left(\frac{m_h}{\sqrt{\rho}},0,0,0\right) \in \sW^\incomp_{p(\rho),q,\frac{c}{\sqrt{\rho}}} \qquad \text{ for all }m_h\in \R^2,\ \rho>0,\ q>p(\rho) + \frac{|m_h|^2}{\rho},\ c>0, 
$$
just by using the explicit formula for $\sW^\incomp$, see \eqref{eq:defn-W}, together with Lemmas~\ref{lemma:geom-properties-f} and \ref{lemma:app-eigenvalues}. Since 
$$
    \sW^\incomp_{p(\rho),q,\frac{c}{\sqrt{\rho}}} \subset \sU^\incomp_{p(\rho),q,\frac{c}{\sqrt{\rho}}} = \interior{\Big(\Big(\sK^\incomp_{p(\rho),q,\frac{c}{\sqrt{\rho}}}\Big)^\co\Big)},
$$
there exist $N\in \N$ (with $N\geq 2$) and $\tau_i \in \R^+$, $(u_i,w_i,U_i,W_i)\in \sK^\incomp_{p(\rho),q,\frac{c}{\sqrt{\rho}}}$ for all $i=1,...,N$ such that 
$$
    \left(\frac{m_h}{\sqrt{\rho}},0,0,0\right)= \sum_{i=1}^N \tau_i (u_i,w_i,U_i,W_i)\qquad \text{ and } \qquad\sum_{i=1}^N \tau_i=1,
$$
see Prop.~\ref{prop:app-caratheodory}. Now let us define 
$$
    (m_h)_i := u_i \sqrt{\rho}, \qquad \text{ and } \qquad (m_v)_i := w_i \sqrt{\rho}
$$
for any $i=1,...,N$. It is easy to verify that $\big((m_h)_i,(m_v)_i,U_i,W_i\big) \in \sK_{\rho,q,c}$ for all $i=1,...,N$, and that 
$$
    (m_h,0,0,0)= \sum_{i=1}^N \tau_i \big((m_h)_i,(m_v)_i,U_i,W_i\big).
$$
Hence $(m_h,0,0,0)\in (\sK_{\rho,q,c})^\co$. As $\sU^\incomp_{p(\rho),q,\frac{c}{\sqrt{\rho}}}$ is open, it contains even a neighbourhood of $\left(\frac{m_h}{\sqrt{\rho}},0,0,0\right)$. Following the same steps as above, this means that even a neighbourhood of $(m_h,0,0,0)$ is contained in $(\sK_{\rho,q,c})^\co$, and thus 
$$
    (m_h,0,0,0) \in \interior{\big((\sK_{\rho,q,c})^\co\big)} = \sU_{\rho,q,c}
$$
as desired. 
\end{proof}

\begin{rem}
    Let us remark that one can find a compressible analogue of the set $\sW_{p,q,c}$ from Sect.~\ref{subsec:ex-hydrost-euler-subset} by following the same steps as in the proof of Prop.~\ref{prop:zero-in-U-compr}. 
\end{rem}

\subsection{Proof of Thm.~\ref{thm:compr-prim-wilddata}} \label{subsec:ex-compr-prim-pf} 

Now we are ready to prove Thm.~\ref{thm:compr-prim-wilddata}. To this end, we will proceed as in the proof of Thm.~\ref{thm:compr-euler-wilddata}, see Sect.~\ref{subsec:ex-compr-euler-pf}.

\begin{proof}[Proof of Thm.~\ref{thm:compr-prim-wilddata}]
Fix $T_0$, $h$, $\widetilde{\rho}$ and $\rho$ as in the proof of Thm.~\ref{thm:compr-euler-wilddata}. Let $\Psi$ be the solution of $\Lap_h \Psi = -\partial_t \rho$ and set $\overline{m_h}:= \Gradh \Psi$. Note that $\rho$ and thus $\Psi$ and $\overline{m_h}$ do not depend on $z$. Moreover $\Psi$ has the regularity specified in the proof of Thm.~\ref{thm:compr-euler-wilddata} (i.e.~$\Psi\in C^2$). In addition to that, define $\overline{m_v}\equiv 0$, $\overline{U}\equiv 0$, $\overline{W}\equiv 0$, $c\equiv 1$, and $q(t,x):= \chi(t) - \partial_t \Psi (t,x)$ like in the proof of Thm.~\ref{thm:compr-euler-wilddata}.

With the help of Prop.~\ref{prop:zero-in-U-compr} we deduce that $(\overline{m_h},\overline{m_v},\overline{U},\overline{W})$ is a subsolution as soon as $\chi(t)\geq C_1$ for all $t\in [0,\infty)$ where $C_1>0$ is a suitable constant. Then we may continue exactly as in the proof of Thm.~\ref{thm:compr-euler-wilddata} to finish the proof of Thm.~\ref{thm:compr-prim-wilddata}.
\end{proof}

\section{Quasi-geostrophic equations} \label{sec:ex-QG}

Finally we apply the general framework from Sect.~\ref{sec:general} to the quasi-geostrophic equations \eqref{eq:QG-curl}, \eqref{eq:QG-mom}. Our goal is to prove Thm.~\ref{thm:QG-wilddata}.

\subsection{Preliminaries and reformulation of the problem} \label{subsec:ex-QG-prel} 

In order to rewrite the quasi-geostrophic equations \eqref{eq:QG-curl}, \eqref{eq:QG-mom} as a set of linear PDEs with a family of constitutive sets, we have to introduce the following notation. Let\footnote{The set $\QGmatrixset$ corresponds to $\symz{n}$, which was used in other examples treated in the present paper, see e.g.~Sect.~\ref{subsec:ex-incomp-euler-prel}.} 
$$
    \QGmatrixset := \left\{ A= \left( \begin{array}{rr} [A]_{11} & [A]_{12} \\ {[A]_{21}} & [A]_{22} \\ {[A]_{31}} & [A]_{32} \end{array} \right) \in \R^{3\times 2}\,\Big|\, [A]_{12}=[A]_{21} \text{ and }[A]_{11} = -[A]_{22} \right\} 
$$
and
$$
    \mathbb{E}^{3\times 2} := \left( \begin{array}{rr} 0 & 1 \\ -1 & 0 \\ 0 & 0 \end{array} \right).
$$

Now we introduce new unknowns $(U,q,c)\in \QGmatrixset\times \R\times \R$, and consider the linear system 
\begin{align} 
	\Curl \QGunknown &= 0 , \label{eq:lin-QG-curl} \\
	\partial_t \QGunknown + \Divh U &= \Curl Q - \Curl \left( \begin{array}{c} 0 \\ 0 \\ q \end{array} \right) , \label{eq:lin-QG-mom} 
\end{align}
together with the family of constitutive sets 
\begin{equation} \label{eq:const-QG}
	\sK_{Q,q,c} = \bigg\{ (\QGunknown,U)\in \R^3 \times \QGmatrixset \,\Big|\, \QGunknown\otimes (\QGunknown_h)^\perp - U - q \mathbb{E}^{3\times 2} = 0, \ |\QGunknown_v| = c\bigg\}.  
\end{equation} 
With the notation of Sect.~\ref{sec:general}, we have $\zeta=(\QGunknown,U)$ and $\theta=(Q,q,c)$ in this setting. We will assume that $q>0$ on $\Theta$ as otherwise $\sK_{Q,q,c}=\emptyset$ (for $q<0$) or $\sK_{Q,q,c}=\{(0,0)\}$ (for $q=0$) both of which are too small to run convex integration\footnote{In particular, $\interior{\big((\sK_{Q,q,c})^\co\big)} = \emptyset$ in both cases.}. 

Furthermore, we will work with the following wave cone
\begin{align*} 
    \Lambda =  \bigg\{ (\QGunknown,U)\in \R^3 \times \QGmatrixset \, &\Big|\, \QGunknown_h\neq 0\text{ and }\exists\eta=(\eta_t,\eta_1,\eta_2,\eta_3)\in \R^{4}\ \text{ with }\ (\eta_1,\eta_2) \neq 0 ,\ \\ 
    & \qquad \qquad \qquad \QGunknown\times \eta_x=0\ \text{ and }\ \QGunknown \eta_t + U\cdot (\eta_x)_h = 0 \ \bigg\}. 
\end{align*} 

\begin{rem} \label{rem:QG-wave-cone}
    In view of Defn.~\ref{defn:wave-cone}, the wave cone should actually be defined as 
    \begin{align} 
        \Lambda' = \bigg\{ (\QGunknown,U)\in \R^3 \times \QGmatrixset \, &\Big|\, \exists\eta=(\eta_t,\eta_x)\in \R^{4}\ \text{ with }\ \eta_x \neq 0 ,\ \label{eq:QG-actual-WC} \\ 
        & \qquad \qquad \qquad \QGunknown\times \eta_x=0\ \text{ and }\ \QGunknown \eta_t + U\cdot (\eta_x)_h = 0 \ \bigg\}. \notag
    \end{align}
    Note however that the latter cone is too large to ensure that suitable differential operators $\opL_{(\QGunknown,U)}$ exist. As already explained in Rem.~\ref{rem:lambdatilde-vs-lambda}, we solve this by shrinking the wave cone by requiring in addition that $(\eta_1,\eta_2)\neq 0$ (which makes the condition $\eta_x\neq 0$ redundant). The condition $\QGunknown_h\neq 0$ is needed to prove \eqref{eq:KbLambda=KbCo}, see Proof of Lemma~\ref{lemma:ass-QG}~\ref{item:ass3-QG} below.
\end{rem}

\subsection{Structural assumptions of Thm.~\ref{thm:conv-int}} \label{subsec:ex-QG-ass}

Let us verify the structural assumptions of Thm.~\ref{thm:conv-int}.

\begin{lemma} \label{lemma:ass-QG} 
    \begin{enumerate}
        \item \label{item:ass1-QG} The family of constitutive sets $(\sK_{Q,q,c})_{(Q,q,c)\in \Theta}$ defined in \eqref{eq:const-QG} is suitable in the sense of Defn.~\ref{defn:suitable-K}. 

        \item \label{item:ass2-QG} Let $(\QGunknown,U)\in\Lambda$. Then there exists a third order homogeneous differential operator
        $$
		      \opL_{(\QGunknown,U)} : C^\infty(\R^{4}) \to C^\infty(\R^{4};\R^3 \times \QGmatrixset) 
        $$
        which is suitable in the sense of Defn.~\ref{defn:suitable-operator}. 

        \item \label{item:ass3-QG} It holds that\footnote{The reader should notice that in contrast to the other examples considered in the present paper (see Sects.~\ref{sec:ex-incomp-euler}-\ref{sec:ex-compr-prim}) where we showed \eqref{eq:KbLambda=KbCo-noint}, we will only prove \eqref{eq:KbLambda=KbCo} here.} 
        $$
            \interior{\big((\sK_{Q,q,c})^\Lambda\big)} = \interior{\big((\sK_{Q,q,c})^\co\big)} \qquad \text{ for all } (Q,q,c) \in \Theta.
        $$ 

        \item \label{item:ass4-QG} There exists a map $D\in C(\Theta\times \ov{B}_{\ov{c}}(0);\R)$ which is a suitable distance map in the sense of Defn.~\ref{defn:suitable-distance-map}. Moreover if $r=2$, then $D$ satisfies the additional property~\ref{item:suitable-D-convergence} introduced in Sect.~\ref{subsec:general-idiK}.
    \end{enumerate}
\end{lemma} 

\begin{proof}[Proof of Lemma~\ref{lemma:ass-QG}~\ref{item:ass1-QG}] 
We observe that $\sK_{Q,q,c}$ and $\sK_{p,q,c}^\hydrost$ (the constitutive set for the hydrostatic Euler equations, see \eqref{eq:const-hydrost-euler}) are in some sense equal. Indeed it is simple to verify that 
$$ 
    \left(\QGunknown_h, \QGunknown_v, \left(\begin{array}{rr} [U]_{12} & -[U]_{11} \\ {[U]_{22}} & -[U]_{21} \end{array} \right),\left( \begin{array}{r} [U]_{32} \\ -[U]_{31} \end{array} \right)\right) \in \sK_{0,q,c}^\hydrost \qquad \text{ for all } (\QGunknown,U)\in \sK_{Q,q,c},
$$
and conversely 
$$ 
    \left( \left(\begin{array}{c} u \\ w \end{array}\right) , \left(\begin{array}{rr} -[U]_{12} & [U]_{11} \\ -[U]_{22} & [U]_{21} \\ -[W]_{2} & [W]_{1} \end{array} \right) \right) \in \sK_{Q,q-p,c} \qquad \text{ for all } (u,w,U,W)\in \sK_{p,q,c}^\hydrost.
$$ 

With this at hand, it is simple to see that suitability of the family $(\sK_{Q,q,c})_{(Q,q,c)\in \Theta}$ follows from suitability of the sets $(\sK_{p,q,c}^\hydrost)_{(p,q,c)\in \Theta}$ which we proved in Sect.~\ref{subsec:ex-hydrost-euler-ass}, see Lemma~\ref{lemma:ass-hydrost-euler}~\ref{item:ass1-hydrost-euler}. 
\end{proof}

\begin{proof}[Proof of Lemma~\ref{lemma:ass-QG}~\ref{item:ass2-QG}] 
Let $(\QGunknown,U) \in \Lambda$, i.e.~$\QGunknown_h\neq 0$ and there exists $\eta=(\eta_t,\eta_1,\eta_2,\eta_3)\in \R^4$ with $\eta_1^2 + \eta_2^2 \neq 0$ and 
\begin{equation} \label{eq:op-lambda-QG}
    \QGunknown \times \eta_x = 0 \qquad \text{ and }\qquad \QGunknown \eta_t + U \cdot (\eta_x)_h = 0. 
\end{equation}

Define
\begin{align*} 
	\opL_{[\QGunknown]_1}[g] &:= \alpha \Big( \parthree{1}{1}{1}g + \parthree{1}{2}{2}g \Big),\\
	\opL_{[\QGunknown]_2}[g] &:= \alpha \Big( \parthree{1}{1}{2}g + \parthree{2}{2}{2}g \Big),\\
	\opL_{[\QGunknown]_3}[g] &:= \alpha \Big( \parthree{1}{1}{3}g + \parthree{2}{2}{3}g \Big),\\
	\opL_{[U]_{12}}[g] &:= -2 \alpha \parthree{t}{1}{2}g ,\\
    \opL_{[U]_{22}}[g] &:= \alpha \Big( \parthree{t}{1}{1}g - \parthree{t}{2}{2}g \Big),\\
	\opL_{[U]_{31}}[g] &:= -\alpha \parthree{t}{1}{3}g + \beta \parthree{2}{2}{2}g + \gamma \parthree{1}{1}{2}g, \\
	\opL_{[U]_{32}}[g] &:= -\alpha \parthree{t}{2}{3}g - \beta \parthree{1}{2}{2}g - \gamma \parthree{1}{1}{1}g,  
\end{align*}
where\footnote{Note that $\eta_1^2 + \eta_2^2 \neq 0$, and accordingly if $\eta_1=0$, then $\eta_2\neq 0$.} 
\renewcommand{\arraystretch}{2.0}
\begin{equation*}
	(\alpha,\beta,\gamma) := \left\{ \begin{array}{ll}
		\left(\frac{[\QGunknown]_1}{\eta_1 (\eta_1^2 + \eta_2^2)},0,\frac{\eta_2 [U]_{31} - \eta_1 [U]_{32}}{\eta_1^2 (\eta_1^2 + \eta_2^2)}\right), & \text{ if }\eta_1\neq 0 , \\
		\left(\frac{[\QGunknown]_2}{\eta_2^3},\frac{[U]_{31}}{\eta_2^3},0\right), & \text{ if }\eta_1 = 0. 
	\end{array}\right. 
\end{equation*}
\renewcommand{\arraystretch}{1}
It is then straightforward to check that item~\ref{item:suitable-operator-a} of Defn.~\ref{defn:suitable-operator} holds. 

Now let $g(t,x):= h((t,x)\cdot \eta) = h(\eta_t t + \eta_1 x_1 + \eta_2 x_2 + \eta_3 x_3)$. We first consider the case $\eta_1\neq 0$. We obtain 
\begin{align*} 
	\opL_{[\QGunknown]_1}[g] &= \alpha \eta_1 (\eta_1^2 + \eta_2^2) \,h'''((t,x)\cdot \eta) = [\QGunknown]_1 \,h'''((t,x)\cdot \eta) ,\\
	\opL_{[\QGunknown]_2}[g] &= \alpha \eta_2 (\eta_1^2 + \eta_2^2) \,h'''((t,x)\cdot \eta) = [\QGunknown]_1 \frac{\eta_2}{\eta_1} \,h'''((t,x)\cdot \eta) ,\\
	\opL_{[\QGunknown]_3}[g] &= \alpha \eta_3 (\eta_1^2 + \eta_2^2) \,h'''((t,x)\cdot \eta) = [\QGunknown]_1 \frac{\eta_3}{\eta_1} \,h'''((t,x)\cdot \eta) ,\\
	\opL_{[U]_{12}}[g] &= -2 \alpha \eta_t \eta_1 \eta_2 \,h'''((t,x)\cdot \eta) = -2 [\QGunknown]_1 \frac{\eta_t \eta_2}{\eta_1^2 + \eta_2^2}  \,h'''((t,x)\cdot \eta) ,\\
    \opL_{[U]_{22}}[g] &= \alpha \eta_t (\eta_1^2 - \eta_2^2) \,h'''((t,x)\cdot \eta) = [\QGunknown]_1\frac{\eta_t (\eta_1^2 - \eta_2^2)}{\eta_1 (\eta_1^2 + \eta_2^2)} \,h'''((t,x)\cdot \eta) ,\\
	\opL_{[U]_{31}}[g] &= \Big[ -\alpha \eta_t \eta_1 \eta_3 + \gamma \eta_1^2 \eta_2 \Big] \,h'''((t,x)\cdot \eta) = \frac{-[\QGunknown]_1\eta_t \eta_3 + \eta_2^2 [U]_{31} - \eta_1 \eta_2 [U]_{32}}{\eta_1^2 + \eta_2^2} \,h'''((t,x)\cdot \eta), \\
	\opL_{[U]_{32}}[g] &= \Big[ -\alpha \eta_t \eta_2 \eta_3 - \gamma \eta_1^3 \Big] \,h'''((t,x)\cdot \eta) = \frac{-[\QGunknown]_1 \eta_t \eta_2 \eta_3 - \eta_1^2 \eta_2 [U]_{31} + \eta_1^3 [U]_{32}}{\eta_1 (\eta_1^2 + \eta_2^2)} \,h'''((t,x)\cdot \eta).  
\end{align*}
According to \eqref{eq:op-lambda-QG}, we have
\begin{align*}
	\eta_1 [\QGunknown]_2 - \eta_2 [\QGunknown]_1 &= 0, \\
    \eta_1 [\QGunknown]_3 - \eta_3 [\QGunknown]_1 &= 0, \\
	\eta_t [\QGunknown]_1 - \eta_1 [U]_{22} + \eta_2 [U]_{12} &=0, \\
    \eta_t [\QGunknown]_2 + \eta_1 [U]_{12} + \eta_2 [U]_{22} &=0, \\
    \eta_t [\QGunknown]_3 + \eta_1 [U]_{31} + \eta_2 [U]_{32} &=0, 
\end{align*}
which yields by elementary manipulations
\begin{align*}
    2 \eta_t \eta_2 [\QGunknown]_1 + (\eta_1^2 + \eta_2^2) [U]_{12} &= 0, \\
    [\QGunknown]_1 \frac{\eta_t}{\eta_1} (\eta_1^2 - \eta_2^2) - (\eta_1^2 + \eta_2^2) [U]_{22} &= 0,\\ 
    -[\QGunknown]_1 \eta_t \eta_3 + \eta_2^2 [U]_{31} - \eta_1 \eta_2 [U]_{32} &= (\eta_1^2 + \eta_2^2) [U]_{31}, \\
    -[\QGunknown]_1 \eta_t \eta_2 \eta_3 - \eta_1^2 \eta_2 [U]_{31} + \eta_1^3 [U]_{32} &= \eta_1 (\eta_1^2 + \eta_2^2) [U]_{32} .
\end{align*}
Hence item \ref{item:suitable-operator-b} of Defn.~\ref{defn:suitable-operator} holds. 

Next we consider $\eta_1=0$. We obtain in a similar fashion
\begin{align*} 
	\opL_{[\QGunknown]_1}[g] &= 0 ,\\
	\opL_{[\QGunknown]_2}[g] &= \alpha \eta_2^3 \,h'''((t,x)\cdot \eta) = [\QGunknown]_2 \,h'''((t,x)\cdot \eta) ,\\
	\opL_{[\QGunknown]_3}[g] &= \alpha \eta_2^2 \eta_3 \,h'''((t,x)\cdot \eta) = [\QGunknown]_2 \frac{\eta_3}{\eta_2} \,h'''((t,x)\cdot \eta) ,\\
	\opL_{[U]_{12}}[g] &= 0 ,\\
    \opL_{[U]_{22}}[g] &= -\alpha \eta_t \eta_2^2 \,h'''((t,x)\cdot \eta) = -[\QGunknown]_2 \frac{\eta_t}{\eta_2} \,h'''((t,x)\cdot \eta) ,\\
	\opL_{[U]_{31}}[g] &= \beta \eta_2^3 \,h'''((t,x)\cdot \eta) = [U]_{31} \,h'''((t,x)\cdot \eta), \\
	\opL_{[U]_{32}}[g] &= -\alpha \eta_t \eta_2 \eta_3 \,h'''((t,x)\cdot \eta) = -[\QGunknown]_2 \frac{\eta_t \eta_3}{\eta_2^2} \,h'''((t,x)\cdot \eta).  
\end{align*}
In this case \eqref{eq:op-lambda-QG} turns into 
\begin{align*}
	\eta_2 [\QGunknown]_1 &= 0, \\ 
    \eta_2 [\QGunknown]_3 - \eta_3 [\QGunknown]_2 &= 0 , \\ 
	\eta_t [\QGunknown]_1 + \eta_2 [U]_{12} &=0, \\
    \eta_t [\QGunknown]_2 + \eta_2 [U]_{22} &=0, \\
    \eta_t [\QGunknown]_3 + \eta_2 [U]_{32} &=0, 
\end{align*}
which implies $[\QGunknown]_1=[U]_{12}=0$ as $\eta_2\neq 0$. Thus Defn.~\ref{defn:suitable-operator}~\ref{item:suitable-operator-b} is satisfied in this case too. 
\end{proof}

\begin{proof}[Proof of Lemma~\ref{lemma:ass-QG}~\ref{item:ass3-QG}] 
Throughout the proof we fix an arbitrary $(Q,q,c) \in \Theta$. As a first step we prove that $\Lambda'$ is complete\footnote{See Defn.~\ref{defn:app-complete-wc} for the definition of ``completeness''.} with respect to $\sK_{Q,q,c}$, where $\Lambda'$ is the larger cone given in \eqref{eq:QG-actual-WC} (the actual ``wave cone'' in the sense of Defn.~\ref{defn:wave-cone}). So let $(\QGunknown_1,U_1),(\QGunknown_2,U_2)\in \sK_{Q,q,c}$. A simple computation shows 
\begin{align}
    U_2 - U_1 &= \QGunknown_2 \cdot \big(((\QGunknown_2)_h)^\perp\big)^\trans - \QGunknown_1 \cdot \big(((\QGunknown_1)_h)^\perp\big)^\trans \notag \\
    &= (\QGunknown_2 - \QGunknown_1) \cdot \big(((\QGunknown_2)_h)^\perp\big)^\trans + \QGunknown_1 \cdot \big(((\QGunknown_2)_h)^\perp - ((\QGunknown_1)_h)^\perp\big)^\trans \label{eq:QG-diff-U}
\end{align}
Setting 
$$
    \eta_x := \left\{\begin{array}{rl} \QGunknown_2-\QGunknown_1 & \text{ if }\QGunknown_2-\QGunknown_1\neq 0, \\ \text{arbitrary, }\neq 0 & \text{ otherwise}, \end{array} \right. \qquad\text{ and }\qquad \eta_t := - ((\QGunknown_2)_h)^\perp \cdot (\eta_x)_h, 
$$
we see that 
$$ 
    (\QGunknown_2 - \QGunknown_1) \times \eta_x = 0,
$$
and, together with \eqref{eq:QG-diff-U}, we obtain 
\begin{align*}
    &(\QGunknown_2-\QGunknown_1)\eta_t + (U_2-U_1) \cdot (\eta_x)_h \\
    &= (\QGunknown_2-\QGunknown_1) \Big( - ((\QGunknown_2)_h)^\perp \cdot (\eta_x)_h + ((\QGunknown_2)_h)^\perp \cdot (\eta_x)_h \Big) + \QGunknown_1 \Big( \big((\QGunknown_2-\QGunknown_1)_h\big)^\perp  \cdot (\eta_x)_h \Big) =0.
\end{align*}
Thus we have shown $(\QGunknown_2,U_2) - (\QGunknown_1,U_1) \in \Lambda'$ as desired. 

Next we prove that in any neighbourhood of a point $(\QGunknown,U)\in \sK_{Q,q,c}$ there is another point $(\widetilde{\QGunknown},\widetilde{U})\in \sK_{Q,q,c}$. So let $(\QGunknown,U)\in \sK_{Q,q,c}$. We immediately get from \eqref{eq:const-QG} that $[\QGunknown]_1^2 + [\QGunknown]_2^2 = 2q$. Since $q>0$ by assumption, we know that $[\QGunknown]_1^2>0$ or $[\QGunknown]_2^2>0$. Let us consider the former case. Then for all $\delta>0$ sufficiently small, 
\begin{align*}
    [\widetilde{\QGunknown}]_1 &:= \sign([\QGunknown]_1) \sqrt{[\QGunknown]_1^2 - \delta} , \\
    [\widetilde{\QGunknown}]_2 &:= \left\{ \begin{array}{rl} \sign([\QGunknown]_2) \sqrt{[\QGunknown]_2^2 + \delta} & \text{ if }[\QGunknown]_2 \neq 0, \\ \sqrt{\delta} & \text{ if }[\QGunknown]_2=0 , \end{array} \right. \\
    [\widetilde{\QGunknown}]_3 &:= [\QGunknown]_3 , \\
    \widetilde{U} &:= \widetilde{\QGunknown}\otimes (\widetilde{\QGunknown}_h)^\perp - q \mathbb{E}^{3\times 2}
\end{align*}
are well-defined, $(\widetilde{\QGunknown},\widetilde{U})\in \sK_{Q,q,c}$ and $(\widetilde{\QGunknown},\widetilde{U})$ arbitrary close to $(\QGunknown,U)$. The case $[\QGunknown]_2^2>0$ works in a similar fashion.

Now we show\footnote{The idea behind the following goes back to \cite[Proof of Lemma~6]{DelSze10} and was also used in \cite[Lemma~5.3.5]{Markfelder}.} that $\interior{\big((\sK_{Q,q,c})^\co\big)} \subset (\sK_{Q,q,c})^\Lambda$. Due to Prop.~\ref{prop:app-Lconvex}, this immediately leads to the desired property $\interior{\big((\sK_{Q,q,c})^\Lambda\big)} = \interior{\big((\sK_{Q,q,c})^\co\big)}$. So let $(\QGunknown,U)\in \interior{\big((\sK_{Q,q,c})^\co\big)}$. Then there exist $N\in \N$, $(\QGunknown_i,U_i)\in\sK_{Q,q,c}$ for $i=1,...,N$ such that 
$$
    (\QGunknown,U) \in \interior{\Big(\big\{(\QGunknown_1,U_1),...,(\QGunknown_N,U_N)\big\}^\co\Big)}, 
$$
see \cite[Prop.~A.5.4]{Markfelder} for a detailed proof of this fact. With the observation above, we may assume that the $(\QGunknown_1)_h,...,(\QGunknown_N)_h$ are pairwise different (if not, one can perturb the $(\QGunknown_i,U_i)$ as explained above). Since $\Lambda'$ is complete, we infer that $(\QGunknown_{i_1},U_{i_1})-(\QGunknown_{i_2},U_{i_2})\in \Lambda'$ for all $i_1,i_2\in \{1,...,N\}$, where $(\QGunknown_{i_1}-\QGunknown_{i_2})_h \neq 0$. The corresponding $\eta$ satisfies $\eta_x\neq 0$. So if $(\eta_1,\eta_2)=0$, then $\eta_3\neq 0$. This however implies $(\QGunknown_{i_1}-\QGunknown_{i_2})_h=0$, a contradiction. Thus we have shown that even $(\QGunknown_{i_1},U_{i_1})-(\QGunknown_{i_2},U_{i_2})\in \Lambda$ for all $i_1,i_2\in \{1,...,N\}$. Then \cite[Prop.~4.2.15]{Markfelder} yields that $(\QGunknown,U)\in (\sK_{Q,q,c})^\Lambda$ which finishes the proof.
\end{proof}

\begin{proof}[Proof of Lemma~\ref{lemma:ass-QG}~\ref{item:ass4-QG}] 
Set $D((Q,q,c),(\QGunknown,U):= 2q + c^2 - |\QGunknown|^2$. Together with the transformation of the constitutive sets $\sK_{Q,q,c}$ to the corresponding sets $\sK_{p,q,c}^\hydrost$ for the case of hydrostatic Euler which is exhibited in the proof of Lemma~\ref{lemma:ass-QG}~\ref{item:ass1-QG}, the desired claim follows from Lemma~\ref{lemma:ass-hydrost-euler}~\ref{item:ass4-hydrost-euler}.
\end{proof}

\subsection{Proof of Thm.~\ref{thm:QG-wilddata}} \label{subsec:ex-QG-pf} 

Before we prove Thm.~\ref{thm:QG-wilddata}, we show that $(0,0)\in \sU_{Q,q,c}$.

\begin{prop} \label{prop:zero-in-U-QG}
    It holds that 
    $$
        (0,0)\in \sU_{Q,q,c} \qquad \text{ for all }Q\in \R,\ q>0,\ c>0.
    $$
\end{prop} 

\begin{proof}
Prop.~\ref{prop:zero-in-U-QG} follows directly from the transformation of the constitutive sets $\sK_{Q,q,c}$ to the corresponding sets $\sK_{p,q,c}^\hydrost$ for the case of hydrostatic Euler, and the fact that $(0,0,0,0)\in \sU_{p,q,c}^\hydrost$, see Rem.~\ref{rem:zero-in-U-incomp}.
\end{proof}

Now we are ready to prove Thm.~\ref{thm:QG-wilddata}, which follows the same lines as the corresponding theorems for the other cases, e.g.~Thm.~\ref{thm:hydrost-euler-wilddata}.

\begin{proof}[Proof of Thm.~\ref{thm:QG-wilddata}]
Choose $(\ov{\QGunknown},\ov{U})=(0,0)$, and $Q=0$, $q>0$ and $c>0$ constant. Then $(\ov{\QGunknown},\ov{U})=(0,0)$ obviously satisfies the linear PDEs \eqref{eq:lin-QG-curl}, \eqref{eq:lin-QG-mom}, and takes values in $\sU_{Q,q,c}$ according to Prop.~\ref{prop:zero-in-U-QG}. So we can proceed as in the proofs of Thms.~\ref{thm:incomp-euler-wilddata} and \ref{thm:hydrost-euler-wilddata}. For the energy balance, we compute
\begin{align*}
    &\partial_t \left( \half |\QGunknown|^2\right) + \Divh \left( \half |\QGunknown|^2 (\QGunknown_h)^\perp \right) - \Div(Q\times \QGunknown ) \\
    &= \partial_t \left( q + \half c^2\right) + \Divh \left(\left( q + \half c^2 \right) (\QGunknown_h)^\perp \right) = 0,
\end{align*}
which has to be understood in the sense of distributions. Here we have used that $q,c$ are constant, and \eqref{eq:QG-curl}.
\end{proof}

\section*{Acknowledgements} 
D.W.B. acknowledges support from the Cambridge Trust and the Cantab Capital Institute for Mathematics of Information. D.W.B. would also like to thank the Isaac Newton Institute for Mathematical Sciences, Cambridge, for support and hospitality during the programme ``Anti-diffusive dynamics : from sub-cellular to astrophysical scales'' where work on this paper was undertaken. This work was supported by EPSRC grant no EP/R014604/1. S.M. acknowledges financial support from the Alexander von Humboldt foundation and also from the Deutsche Forschungsgemeinschaft (DFG, German Research Foundation) SPP 2410 \emph{Hyperbolic Balance Laws in Fluid Mechanics: Complexity, Scales, Randomness (CoScaRa)}, within the Project 525935467 \emph{Convex integration: towards a mathematical understanding of turbulence, Onsager conjectures and admissibility criteria}. D.W.B. and E.S.T. have benefitted from the inspiring environment of the CRC 1114 “Scaling Cascades in Complex Systems”, Project Number 235221301, Project C09, funded by the Deutsche Forschungsgemeinschaft (DFG). Moreover, this work was also supported in part by the DFG Research Unit FOR 5528 on Geophysical Flows.

\appendix 

\section{Some linear algebra} \label{app:linear-algebra}
We recall the following lemma, which is a version of \cite[Lemma~A.3]{Markfelder24}.

\renewcommand{\arraystretch}{1.3}
\begin{lemma}[See {\cite[Lemma~A.3]{Markfelder24}}] \label{lemma:app-eigenvalues} 
	Let $p,q\in \R$ and $(u,U)\in \R^2 \times \symz{2}$. Then the following assertions are equivalent: 
	\begin{align}
        &\bullet\ \lambda_{\max}(u\otimes u - U -(q-p)\id_2)\leq 0, \label{eq:app-eigenvalues1} \\
        &\bullet\ \frac{[u]_1^2 + [u]_2^2}{2} - (q-p) + \sqrt{\left(\frac{[u]_1^2 - [u]_2^2}{2} - [U]_{11} \right)^2 + \Big( [u]_1 [u]_2 - [U]_{12} \Big)^2 } \leq 0 , \label{eq:app-eigenvalues2} \\
        &\bullet\ \left\{ \begin{array}{r} [u]_1^2 + [u]_2^2 - 2(q-p) \leq 0, \\ \Big( [u]_1^2 - [U]_{11} - (q-p)\Big) \Big( [u]_2^2 + [U]_{11} - (q-p) \Big) - \Big( [u]_1 [u]_2 - [U]_{12} \Big)^2 \geq 0 . \end{array} \right. \label{eq:app-eigenvalues3}
    \end{align}
	The above claim is also true if one replaces all non-strict inequalities by strict inequalities in \eqref{eq:app-eigenvalues1}-\eqref{eq:app-eigenvalues3} (i.e.~one replaces all ``$\leq$'' and ``$\geq$'' by ``$<$'' and ``$>$'' respectively).
\end{lemma} 
\renewcommand{\arraystretch}{1}

\section{Convex hulls in general} \label{app:convex}
In this section we recall the definition of the convex hull. Most of the content can be also found in \cite[Sect.~A.5.1]{Markfelder}. 

In this section we suppose $M \in \N$.

\begin{defn}[See {\cite[Defn.~A.5.1]{Markfelder}}] \label{defn:app-convex}
	\begin{itemize}
		\item The \emph{closed line segment} $[p,q]\subset \R^M$ between two points $p,q\in \R^M$ is defined by 
		$$
			[p,q]=\left\{\tau p + (1-\tau) q \,\Big|\, \tau\in [0,1]\right\}.
		$$
		
		\item A set $S\subset \R^M$ is called \emph{convex} if $[p,q]\subset S$ for all $p,q\in S$.
		
		\item The \emph{convex hull $\sK^\co$} of a set $\sK\subset \R^M$ is the smallest convex set which contains $\sK$.
	\end{itemize}
\end{defn}

The following fact is well-known.

\begin{prop}[See e.g.~{\cite[Prop.~A.5.2]{Markfelder}}] \label{prop:app-caratheodory}
	Let $\sK\subset \R^M$. Then 
	\begin{align*}
		\sK^\co = \bigg\{ p\in \R^M\, \Big|\, &\exists N\in \N, \exists (\tau_i,p_i)\in \R^+ \times \R^M \text{ for all }i=1,...,N \text{ such that } \\
		& \bullet \ \sum_{i=1}^N \tau_i = 1,  \\
		& \bullet \ p_i \in \sK \text{ for all }i=1,...,N, \text{ and } \\
		& \bullet \ p= \sum_{i=1}^N \tau_i p_i \bigg\} .
	\end{align*}
\end{prop}

\section{Tetrahedron in $\R^3$} \label{app:tetrahedron}

\begin{lemma} \label{lemma:app-tetrahedron} 
    Let $a_1< a_2$, $b_1< b_2$ and set
    $$
        \sK := \Big\{(a_1,b_1,a_1b_1), (a_1,b_2,a_1b_2), (a_2,b_1,a_2b_1), (a_2,b_2,a_2b_2)\Big\} \subset \R^3.
    $$
    Then 
    \begin{align*}
		\sK^\co = \bigg\{ (x,y,z)\in \R^3 \,\Big|\, &\bullet\ a_1\leq x \leq a_2,\\
        &\bullet\ b_1\leq y \leq b_2, \\
        &\bullet\ z \leq xy + \min\big\{ (x-a_1)(b_2-y) , (a_2-x) (y-b_1) \big\}, \\
        &\bullet\ z \geq xy - \min\big\{ (x-a_1)(y-b_1) , (a_2-x) (b_2-y) \big\} \bigg\} . 
	\end{align*} 
\end{lemma}

The content of Lemma~\ref{lemma:app-tetrahedron} is elementary and easy to prove which is why we skip its proof.

\section{$\Lambda$-convex hulls and the $H_N$-condition} \label{app:Lconvex}

In this section we summarize the definitions and some properties of the $\Lambda$ convex hull and the $H_N$-condition. Most of the content can be found in \cite[Sect.~4.2]{Markfelder} and references therein.

In this section we suppose $M>1$, $\sK\subset \R^M$ a set and $\Lambda\subset \R^M$ a cone.

\subsection{Definition and properties}

Let us first recall the definition of the $\Lambda$-convex hull.

\begin{defn}[See {\cite[Defn.~4.2.2]{Markfelder}}] \label{defn:app-Lconvex}
    \begin{itemize}
        \item A set $S\subset \R^M$ is called \emph{$\Lambda$-convex} if $[p,q]\subset S$ for all $p,q\in S$ with $p-q\in \Lambda$.

        \item The \emph{$\Lambda$-convex hull $\sK^\Lambda$} of $\sK$ is the smallest $\Lambda$-convex set which contains $\sK$.
    \end{itemize}
\end{defn}

\begin{prop}[See {\cite[Prop.~4.2.3]{Markfelder}}] \label{prop:app-Lconvex}
    \begin{itemize}
        \item Every convex set is $\Lambda$-convex.
        \item $\sK^\Lambda\subset \sK^\co$.
    \end{itemize}
\end{prop}
        
Next we recall the definition of the $H_N$-condition and the barycenter, see \cite[Sect.~4.2]{Markfelder} and references therein. 

\begin{defn}[See {\cite[Defn.~4.2.4]{Markfelder}}] \label{defn:app-hn} 
	Let $N\in \N$ and $(\tau_i,p_i)\in \R^+ \times \R^M$ for $i=1,...,N$. We say that the family of pairs $\big\{(\tau_i,p_i)\big\}_{i=1,...,N}$ satisfies the \emph{$H_N$-condition} if the following holds.
	\begin{itemize}
		\item If $N=1$, then $\tau_1=1$.
		\item If $N\geq 2$, then (after relabeling if necessary) $p_2-p_1 \in\Lambda$ and the family 
		\begin{equation} \label{eq:defn-hn-iteration}
			\left\{ \left( \tau_1 + \tau_2 , \frac{\tau_1}{\tau_1 + \tau_2} p_1 + \frac{\tau_2}{\tau_1 + \tau_2} p_2\right) \right\} \cup \big\{(\tau_i,p_i)\big\}_{i=3,...,N}
		\end{equation}
		satisfies the $H_{N-1}$-condition.
	\end{itemize}
\end{defn} 

\begin{defn}[See {\cite[Defn.~4.2.5]{Markfelder}}] \label{defn:app-barycenter}
	Let $N\in \N$ and $(\tau_i,p_i)\in \R^+ \times \R^M$ for $i=1,...,N$ with $\sum_{i=1}^N \tau_i =1 $. Then 
	$$
		p := \sum_{i=1}^N \tau_i p_i
	$$
	is the \emph{barycenter} of the family of pairs $\big\{(\tau_i,p_i)\big\}_{i=1,...,N}$.
\end{defn}

The following proposition covers one of the most important properties of the $\Lambda$-convex hull.

\begin{prop}[See {\cite[Prop.~4.2.9]{Markfelder}}] \label{prop:app-laminates}
    It holds that 
    \begin{align*}
        \sK^\Lambda = \bigg\{ p\in \R^M\, \Big|\, &\exists N\in \N, \exists (\tau_i,p_i)\in \R^+ \times \R^M \text{ for all }i=1,...,N \text{ such that } \\
        & \bullet \ \text{the family } \big\{ (\tau_i,p_i)\big\}_{i=1,...,N} \text{ satisfies the $H_N$-condition, } \\
        & \bullet \ p_i \in \sK \text{ for all }i=1,...,N \text{ and } \\
        & \bullet \ p \text{ is the barycenter of the family }\big\{ (\tau_i,p_i)\big\}_{i=1,...,N}, \text{ i.e. } p= \sum_{i=1}^N \tau_i p_i \bigg\} .
    \end{align*}
\end{prop}

\subsection{Complete wave cones} 

\begin{defn}[See {\cite[Defn.~4.2.12]{Markfelder}}] \label{defn:app-complete-wc}
    The wave cone $\Lambda$ is called \emph{complete with respect to $\sK$} if $p-q \in \Lambda$ for all $p,q\in \sK$.
\end{defn}

In the case where the wave cone is complete, the $\Lambda$-convex hull has the following important property.

\begin{prop}[See {\cite[Cor.~4.2.13]{Markfelder}}] \label{prop:app-complete-wc}
    If $\Lambda$ is complete with respect to $\sK$, then $\sK^\Lambda=\sK^\co$.
\end{prop}

\section{Construction of a subsolution for any initial data in $L^\infty$} \label{app:sfad}

In this section we construct a subsolution for any $L^\infty$ initial data of the incompressible Euler equations \eqref{eq:incomp-euler-mass}, \eqref{eq:incomp-euler-mom}. We build on the approach of \name{Wiedemann}~\cite{Wiedemann11}, who considered $L^2$ initial data. Due to the fact that we are dealing with $L^\infty$ data, and in order to construct subsolutions that fit into our framework elaborated in Sect.~\ref{sec:general}, we have to convolve the initial data with a different kernel compared to \cite{Wiedemann11}. Below we give a detailed proof of the following statement. 

\begin{prop} \label{prop:app-sfad} 
    Let $v_0 \in L^\infty (\T^3;\R^3)$ which satisfies $\Div v_0 = 0$ in the sense of distributions. Then there exist\footnote{Here $\symz{3}$ denotes the space of all symmetric $3\times 3$-matrices with zero trace.}
    \begin{align*}
        v &\in \Cweak([0,\infty);L^2(\T^3;\R^3)) \cap C^1 ((0,\infty)\times \T^3;\R^3) , \\
        U &\in \Cweak([0,\infty);L^2(\T^3;\symz{3})) \cap C^1 ((0,\infty)\times \T^3;\symz{3})
    \end{align*}
    with the following properties:
    \begin{enumerate}
        \item \label{item:sfad-proposition-a} The PDEs 
        \begin{align*}
            \Div v &= 0, \\
            \partial_t v + \Div U &= 0
        \end{align*} 
        hold pointwise for all $(t,x)\in (0,\infty)\times \T^3$. 

        \item \label{item:sfad-proposition-b} The following initial conditions are satisfied: 
        $$
            v(0,\cdot ) = v_0, \qquad U(0,\cdot) = 0.
        $$

        \item \label{item:sfad-proposition-c} There exists a universal\footnote{In particular, $C$ is independent of $v_0$.} constant $C>0$ such that 
        \begin{align*}
            \| v(t,\cdot) \|_{L^\infty(\T^3;\R^3)} &\leq C \| v_0 \|_{L^\infty(\T^3;\R^3)} , \\
            \| U(t,\cdot) \|_{L^\infty(\T^3;\R^3)} &\leq C \| v_0 \|_{L^\infty(\T^3;\R^3)} 
        \end{align*}
        for all $t\in [0,\infty)$.
    \end{enumerate}
\end{prop}

The proof of Prop.~\ref{prop:app-sfad} is quite technical. However we need to convolve with a different kernel compared to \cite{Wiedemann11}.

Let us first define
\begin{align} 
    F(t,r) &:= \frac{- 105 t^2}{4 \pi ^9 r^{10}} \bigg[ \left(7 \pi ^3 r^3 t+75 \pi  r t^3\right) \sin \left(\frac{2 \pi  r}{t}\right) + \left(165 \pi  r t^3-29 \pi ^3 r^3 t\right) \sin \left(\frac{4 \pi  r}{t}\right) \label{eq:app-defn-F} \\
    & \qquad\qquad\qquad + \left(2 \pi ^4 r^4+6 \pi ^2 r^2 t^2-120 t^4\right) \cos \left(\frac{2 \pi  r}{t}\right) \notag \\
    & \qquad\qquad\qquad + \left(4 \pi ^4 r^4-96 \pi ^2 r^2 t^2+120 t^4\right) \cos \left(\frac{4 \pi  r}{t}\right)\bigg] \notag
\end{align}
for $(t,r)\in(0,\infty) \times (0,\infty)$. Next we show some properties of $F$. 

\begin{lemma} \label{lemma:app-sfad-properties-F}
    The function $F$ as defined in \eqref{eq:app-defn-F} has the following properties: 
    \begin{enumerate}
        \item \label{item:sfad-propF-a} It holds that 
        $$
            \lim\limits_{r\to 0} F(t,r) = \frac{14\pi}{3 t^4} \qquad \text{ for all }t>0.
        $$
        Thus we may consider $F\in C^0 ((0,\infty)\times [0,\infty) ; \R)$.

        \item \label{item:sfad-propF-b} We even have $F\in C^1 ((0,\infty)\times [0,\infty) ; \R)$ with $\partial_r F(t,r=0) = 0$. 
        
        \item \label{item:sfad-propF-c} There exist constants $C_\ell>0$ ($\ell=1,...,7$) such that 
        \begin{align}
            |F(r,t)|&\leq \frac{C_1}{t^4} \qquad \text{ for all }t>0 \text{ and } r\geq 0, \label{eq:app-F-bound} \\
            \int_0^\infty |F(t,r)| t r^2 \dr &\leq C_2 \qquad \text{ for all }t>0, \label{eq:app-F-L1bound-tr2} \\
            \int_0^\infty |F(t,r)| r^3 \dr &\leq C_3 \qquad \text{ for all }t>0, \label{eq:app-F-L1bound-r3} \\
            \int_0^\infty |\partial_t F(t,r)| t r^2 \dr &\leq \frac{C_4}{t} \qquad \text{ for all }t>0, \label{eq:app-F_t-L1bound-tr2} \\ 
            \int_0^\infty |\partial_t F(t,r)| r^3 \dr &\leq \frac{C_5}{t} \qquad \text{ for all }t>0, \label{eq:app-F_t-L1bound-r3} \\ 
            \int_0^\infty |\partial_r F(t,r)| t r^2 \dr &\leq \frac{C_6}{t} \qquad \text{ for all }t>0, \label{eq:app-F_r-L1bound-tr2} \\ 
            \int_0^\infty |\partial_r F(t,r)| r^3 \dr &\leq \frac{C_7}{t} \qquad \text{ for all }t>0.\label{eq:app-F_r-L1bound-r3} 
        \end{align}

        \item \label{item:sfad-propF-d} There exist $G\in C^1 ((0,\infty)\times [0,\infty) ; \R)$ and $C_8>0$ such that 
        \begin{align}
            \partial_r G(t,r) &= r F(t,r) \qquad \text{ for all }t>0 \text{ and }r\geq 0, \label{eq:app-Gprime} \\
            \int_0^\infty |G(t,r)| r^2 \dr &\leq t C_8 \qquad \text{ for all }t>0. \label{eq:app-G-L1bound} 
        \end{align}
    \end{enumerate}
\end{lemma}

\begin{proof}
Obviously $F\in C^\infty((0,\infty)\times (0,\infty);\R)$. Plugging the series expansion of $\sin$ and $\cos$ into \eqref{eq:app-defn-F}, we obtain
\begin{equation} \label{eq:app-F-series}
    F(t,r) = \frac{14 \pi }{3 t^4}-\frac{752 \pi ^3}{165 t^4} \left(\frac{r}{t}\right)^2 + \frac{1}{t^4} O\left(\left(\frac{r}{t}\right)^4\right).
\end{equation}
This yields \ref{item:sfad-propF-a}, and $\partial_r F(t,r=0) = 0$ for all $t>0$. It is straightforward to compute the derivatives of \eqref{eq:app-defn-F}, in particular we obtain
\begin{align} 
    \partial_t F(t,r) &= \frac{-105}{4 \pi ^9 r^{10}} \bigg[\left(4 \pi ^5 r^5+33 \pi ^3 r^3 t^2+135 \pi  r t^4\right) \sin \left(\frac{2 \pi  r}{t}\right) \label{eq:app-d_t-F} \\
    &\qquad \qquad \qquad + \left(16 \pi ^5 r^5-471 \pi ^3 r^3 t^2+1305 \pi  r t^4\right) \sin \left(\frac{4 \pi  r}{t}\right) \notag \\
    &\qquad \qquad \qquad + \left(-10 \pi ^4 r^4 t-126 \pi ^2 r^2 t^3-720 t^5\right) \cos \left(\frac{2 \pi  r}{t}\right) \notag \\
    &\qquad \qquad \qquad + \left(124 \pi ^4 r^4 t-1044 \pi ^2 r^2 t^3+720 t^5\right) \cos \left(\frac{4 \pi  r}{t}\right)\bigg] , \notag \\ 
    \partial_r F(t,r) &= \frac{105 t}{4 \pi ^9 r^{11}} \bigg[\left(4 \pi ^5 r^5+61 \pi ^3 r^3 t^2+435 \pi  r t^4\right) \sin \left(\frac{2 \pi  r}{t}\right) \label{eq:app-d_r-F} \\
    &\qquad \qquad \qquad + \left(16 \pi ^5 r^5-587 \pi ^3 r^3 t^2+1965 \pi  r t^4\right) \sin \left(\frac{4 \pi  r}{t}\right) \notag \\
    &\qquad \qquad \qquad + \left(-2 \pi ^4 r^4 t-102 \pi ^2 r^2 t^3-1200 t^5\right) \cos \left(\frac{2 \pi  r}{t}\right) \notag \\
    &\qquad \qquad \qquad + \left(140 \pi ^4 r^4 t-1428 \pi ^2 r^2 t^3+1200 t^5\right) \cos \left(\frac{4 \pi  r}{t}\right)\bigg] \notag 
\end{align}
for all $(t,r)\in (0,\infty)^2$. In addition to that we get from \ref{item:sfad-propF-a} that $\partial_t F(t,r=0) = -\frac{56 \pi }{3 t^5}$. Computing the limits $r\to 0$ in \eqref{eq:app-d_t-F} and \eqref{eq:app-d_r-F} we observe that the partial derivatives $\partial_t F$ and $\partial_r F$ are continuous in $r=0$ and consequently $F\in C^1((0,\infty)\times [0,\infty);\R)$, i.e.~\ref{item:sfad-propF-b}. 

Next we deduce from \eqref{eq:app-F-series} that there exists $\ep>0$ (independent of $t$ and $r$) such that 
\begin{equation} \label{eq:app-F-bound-a}
    |F(t,r)| \leq \frac{1}{t^4} \left( \frac{14\pi}{3} + 1 \right) \qquad \text{ for all }t>0 \text{ and }0\leq \frac{r}{t}\leq \ep.
\end{equation}
Just by estimating $|\sin(y)|\leq 1$, $|\cos(y)|\leq 1$ for all $y\in \R$, we find 
\begin{equation} \label{eq:app-F-bound-Large-Arg} 
    |F(t,r)| \leq \frac{315}{2 \pi ^9 t^4}  \left[\pi ^4 \left(\frac{r}{t}\right)^{-6} + 6 \pi ^3 \left(\frac{r}{t}\right)^{-7} + 17 \pi ^2 \left(\frac{r}{t}\right)^{-8} + 40 \pi \left(\frac{r}{t}\right)^{-9} + 40 \left(\frac{r}{t}\right)^{-10}\right] 
\end{equation}
for all $t>0$, $r>0$. Hence 
\begin{equation} \label{eq:app-F-bound-b} 
    |F(t,r)| \leq \frac{315}{2 \pi ^9 t^4} \left(\pi ^4 \ep^{-6} + 6 \pi ^3 \ep^{-7} + 17 \pi ^2 \ep^{-8} + 40 \pi \ep^{-9} + 40 \ep^{-10}\right)  \quad \text{ for all }t>0 \text{ and }\frac{r}{t}\geq \ep.
\end{equation}
Combining \eqref{eq:app-F-bound-a} and \eqref{eq:app-F-bound-b} we end up with \eqref{eq:app-F-bound}. 

Next we show \eqref{eq:app-F-L1bound-tr2}. To this end we write 
$$
    \int_0^\infty |F(t,r)| t r^2 \dr = \int_0^{\ep t} |F(t,r)| t r^2 \dr + \int_{\ep t}^\infty |F(t,r)| t r^2 \dr,
$$
with the $\ep>0$ we found above. Using \eqref{eq:app-F-bound-a} we obtain 
$$
    \int_0^{\ep t} |F(t,r)| t r^2 \dr \leq \frac{1}{t^4} \left( \frac{14\pi}{3} + 1 \right) t \frac{\ep^3 t^3}{3},
$$
and from \eqref{eq:app-F-bound-Large-Arg} we deduce 
$$ 
     \int_{\ep t}^\infty |F(t,r)| t r^2 \dr \leq \frac{3}{4 \pi ^9} \left(70 \pi ^4 \ep^{-3} + 315 \pi ^3 \ep^{-4} + 714 \pi ^2 \ep^{-5} + 1400 \pi \ep^{-6} + 1200 \ep^{-7}\right) .
$$
Thus \eqref{eq:app-F-L1bound-tr2} holds. The proofs of \eqref{eq:app-F-L1bound-r3}-\eqref{eq:app-F_r-L1bound-r3} work in a very similar way. We leave the details to the reader. 

It remains to show \ref{item:sfad-propF-d}. We set 
$$
    G(t,r) := -\frac{105 t^3}{2 \pi ^9 r^8} \left[\left(\pi ^3 r^3-15 \pi  r t^2\right) \cos \left(\frac{\pi  r}{t}\right)+\left(15 t^3-6 \pi ^2 r^2 t\right) \sin \left(\frac{\pi  r}{t}\right)\right] \sin \left(\frac{3 \pi  r}{t}\right) 
$$
for $t>0$, $r\geq 0$. Then obviously $G\in C^1 ((0,\infty)\times [0,\infty) ; \R)$. It is straightforward to check \eqref{eq:app-Gprime}. The bound \eqref{eq:app-G-L1bound} can be proven in the same way as the bounds in \ref{item:sfad-propF-c}. We leave the details to the reader.
\end{proof}

Now we are ready to prove Prop.~\ref{prop:app-sfad}.

\begin{proof}[Proof of Prop.~\ref{prop:app-sfad}] 
Considering $v_0 \in L^\infty (\T^3;\R^3)$ as a periodic function in $L^\infty(\R^3;\R^3)$, we define 
\begin{align*}
    v(t,x) &:= \int_{\R^3} t F(t,|x-y|) v_0(y) \dy , \\
    U_{ij}(t,x) &:= \int_{\R^3} (x_i-y_i) F(t,|x-y|) [v_0(y)]_j \dy + \int_{\R^3} (x_j-y_j) F(t,|x-y|) [v_0(y)]_i \dy
\end{align*}
for $(t,x)=(t,x_1,x_2,x_3)\in (0,\infty)\times \R^3$, $i,j=1,2,3$. 

Due to item~\ref{item:sfad-propF-b} of Lemma~\ref{lemma:app-sfad-properties-F}, the function $(t,x)\mapsto F(t,|x|)$ lies in $C^1((0,\infty)\times \R^3;\R)$. Moreover we infer from Lemma~\ref{lemma:app-sfad-properties-F}~\ref{item:sfad-propF-c} that 
\begin{align}
    \int_{\R^3} |t F(t,|x|)| \dx &= 4\pi t \int_0^\infty |F(t,r)| r^2 \dr \leq 4\pi C_2 \qquad \text{ for all }t>0, \label{eq:app-tF-L1} \\
    \int_{\R^3} |x_i F(t,|x|)| \dx &\leq 4\pi \int_0^\infty |F(t,r)| r^3 \dr \leq 4\pi C_3 \qquad \text{ for all }t>0, i=1,2,3. \label{eq:app-xF-L1} 
\end{align}
Hence Young's inequality yields that $v$ and $U$ are well-defined, and $v(t,\cdot)\in L^\infty(\R^3;\R^3)$, $U(t,\cdot)\in L^\infty(\R^3;\R^{3\times 3})$ for all $t>0$ where the corresponding $L^\infty$ bounds are uniform in time. Moreover $v$ and $U$ are periodic because $v_0$ is periodic, and hence we may consider the former as functions on $(0,\infty)\times \T^3$. 

Next we note that $U(t,x)$ is obviously a symmetric matrix for all $(t,x)\in (0,\infty)\times \T^3$. Let us now compute the trace of $U(t,x)$. With the help of \eqref{eq:app-Gprime} we find
$$
    \tr (U(t,x)) = 2 \int_{\R^3} F(t,|x-y|) (x-y) \cdot v_0(y) \dy = - 2 \int_{\R^3} \Grad_y G(t,|x-y|) \cdot v_0(y) \dy.
$$
Since $\Cc(\R^3)$ is dense in $W^{1,1}(\R^3)$, we deduce from the fact that $\Div v_0 = 0$ in the sense of distributions, that 
\begin{equation} \label{eq:app-weakly-divfree}
    \int_{\R^3} \Grad_y \phi(y) \cdot v_0(y) \dy = 0 \qquad \text{ for all }\phi \in W^{1,1}(\R^3).
\end{equation}
Observing that $y\mapsto G(t,|x-y|) \in W^{1,1}(\R^3)$ for all $t>0$, $x\in \R^3$ (see Lemma~\ref{lemma:app-sfad-properties-F}~\ref{item:sfad-propF-d} and \eqref{eq:app-xF-L1}), we conclude that $\tr(U(t,x))=0$ for all $(t,x)\in (0,\infty)\times \T^3$.

From the bounds in Lemma~\ref{lemma:app-sfad-properties-F}~\ref{item:sfad-propF-c} we infer 
$$
    v \in C^1 ((0,\infty)\times \T^3;\R^3) \quad \text{ and }\quad U \in C^1 ((0,\infty)\times \T^3;\symz{3}).
$$
Moreover we find
$$
    \Div v = \sum_{j=1}^3 \int_{\R^3} \partial_{x_j} \big[ t F(t,|x-y|) \big] [v_0(y)]_j \dy = -\sum_{j=1}^3 \int_{\R^3} \partial_{y_j} \big[ t F(t,|x-y|) \big] [v_0(y)]_j \dy,
$$
and, since the bounds in Lemma~\ref{lemma:app-sfad-properties-F}~\ref{item:sfad-propF-c} ensure that $y\mapsto t F(t,|x-y|) \in W^{1,1}(\R^3)$ for all $t>0$, $x\in \R^3$, we conclude from \eqref{eq:app-weakly-divfree} that $\Div v = 0$ for all $(t,x)\in (0,\infty)\times \T^3$. In addition to that we compute
\begin{align*}
    \partial_t v + \Div U &= \int_{\R^3} \Big( 4 F(t,|x-y|) + t \partial_t F(t,|x-y|) + |x-y| \partial_r F(t,|x-y|) \Big) v_0(y) \dy \\
    &\qquad + \sum_{j=1}^3 \partial_{x_j} \big[ (x-y) F(t,|x-y|) \big] [v_0(y)]_j \dy \ = \ 0,
\end{align*}
because 
$$
    4 F(t,r) + t \partial_t F(t,r) + r \partial_r F(t,r) = 0
$$
for all $t>0$ and $r\geq 0$, and
$$
    \sum_{j=1}^3 \partial_{x_j} \big[ (x-y) F(t,|x-y|) \big] [v_0(y)]_j \dy = -\sum_{j=1}^3 \partial_{y_j} \big[ (x-y) F(t,|x-y|) \big] [v_0(y)]_j \dy = 0
$$
due to \eqref{eq:app-weakly-divfree} and $y\mapsto (x-y) F(t,|x-y|) \in W^{1,1}(\R^3)$, the latter of which follows from Lemma~\ref{lemma:app-sfad-properties-F}~\ref{item:sfad-propF-c}. Thus item~\ref{item:sfad-proposition-a} of Prop.~\ref{prop:app-sfad} holds. 

Next we set 
$$
    v(0,x):= v_0(x), \qquad \text{ and } \qquad U(0,x):= 0,
$$
so Prop.~\ref{prop:app-sfad}~\ref{item:sfad-proposition-b} is satisfied by construction. Furthermore we obtain Prop.~\ref{prop:app-sfad}~\ref{item:sfad-proposition-c} from Young's inequality and \eqref{eq:app-tF-L1}, \eqref{eq:app-xF-L1}.

Finally, since $v\in \Cweak((0,\infty);L^2(\T^3;\R^3))$ and $U\in \Cweak((0,\infty);L^2(\T^3;\symz{3}))$, it remains to prove that $v$ and $U$ are weakly continuous at $t=0$ as functions from $[0,\infty)$ to $L^2(\T^3;\R^3))$ and $L^2(\T^3;\symz{3})$ respectively. To this end we will work in Fourier space. The Fourier transform of the function $x\mapsto F(t,|x|)$ (with $t>0$) is given by the Hankel transform
\begin{equation} \label{eq:app-hankel-trafo}
    \widehat{F}(t,|k|) = \int_0^\infty F(t,r) \frac{2\sin(2\pi r |k|)}{r|k|}r^2 \dr,
\end{equation}
see \cite[Appendix~B.5]{Grafakos} for details. This way we obtain\footnote{We note that the calculation of $\widehat{F}(t,|k|)$ from $F(t,|x|)$ via \eqref{eq:app-hankel-trafo} is quite involved. It is simpler to compute $F(t,|x|)$ from $\widehat{F}(t,|k|)$ via the inverse Hankel transform which coincides with the Hankel transform itself, i.e.
$$
    F(t,|x|) = \int_0^\infty \widehat{F}(t,s) \frac{2\sin(2\pi s |x|)}{s|x|}s^2 \ds.
$$
This way one can verify that $\widehat{F}(t,|k|)$ indeed reads as given in \eqref{eq:app-Fhat}.} 
\begin{equation} \label{eq:app-Fhat}
    \widehat{F}(t,|k|) = \frac{1}{t} \cdot \left\{ \begin{array}{cl} 1 & \text{ if }|k| \leq \frac{1}{t}, \\
    P(t|k|-1) & \text{ if }\frac{1}{t} \leq |k| \leq \frac{2}{t} , \\
    0 & \text{ else}, \end{array} \right.
\end{equation}
where $P:\R\mapsto \R$, $P(t):= 20t^7 - 70 t^6+ 84 t^5 - 35 t^4 + 1$. Note that $P(0)=1$, $P(1)=0$, $P'(0)=P'(1)=0$ as well as $P''(0)=P''(1)=0$.

Now using convolution theorem we find
$$
    \widehat{v}(t,k) = t \widehat{F}(t,|k|) \widehat{v_0}(k),
$$
and hence\footnote{More precisely we observe $\widehat{v}(t,k) \to \widehat{v_0}(k)$ pointwise in $k$ and conclude strong $L^2$ convergence using Lebesgue's dominated convergence theorem.} $\widehat{v}(t,\cdot) \to \widehat{v_0}$ strongly in $L^2$ as $t\to 0$. Thus $v(t,\cdot) \to v_0$ strongly in $L^2$ as $t\to 0$.

In order to find $\widehat{U}(t,k)$, we have to compute the Fourier transform of $x\mapsto x F(t,|x|)$ (with $t>0$) which is much more involved. We claim that 
\begin{equation} \label{eq:app-xFhat}
    \widehat{xF}(t,k) = -\frac{k}{2\pi \imag |k|^2} \cdot \left\{ \begin{array}{cl} P'(t|k|-1) |k| & \text{ if }\frac{1}{t} \leq |k| \leq \frac{2}{t} , \\
    0 & \text{ else}. \end{array} \right. 
\end{equation}
As above, instead of computing the Fourier transform directly, we compute the Fourier inverse of $\widehat{xF}(t,k)$ by mimicking the Hankel transform. To this end we use spherical coordinates $(s,\Phi,\Theta)$ to describe $k$, where we rotate the coordinate system such that $x\cdot k= |x| s \cos(\Theta)$. This means that 
$$
    k = s \left( \begin{array}{c} \cos (\theta ) \sin (\Theta ) \cos (\Phi ) \cos (\phi )+\sin (\theta ) \cos (\Theta ) \cos (\phi )-\sin (\Theta ) \sin (\Phi ) \sin (\phi ) \\ \cos (\theta ) \sin (\Theta ) \cos (\Phi ) \sin (\phi )+\sin (\theta ) \cos (\Theta ) \sin (\phi )+\sin (\Theta ) \sin (\Phi ) \cos (\phi ) \\ \cos (\theta ) \cos (\Theta )-\sin (\theta ) \sin (\Theta ) \cos (\Phi ) \end{array} \right),
$$
where $(r,\phi,\theta)$ are the spherical coordinates of $x$. Then we can compute the Fourier inverse
\begin{align*}
    &\int_0^\infty \int_0^{2\pi} \int_0^\pi \widehat{xF}(t,k) \exp(2\pi \imag |x| s \cos(\Theta)) s^2 \sin(\Theta) \dTheta \dPhi \ds \\
    &= -\frac{x}{2 \pi^2 r^3}\int_{1/t}^{2/t} P'(t s-1) \big(\sin (2 \pi  r s)-2 \pi  r s \cos (2 \pi  r s)\big) \ds \\
    &= x F(t,|x|),
\end{align*}
which verifies that $\widehat{xF}(t,k)$ as given in \eqref{eq:app-xFhat} is indeed the Fourier inverse of $x F(t,|x|)$.

Now again by convolution theorem we have 
$$
    \widehat{U}_{ij}(t,k) = [\widehat{xF}(t,k)]_i [\widehat{v_0}(k)]_j + [\widehat{xF}(t,k)]_j [\widehat{v_0}(k)]_i.
$$
Next we observe that $U(t,\cdot)\rightharpoonup 0$ weakly in $L^2$ as $t\to 0$, is equivalent to $\widehat{U}(t,\cdot)\rightharpoonup 0$ weakly in $l^2$ as $t\to 0$. Moreover we infer from \eqref{eq:app-xFhat} that there exists a constant $C>0$ such that $|\widehat{xF}(t,k)|\leq C$ for all $t>0$, $k\in \Z^3$. Hence we find for a fixed $g\in l^2(\Z^3;\R)$ 
$$
    \left| \sum_{k\in \Z^3} \widehat{U}_{ij}(t,k) g(k) \right| \leq C \|v_0\|_{L^2(\T^3;\R^3)} \left( \sum_{k\in \Z^3, |k|\geq \frac{1}{t}} g(k)^2 \right)^{1/2} \to 0 \qquad\text{ as }t\to 0,
$$
which finishes the proof. 
\end{proof}

\section{Construction of a suitable function $\chi$ used in the proof of Thm.~\ref{thm:compr-euler-wilddata}} 

The aim of this section is to prove the following statement. 

\begin{lemma} \label{lemma:app-constr-phi}
    For any given $\ov{T},\ov{C}>0$ there exists $\ov{\ep}(\ov{T},\ov{C})>0$ with the following property: for all $0<\ep<\ov{\ep}(\ov{T},\ov{C})$ there exists $\ov{\chi}\in C^\infty([0,\ov{T}])$ such that 
    \begin{align}
        &\ov{\chi}(t)\geq \ov{C} \qquad \text{ for all }t\in [0,\ov{T}], \label{eq:app-constr-phi-p1} \\
        &\ov{\chi}'(t) = \ep (\ov{\chi}(t))^2 + 2 \ov{\chi}(t) +1 \qquad \text{ for all }t\in [0,\ov{T}]. \label{eq:app-constr-phi-p2}
    \end{align}
\end{lemma}

\begin{proof} 
For any $0<\ep<1$ we set 
\begin{equation} \label{eq:app-constr-phi-defn}
    \ov{\chi}(t) := \frac{\ov{C} \sqrt{1-\ep} \cosh (t \sqrt{1-\ep}) + (1+\ov{C}) \sinh (t \sqrt{1-\ep})}{\sqrt{1-\ep} \cosh (t \sqrt{1-\ep}) - (1+\ov{C} \ep ) \sinh (t \sqrt{1-\ep})} .
\end{equation}
Note that $\ov{\chi}\in C^\infty([0,T_\ep])$ as long as $T_\ep$ is smaller than the blow-up time at which the denominator in \eqref{eq:app-constr-phi-defn} becomes zero, i.e. 
\begin{equation} \label{eq:app-constr-phi-blowup}
    T_\ep < \frac{1}{\sqrt{1-\ep}} \Artanh \left(\frac{\sqrt{1-\ep}}{1+\ov{C}\ep}\right).
\end{equation}
Since the right-hand side of \eqref{eq:app-constr-phi-blowup} tends to $\infty$ as $\ep\to 0$, we infer that for $\ep<\ov{\ep}(\ov{T},\ov{C})$ with some suitable $\ov{\ep}(\ov{T},\ov{C})$, it holds that $\ov{\chi}\in C^\infty([0,\ov{T}])$.

It is straightforward to verify that 
$$
    \ov{\chi}'(t) = \frac{(1-\ep) (\ep \ov{C}^2 + 2 \ov{C} + 1)}{\big(\sqrt{1-\ep} \cosh (t \sqrt{1-\ep}) - (1+\ov{C} \ep ) \sinh (t \sqrt{1-\ep})\big)^2}. 
$$
Thus $\ov{\chi}$ is increasing on $[0,\ov{T}]$. Since $\ov{\chi}(0)=\ov{C}$, this immediately shows \eqref{eq:app-constr-phi-p1}.

Moreover another simple computation proves \eqref{eq:app-constr-phi-p2} which finishes the proof.
\end{proof}

\AtNextBibliography{\footnotesize}
\printbibliography[heading=bibintoc]
	
\end{document}